%% file: proxtest.tex
\pdfoutput=1
\documentclass[a4paper,english,11pt]{texbase/shinyart} 
\usepackage[utf8]{inputenc}
\usepackage[T1]{fontenc}
\usepackage{csquotes}
\usepackage{babel}
\usepackage{tikz}
\usepackage{float}
\usepackage{algorithmic}
\usepackage{enumitem}
\usepackage{comment}
\usepackage[color]{changebar}\nochangebars
\usepackage{texbase/eqrefpage}
\usepackage{texbase/nicecleveref}
\cbcolor{blue}

\input{texbase/texbase}

\input{texbase/symbols-common}

\input{texbase/symbols-ten}

\input{texbase/symbols-ip}

\input{cpaccel-common}

\title{Testing and non-linear preconditioning of the proximal point method}

\makeatletter
\AtBeginDocument{
\hypersetup{
    pdftitle = {Testing and non-linear preconditioning of the proximal point method},
    pdfauthor = {T.~Valkonen}
}
}
\makeatother

\begin{document}

\date{2017-03-16 (revised 2018-08-23)}
\author{
    Tuomo Valkonen\thanks{ModeMat, Escuela Politécnica Nacional, Quito, Ecuador; \emph{previously} Department of Mathematical Sciences, University of Liverpool, United Kingdom. \email{tuomo.valkonen@iki.fi}}
    }

\maketitle

\begin{abstract}
    Employing the ideas of non-linear preconditioning and testing of the classical proximal point method, we formalise common arguments in convergence rate and convergence proofs of optimisation methods to the verification of a simple iteration-wise inequality. When applied to fixed point operators, the latter can be seen as a generalisation of firm non-expansivity or the $\alpha$-averaged property.
    The main purpose of this work is to provide the abstract background theory for our companion paper ``Block-proximal methods with spatially adapted acceleration''. In the present account we demonstrate the effectiveness of the general approach on several classical algorithms, as well as their stochastic variants.
    Besides, of course, the proximal point method, these method include the gradient descent, forward--backward splitting, Douglas--Rachford splitting, Newton's method, as well as several methods for saddle-point problems, such as the Alternating Directions Method of Multipliers, and the Chambolle--Pock method.

    \textbf{Get the version from \url{http://tuomov.iki.fi/publications/}, citations broken in this one due arXiv being stuck in the 70s and not supporting biblatex (or 80s bibtex for that matter), hence not modern bibliography styles or utf8.}
\end{abstract}

\newenvironment{demonstration}[1][Proof of convergence]{\begin{proof}[#1]}{\end{proof}}


\section{Introduction}

The proximal point method for monotone operators \cite{martinet1970regularisation,rockafellar1976monotone}, while infrequently used by itself, can be found as a building block of many popular optimisation algorithms. Indeed, many important application problems can be written in the form \cbstart
\begin{equation*}
	\tag{P}
	\label{eq:problem}
	\min_x G(x) + J(x) + F(Kx)
\end{equation*}
for convex $G, J$ and $F$, and a linear operator $K$, with $G$ and $F$ non-smooth and $J$ smooth. \cbend
Examples abound in image processing and data science. 
The problem \eqref{eq:problem} can often be solved by methods such as forward--backward splitting, ADMM (alternating directions method of multipliers) and their variants \cite{beck2009fista,loris2011generalization,gabay,chambolle2010first}. They all involve a proximal point step.

The equivalent saddle point form of \eqref{eq:problem} is
\begin{equation*}
	\tag{S}
	\label{eq:saddle}
	\min_x ~\max_y~ G(x) + \cbstart J(x) \cbend + \iprod{Kx}{y} - F^*(y).
\end{equation*}
In particular within mathematical image processing and computer vision, a popular algorithm for solving \eqref{eq:saddle} with $J=0$ is the primal--dual method of Chambolle and Pock \cite{chambolle2010first}. As discovered in \cite{he2012convergence}, the method can most concisely be written as a \emph{preconditioned proximal point method}, solving on each iteration for $\nextu=(\nextx,\nexty)$ the variational inclusion 
\begin{equation*}
	\tag{PP$_0$}
	\label{eq:pp0}
	0 \in H(\nextu) + \Precond_{i+1}(\nextu-\thisu),
\end{equation*}
where the monotone operator
\begin{equation*}
    H(u) \defeq
        \begin{pmatrix}
            \subdiff G(x) + K^* y \\
            \subdiff F^*(y) -K x
        \end{pmatrix}
\end{equation*}
encodes the optimality condition $0 \in H(\realoptu)$ for \eqref{eq:saddle}.
In the standard proximal point method \cite{rockafellar1976monotone}, one would take $\Precond_{i+1}=I$ the identity. With this choice, \eqref{eq:pp0} is generally difficult to solve.
In the Chambolle--Pock method the \emph{preconditioning operator} is given for suitable step length parameters $\tau_i,\sigma_{i+1},\theta_i>0$ by
\begin{equation}
	\label{eq:precond-cpock-intro}
	\Precond_{i+1} \defeq \begin{pmatrix} \inv\tau_i I  & -K^* \\ -\theta_i K & \inv\sigma_{i+1} I \end{pmatrix}.
\end{equation}
This choice of $\Precond_{i+1}$ decouples the primal $x$ and dual $y$ updates, making the solution of \eqref{eq:pp0} feasible in a wide range of problems.
If $G$ is strongly convex, the step length parameters $\tau_i,\sigma_{i+1},\theta_i$ can be chosen to yield $O(1/N^2)$ convergence rates of an ergodic duality gap and the quadratic distance $\norm{\thisx-\realoptx}^2$.

In our earlier work \cite{tuomov-cpaccel}, we have modified $\Precond_{i+1}$ as well as the condition \eqref{eq:pp0} to still allow a level of mixed-rate acceleration when $G$ is strongly convex only on sub-spaces.
Our convergence proofs were based on \term{testing} the abstract proximal point method by a suitable operator, which encodes the desired and achievable convergence rates on relevant subspaces.

In the present paper, we extend this theoretical approach to non-linear preconditioning, non-invertible step-length operators, and arbitrary monotone operators $H$. Our main purpose is to provide the abstract background theory for our companion paper \cite{tuomov-blockcp}. Here, within these pages, we demonstrate that several classical optimisation methods---including the second-order Newton's method---can also be seen as variants of the proximal point method, and that their common convergence rate and convergence proofs reduce to the verification of a simple iteration-wise inequality. Through application of our theory to Browder's fixed point theorem \cite{browder1965nonexpansive} in \cref{sec:browder}, we see that our inequality generalises the concepts of firm non-expansivity or the $\alpha$-averaged property. Our theory also covers stochastic variants of the considered algorithms.

In \cref{sec:general-proofs}, we start by developing our theory for general monotone operators $H$. This extends, simplifies, and clarifies the more disconnected results from \cite{tuomov-cpaccel} that concentrated on saddle-point problems with preconditioners derived from \eqref{eq:precond-cpock-intro}.
We demonstrate our results on the basic proximal point method, gradient descent, forward--backward splitting, Douglas--Rachford splitting, and Newton's method. The proximal step in forward--backward splitting and proximal Newton's method can be introduced completely ``free'', without any additional proof effort, in our approach.
\cbstart
In \cref{sec:stochastic-examples} we demonstrate the further flexibility of our techniques by application to stochastic block coordinate methods. We refer to \cite{wright2015coordinate} for a review of this class of methods.
In the final \cref{sec:saddle,sec:gap} we specialise our work to saddle-point problems, and demonstrate the results on variants of the Chambolle--Pock method, and the Generalised Iterative Soft Thresholding (GIST) algorithm of \cite{loris2011generalization}.
Some of the derivations in these last two sections are quite abstract and general, as we will need this for our companion paper \cite{tuomov-blockcp} where we develop stochastic primal-dual methods with coordinate-wise adapted step lengths.
\cbend

Besides already cited works, other previous work related to ours includes that on generalised proximal point methods, such as \cite{censor1992proximal,chen1993convergence}, as well inertial methods for variational inclusions \cite{lorenz2014accelerated}.

\section{An abstract preconditioned proximal point iteration}
\label{sec:general-proofs}

\subsection{Notation and general setup}

We use $\convex(X)$ to denote the space of convex, proper, lower semicontinuous functions from $X$ to the extended reals $\extR \defeq [-\infty, \infty]$, and $\linear(X; Y)$ to denote the space of bounded linear operators between Hilbert spaces $X$ and $Y$. 
We denote the identity operator by $I$.
For $T,S \in \linear(X; X)$, we write $T \ge S$ when $T-S$ is positive semidefinite.
Also for possibly non-self-adjoint $T$, we introduce the inner product and norm-like notations
\begin{equation}
    \label{eq:iprod-def}
    \iprod{x}{z}_T \defeq \iprod{Tx}{z},
    \quad
    \text{and}
    \quad
    \norm{x}_T \defeq \sqrt{\iprod{x}{x}_T}.
\end{equation}
For a set $A \subset \R$, we write $A \ge 0$ if every element $t \in A$ satisfies $t \ge 0$.

Our overall wish is to find some $\realoptu \in \Space$, on a Hilbert space $\Space$, solving for a given set-valued map $H: \Space \setto \Space$ the variational inclusion
\begin{equation}
	\label{eq:h-opt}
	0 \in H(\realoptu).
\end{equation}
\cbstart
Throughout the manuscript, $\realoptu$ stands for an arbitrary root of a relevant map $H$.
\cbend
In the present \cref{sec:general-proofs}, $H$ will be arbitrary, but in \cref{sec:saddle,sec:gap}, where we specialise the results, we concentrate on $H$ arising from the saddle point problem \eqref{eq:saddle}.

Our strategy towards finding a solution $\realoptu$ is to introduce an arbitrary non-linear iteration-dependent \term{preconditioner} $\BregFn_{i+1}: \Space \to \Space$ and a \term{step length operator} $\Step_{i+1} \in \L(\Space; \Space)$. With these, we define the generalised proximal point method, which on each iteration $i \in \N$ solves $\nextu$ from 
\begin{equation*}
    \label{eq:pp}
    \tag{PP}
    0 \in \Step_{i+1} H(\nextu) + \BregFn_{i+1}(\nextu).
\end{equation*}
We assume that $\BregFn_{i+1}$ splits into $\Precond_{i+1} \in \linear(\Space; \Space)$, and $\AltBregFn_{i+1}: \Space \to \Space$ as
\begin{equation}
	\label{eq:bregfn-split}
	\BregFn_{i+1}(u) = \AltBregFn_{i+1}(u) + \Precond_{i+1}(u-\thisu).
\end{equation}
More generally, to rigorously extend our approach to cases that would otherwise involve set-valued $\BregFn_{i+1}$, we also consider for $\tilde H_{i+1}: \Space \setto \Space$ the iteration
\begin{equation}
	\label{eq:ppext}
	\tag{PP$^\sim$}
	0 \in \tilde H_{i+1}(\nextu)+\Precond_{i+1}(\nextu-\thisu).
\end{equation}
We say that \eqref{eq:pp} or \eqref{eq:ppext} is \term{solvable} for the iterates $\{\nextu\}_{i \in \N} \subset \Space$ if given any $u^0 \in \Space$, we can solve the corresponding inclusion to iteratively calculate $\nextu$ from $\thisu$ for each $i \in \N$.

\subsection{Basic estimates}
\label{sec:convergence-basic}

We analyse the preconditioned proximal point methods \eqref{eq:pp} and \eqref{eq:ppext} by applying a \term{testing operator} $\Test_{i+1} \in \L(\Space; \Space)$, following the ideas introduced in \cite{tuomov-cpaccel}.
The product $\Test_{i+1}\Precond_{i+1}$ with the linear part of the preconditioner, will, as we soon demonstrate, be an indicator of convergence rates.
In essence, as seen in the descent inequality \eqref{eq:convergence-result-main-h} of the next result, the operator forms a local metric (in the differential geometric sense) that measures closeness to a solution.

\begin{theorem}
    \label{thm:convergence-result-main-h}
    On a Hilbert space $\Space$, let $\tilde H_{i+1}: \Space \setto \Space$, and $\Precond_{i+1}, \Test_{i+1} \in \linear(\Space; \Space)$ for $i \in \N$.
	Suppose \eqref{eq:ppext} is solvable for $\{\nextu\}_{i \in \N} \subset \Space$.
    \cbstart
    If for all $i \in \N$, $\Test_{i+1}\Precond_{i+1}$ is self-adjoint, and for some $\Penalty_{i+1} \in \R$ and $\realoptu \in \Space$ the \term{fundamental condition}
	\begin{equation}
		\label{eq:convergence-fundamental-condition-iter-h}
		\tag{CI$^\sim$}
        \begin{split}
		\iprod{\tilde H_{i+1}(\nextu)}{\nextu-\realoptu}_{\Test_{i+1}}
        &
		\ge 
        \frac{1}{2}\norm{\nextu-\realoptu}_{\Test_{i+2}\Precond_{i+2}-\Test_{i+1}\Precond_{i+1}}^2
        \\ \MoveEqLeft[-1]
        - \frac{1}{2}\norm{\nextu-\thisu}_{\Test_{i+1} \Precond_{i+1}}^2
        - \Penalty_{i+1}(\realoptu),
        \end{split}
	\end{equation}
    holds, then so do the \term{quantitative $\Delta$-Féjer monotonicity}
    \begin{equation}
        \label{eq:quantitative-fejer-pre}
        \tag{QF}
        \frac{1}{2}\norm{\nextu-\realoptu}_{\Test_{i+2}\Precond_{i+2}}^2
        \le
        \frac{1}{2}\norm{\thisu-\realoptu}_{\Test_{i+1}\Precond_{i+1}}^2
        + \Penalty_{i+1}(\realoptu)
        \quad (i \in \N)
    \end{equation}
    \cbend%
    as well as the \term{descent inequality}
    \begin{equation}
        \label{eq:convergence-result-main-h}
        \tag{DI}
        \frac{1}{2}\norm{u^N-\realoptu}^2_{\Test_{N+1}\Precond_{N+1}}
        \le
        \frac{1}{2}\norm{u^0-\realoptu}^2_{\Test_{1}\Precond_{1}}
        +
        \sum_{i=0}^{N-1} \Penalty_{i+1}(\realoptu)
        \quad
        (N \ge 1).
    \end{equation}	
\end{theorem}

\cbstart
The main condition \eqref{eq:convergence-fundamental-condition-iter-h} of \cref{thm:convergence-result-main-h} essentially writes in abstract and step-dependent form the three-point formulas that hold for convex smooth functions (see \cref{sec:three-point}). The term $\frac{1}{2}\norm{\nextu-\realoptu}_{\Test_{i+2}\Precond_{i+2}-\Test_{i+1}\Precond_{i+1}}^2$ is able to measure the strong monotonicity of $H$ or the approximation $\tilde H_{i+1}$. Indeed, if we have the estimate
\[
    \iprod{\tilde H_{i+1}(\nextu)}{\nextu-\realoptu}_{\Test_{i+1}} \ge \frac{1}{2}\norm{\nextu-\realoptu}_{\Test_{i+1}\Gamma}^2,
\]
then this suggests to update the local metrics as
\[
    \Test_{i+2}\Precond_{i+2} \simeq \Test_{i+1}(\Precond_{i+1}+\Gamma),
\]
where we write $\simeq$ to indicate that only the norm induced by the two operators has to be the same: $\Test_{i+1}\Gamma$ might not be self-adjoint, while $\Test_{i+2}\Precond_{i+2}$ has to be self-adjoint. As we will see in \cref{sec:pd-example}, these metric update and self-adjointness conditions effectively give popular primal--dual optimisation methods their necessary forms.
The term $\frac{1}{2}\norm{\nextu-\thisu}_{\Test_{i+1} \Precond_{i+1}}^2$, on the other hand, as we shall see in more detail in \cref{sec:examples-basic}, gives the necessary leeway for taking a forward step instead of a proximal step with respect to some components of $H$. The term $\Delta_{i+1}$ can model function value differences or duality gaps, as will be the case in this work, but in other contexts, such as the stochastic methods of our companion paper \cite{tuomov-blockcp}, it will be a penalty for the dissatisfaction of the metric update; hence the negated sign and the right-hand position in \eqref{eq:convergence-result-main-h}.

Specialised to \eqref{eq:pp}, we obtain the following result. The condition \eqref{eq:convergence-fundamental-condition-iter-h-alt} is often more practical to verify than \eqref{eq:convergence-fundamental-condition-iter-h} thanks to the additional structure introduced by $H(\realoptu) \ni 0$. Indeed, in many of our examples, we can eliminate $H$ through monotonicity. To derive gap and function value estimates in \cref{sec:gap}, we will however need \eqref{eq:convergence-fundamental-condition-iter-h}.
\cbend

\begin{corollary}
    \label{cor:convergence-result-main-h-alt}
    On a Hilbert space $U$, let $H: \Space \setto \Space$. Also let $\Test_{i+1}, \Step_{i+1}, \Precond_{i+1} \in \linear(\Space; \Space)$, and $\AltBregFn_{i+1}: \Space \to \Space$ for $i \in \N$.	
    Suppose \eqref{eq:pp} is solvable for $\{\nextu\}_{i \in \N} \subset \Space$ with $\BregFn_{i+1}$ as in \eqref{eq:bregfn-split}.
    Let $\realoptu \in \inv H(0)$.
    \cbstart%
    If for all $i \in \N$, $\Test_{i+1}\Precond_{i+1}$ is self-adjoint, and  for some $\Penalty_{i+1} \in \R$ and $\realoptu \in \Space$ the condition
    \begin{equation}
        \label{eq:convergence-fundamental-condition-iter-h-alt}
        \tag{CI}
        \begin{split}
        \iprod{\Step_{i+1}[H(\nextu)-H(\realoptu)] + \AltBregFn_{i+1}(\nextu)}{\nextu-\realoptu}_{\Test_{i+1}}
        &
        \ge
        \frac{1}{2}\norm{\nextu-\realoptu}_{\Test_{i+2}\Precond_{i+2}-\Test_{i+1}\Precond_{i+1}}^2
        \\ \MoveEqLeft[-1]
        -\frac{1}{2}\norm{\nextu-\thisu}_{\Test_{i+1} \Precond_{i+1}}^2
        - \Penalty_{i+1}(\realoptu),
        \end{split}
    \end{equation}
    holds, then \eqref{eq:convergence-fundamental-condition-iter-h}, \eqref{eq:quantitative-fejer-pre}, and \eqref{eq:convergence-result-main-h} hold for $\tilde H_{i+1}(u) \defeq \Step_{i+1} H(u)+\AltBregFn_{i+1}(u)$.
    \cbend%
\end{corollary}

\begin{proof}[Proof of \cref{thm:convergence-result-main-h}]
	Inserting \eqref{eq:ppext} into \eqref{eq:convergence-fundamental-condition-iter-h}, we obtain
	\begin{multline}
		\label{eq:convergence-fundamental-condition-iter-h-transformed}	
		\frac{1}{2}\norm{\nextu-\thisu}_{\Test_{i+1} \Precond_{i+1}}^2
		+ \frac{1}{2}\norm{\nextu-\realoptu}_{\Test_{i+1}\Precond_{i+1}-\Test_{i+2}\Precond_{i+2}}^2
		\\
		- \iprod{\nextu-\thisu}{\nextu-\realoptu}_{\Test_{i+1}\Precond_{i+1}}
		\ge - \Penalty_{i+1}(\realoptu).
	\end{multline}
	We recall for general self-adjoint $M$ the three-point formula
	\begin{equation}
	   \label{eq:standard-identity}
	   \iprod{\nextu-\thisu}{\nextu-\realoptu}_{M}
	   = \frac{1}{2}\norm{\nextu-\thisu}_{M}^2
	       - \frac{1}{2}\norm{\thisu-\realoptu}_{M}^2
	       + \frac{1}{2}\norm{\nextu-\realoptu}_{M}^2.
	\end{equation}	
	Using this  with $M=\Test_{i+1}\Precond_{i+1}$, we rewrite \eqref{eq:convergence-fundamental-condition-iter-h-transformed} as the quantitative $\Delta$-Féjer monotonicity \eqref{eq:quantitative-fejer-pre}.
    Summing this over $i=0,\ldots,N-1$, we obtain the descent inequality \eqref{eq:convergence-result-main-h}.
\end{proof}

\begin{remark}[Bregman divergences and Banach spaces]
    \def\dualprod#1#2{\langle#1\,|\,#2\rangle}
    \cbstart%
    Let $X$ be a Banach space and $J \in \convex(X)$. Then for $x \in \Dom J$ and $p \in \subdiff J(x)$ 
    one can define the asymmetric \term{Bregman divergence} (or distance)
    \[
        D_J^{p}(z, x) \defeq J(z)-J(x)-\dualprod{p}{z-x}_X,
        \quad (x \in X),
    \]
    where $\dualprod{\freevar}{\freevar}_X: X^* \times X \to \R$ denotes the dual product.
    This is non-negative, but not a true distance, as it can happen that $D_J^{p}(z, x)=0$ for $z \ne x$. However with $\realoptx,z \in \Dom J$ and $q \in \subdiff J(z)$, we deduce \cite{chen1993convergence}
    \[
        \begin{split}
        D_J^{p}(\realoptx, x)  - D_J^{q}(\realoptx, z) + D_J^q(x,z)
        &
        =
        [J(\realoptx)-J(x)-\dualprod{p}{\realoptx-x}_X]
        -
        [J(\realoptx)-J(z)-\dualprod{q}{\realoptx-z}_X]
        \\
        \MoveEqLeft[-1]
        +
        [J(x)-J(z)-\dualprod{q}{x-z}_X]
        \\
        & =
        \dualprod{p-q}{x-\realoptx}_X.
        \end{split}
    \]
    Therefore, the Bergman distance satisfies an analogue of the standard three-point identity \eqref{eq:standard-identity}. It allows generalising our techniques to Banach spaces and the algorithm
    \[
        0 \in \Test_{i+1}\tilde H_{i+1}(\nextu) + (\nexxt{p} - \this{q})
        \quad\text{with}\quad
        \this{q} \in \subdiff J_{i+1}(\thisu)
        \ \text{and}\ 
        \nexxt{p} \in \subdiff J_{i+1}(\nextu)
    \]
    where for each $i \in \N$ now $\Test_{i+1}\Precond_{i+1}$ has been replaced by $J_{i+1} \in \convex(X)$.
    The convergence will, however, be with respect to $D_{J_{i+1}}$.
    Indeed, if $X$ is, in fact, a Hilbert space and we take $J_{i+1}(x)=\frac{1}{2}\norm{x}_{\Test_{i+1}\Precond_{i+1}}^2$, then $D_{J_{i+1}}^{x-z}(z, x)=\frac{1}{2}\norm{z-x}_{\Test_{i+1}\Precond_{i+1}}^2$.

    Proximal point methods based on general Bregman divergences in place of the squared norm are studied in, e.g., \cite{censor1992proximal,chen1993convergence,hohage2014generalization,hua2016block}.
    \cbend%
    %
\end{remark}

The next two results demonstrate how the estimate of \cref{thm:convergence-result-main-h} can be used to prove convergence with or without rates.

\begin{proposition}[Convergence with a rate]
    \label{prop:rate}
    Suppose the descent inequality \eqref{eq:convergence-result-main-h} holds with $\Penalty_{i+1}(\realoptu) \le 0$, and that $\Test_{N+1} \Precond_{N+1} \ge \mu(N) I$ for all $N \ge 1$. Then $\norm{u^N - \realoptu}^2 \to 0$ at the rate $O(1/\mu(N))$.
\end{proposition}

\begin{proof}
    Immediate from \eqref{eq:convergence-result-main-h}.
\end{proof}

\cbstart

We can also obtain superlinear convergence from \eqref{eq:quantitative-fejer-pre}, a form of quantitative Féjer monotonicity when $\Penalty_{i+1}(\realoptu) \le 0$.

\begin{proposition}[Superlinear convergence]
    \label{prop:superlinear}
    Suppose \eqref{eq:quantitative-fejer-pre} holds with $\Penalty_{i+1}(\realoptu) \le 0$, and that $\Test_{i+1} \Precond_{i+1} = \tauTest_i I$ for some $\tauTest_i$ for all $i \in \N$. If $\tauTest_i/\tauTest_{i+1} \to 0$, then $u^N \to \realoptu$ superlinearly.
\end{proposition}

\begin{proof}
    Immediate from \eqref{eq:quantitative-fejer-pre}.
\end{proof}

The scalar $\tauTest_N$ has its index off-by-one intentionally; the reason will become more apparent once we get to primal--dual methods. It is also possible to obtain superlinear convergences of different orders $q>1$ from \eqref{eq:convergence-result-main-h} or \eqref{eq:quantitative-fejer-pre}. However, the conventional notions $\norm{\nextu-\realoptu}/\norm{\thisu-\realoptu}^q \to c \in \R$ cannot be characterised without involving the iterates.
Indeed, assuming $\tauTest_{i+1} \ge C/\norm{\thisx-\realoptx}^{2q}$, eqref{eq:convergence-result-main-h} characterises superlinear convergence of order $q$.
It would also be possible to introduce new notions of the order of superlinear convergence, not involving the iterates and more in spirit with the testing approach, such as $\tauTest_i^q/\tauTest_{i+1} \to c$, if such a notion would turn out to be useful.

To obtain weak convergence, we do not need $\Test_{i+1} \Precond_{i+1}$ to grow, but we need some additional technical assumptions. First of all, some of the leeway that the fundamental condition \eqref{eq:convergence-fundamental-condition-iter-h} included for the forward steps, is now required to obtain convergence. Secondly, we need some weak-to-strong outer semicontinuity from $H$, which we write more abstractly in terms of $\tilde H_{i+1}$.
It would be possible to improve this requirement based on the Brezis--Crandall--Pazy property \cite{bresiz1970perturbations}.
\cbend

\begin{proposition}[Weak convergence]
    \label{prop:rateless}
    Suppose for all $i \in \N$ that $\Test_{i}\Precond_{i} = \Test_0\Precond_0 \ge 0$ is self-adjoint, and that the iterates of the preconditioned proximal point method \eqref{eq:ppext} satisfy the fundamental condition \eqref{eq:convergence-fundamental-condition-iter-h} with $\Penalty_{i+1}(\realoptu) \le -\frac{\delta}{2}\norm{\nextu-\thisu}_{\Test_{i+1}\Precond_{i+1}}^2$ for all \cbstart  $\realoptu \in \inv H(0)$ and some $\delta > 0$. 
    Suppose either that $\Test_0\Precond_0$ has a bounded inverse, or that $\inv{(\Test_0 H + \Test_0\Precond_0)} \circ \Test_0\Precond_0$ is bounded on bounded sets.
    If $H$ is strong-to-strong outer semicontinuous and
    \begin{equation}
        \label{eq:precond-continuity}
        \nexxt{w} \defeq - \Test_{0}\Precond_{0}(\nextu-\thisu) \to 0,\ 
        w^{i_k} \in  \Test_0 \tilde H_{i_k}(u^{i_k}),\ 
        u^{i_k} \weakto \optu 
        \implies
        0 \in H(\optu),
    \end{equation}
    then $\Test_0\Precond_0(\thisu-\realoptu) \weakto 0$ weakly in $\Space$ for some $\realoptu \in \inv H(0)$.
    \cbend%
\end{proposition}

\cbstart
For the proof, we use the next lemma. Its earliest version is contained in the proof of \cite[Theorem 1]{opial1967weak}, but can be found more explicitly stated as \cite[Lemma 6]{browder1967convergence}. \cbend

\begin{lemma}
    \label{lemma:opial}
    On a Hilbert space $X$, let $\hat X \subset X$ be closed and convex, and $\{\thisx\}_{i \in \N} \subset X$. If the following conditions hold, then $\thisx \weakto x^*$ weakly in $X$ for some $x^* \in \hat X$:
    \begin{enumerate}[label=(\roman*)]
        \item\label{item:opial-non-increasing} $i \mapsto \norm{\thisx-x^*}$ is non-increasing for all $x^* \in \hat X$ \cbstart (Féjer monotonicity) \cbend.
        \item\label{item:opial-limit} All weak limit points of $\{\thisx\}_{i \in \N}$ belong to $\hat X$.
    \end{enumerate}
\end{lemma}


\begin{proof}[Proof of \cref{prop:rateless}]
    \cbstart%
    To use \cref{lemma:opial}, we need a closed and convex solution set. However, $\inv H(0)$ may generally be non-convex and not closed.
    Since $\Test_{i+1}\Precond_{i+1} = \Test_{i+2}\Precond_{i+2}$, using the strong-to-strong outer semicontinuity of $H$, it is easy to see that \eqref{eq:convergence-fundamental-condition-iter-h} holds for all $\realoptu \in \hat U \defeq \closure \conv \inv H(0)$. Consequently the descent inequality \eqref{eq:convergence-result-main-h} holds for all $\realoptu \in \hat U$.

    We apply \cref{thm:convergence-result-main-h} on any $\realoptu \in \hat U$.
    From the quantitative $\Delta$-Féjer monotonicity \eqref{eq:quantitative-fejer-pre}, since $\Penalty_{i+1}(\realoptu) \le -\frac{\delta}{2}\norm{\nextu-\thisu}_{\Test_{i+1}\Precond_{i+1}}^2$ and $\Test_{i+1}\Precond_{i+1} \equiv \Test_0\Precond_0 =: A$, we have
    \begin{equation}
        \label{eq:rateless:fejer}
        \frac{1}{2}\norm{\nextu-\realoptu}_{A}^2
        +
        \frac{\delta}{2}\norm{\nextu-\thisu}_{A}^2
        \le
        \frac{1}{2}\norm{\thisu-\realoptu}_{A}^2
    \end{equation}
    This implies the condition \cref{lemma:opial}\ref{item:opial-non-increasing} for the sequence $\{\thisx \defeq A^{1/2}\thisu\}_{i \in \N}$.

    Let then  $\nexxt w \defeq -A(\nextu-\thisu)$ as in \eqref{eq:precond-continuity}.
    From \eqref{eq:rateless:fejer}, we deduce that $\nexxt w \to 0$ as $i \to \infty$. By \eqref{eq:ppext} and \eqref{eq:precond-continuity}, any weak limit point $u^*$ of the sequence $\{\thisu\}_{i \in \N}$ then satisfies $u^* \in \inv H(0) \subset \hat U$.
    Let then $x^*$ be any weak limit point of $\{\thisx\}_{i \in \N}$.
    We need to show that $x^* \in \hat X \defeq A^{1/2} \hat U$.
    If $\Test_0\Precond_0=A$ has a bounded inverse, then this is clear as the weak convergence of $\{x^{i_k}\}$ implies the weak convergence of $\{u^{i_k}=A^{-1/2}x^{i_k}\}$.
    Otherwise, when $\inv{(\Test_0 H + \Test_0\Precond)_0} \circ \Test_0\Precond_0$ is bounded on bounded sets, since $\nextu \in \inv{(\Test_0 H + \Test_0\Precond_0)}(\Test_0\Precond_0 \thisu)=\inv{(H + A)}(A^{1/2}\thisx)$, we see that $\{\nextu\}_{i \in \N}$ is bounded. Hence a subsequence converges to some $u^* \in \inv H(0)$. But this implies that $x^*=A^{1/2}u^*$ as required.

    By \cref{lemma:opial} now $x^i \weakto x^* \in A^{1/2}\hat U$.
    This implies $\Test_0\Precond_0(\thisu-u^*) \weakto 0$ weakly for some $u^* \in \inv H(0)$.
    \cbend%
\end{proof}


\subsection{Examples of first-order methods}
\label{sec:examples-basic}

We now look at several concrete examples.

\cbstart

\begin{example}[The proximal point method]
	\label{example:prox}
    For all $i \in \N$, take $\Precond_i=I$, $\AltBregFn_i = 0$, and $\Step_{i+1}=\tau_i I$ for some $\tau_i > 0$.
	Then \eqref{eq:pp} is the standard proximal point method $\nextu \in \inv{(I+\tau_i H)}(\thisu)$. If the operator $H: \Space \setto \Space$ is maximal monotone, $\{\thisu\}_{i \in \N}$ converges weakly to some $\realoptu \in \inv H(0)$ for any starting point $u^0 \in \Space$.
\end{example}

\begin{demonstration}
    We take $\Test_{i+1}=\tauTest_i I$ for some $\tauTest_i>0$. Then the fundamental condition \eqref{eq:convergence-fundamental-condition-iter-h-alt} reads
    \begin{equation}
        \label{eq:ci-prox}
        \tauTest_i\tau_i\iprod{H(\nextu)-H(\realoptu)}{\nextu-\realoptu}
        \ge
        \frac{\tauTest_{i+1}-\tauTest_i}{2}\norm{\nextu-\realoptu}^2
        -\frac{\tauTest_i}{2}\norm{\nextu-\thisu}^2 - \Penalty_{i+1}(\realoptu).
    \end{equation}
    As long as $\tauTest_{i} \ge \tauTest_{i+1}$, the monotonicity of $H$ clearly proves \eqref{eq:ci-prox}, thus \eqref{eq:convergence-fundamental-condition-iter-h-alt}, with $\Penalty_{i+1}(\realoptu)=-\frac{\tauTest_i}{2}\norm{\nextu-\thisu}^2$.
    Using the maximal monotonicity, Minty's theorem guarantees the solvability of \eqref{eq:pp}. Thus the conditions of \cref{cor:convergence-result-main-h-alt} are satisfied.
    Maximal monotonicity also guarantees that $H$ is weak-to-strong outer semicontinuous; see \cref{lemma:maxmono-outersemi}. This establishes the iteration outer semicontinuity condition \eqref{eq:precond-continuity}.
    Taking $\tauTest_{i} \equiv \tauTest_0$ for constant $\tauTest_0>0$, so that $\Test_{i+1}\Precond_{i+1}=\Test_0\Precond_0=\tauTest_0 I$, it remains to refer to \cref{prop:rateless}.
\end{demonstration}

Suppose $H$ is strongly monotone, that is, for some $\gamma>0$ holds
\[
    \iprod{H(u)-H(u')}{u-u'} \ge \gamma\norm{u-u'}^2
    \quad (u, u' \in \Space).
\]
Then from \eqref{eq:ci-prox}, we immediately also derive convergence rates as follows.
Letting $\tauTest_i \upto \infty$ will obviously give the fastest convergence, however, the $O(1/N^2)$ step length rule will be useful later on with splitting methods, combining the simple proximal step with other algorithmic elements.

\begin{example}[Acceleration and linear convergence of the proximal point method]
    \label{example:prox-accel}
    Suppose $H$ is strongly monotone for some factor $\gamma>0$.
    If we choose $\tau_{i+1} \defeq \tau_i/\sqrt{1+2\gamma\tau_i}$, then the proximal point method satisfies $\norm{u^N-\realoptu}^2 \to 0$ at the rate $O(1/N^2)$.
    If we keep $\tau_i=\tau_0>0$ constant, we get linear convergence of the iterates.
    If $\tau_i \upto \infty$, we get superlinear convergence.
\end{example}

\begin{demonstration}
    Clearly \eqref{eq:ci-prox} holds with $\Penalty_{i+1}(\realoptu)=0$ provided we update
    \[
        \tauTest_{i+1} \defeq \tauTest_i(1+2\gamma\tau_i).
    \]
    Then \cref{thm:convergence-result-main-h} gives the descent inequality \eqref{eq:convergence-result-main-h}, which now reads
    \begin{equation*}
        \label{eq:di-prox}
        \frac{\tauTest_N}{2}\norm{u^N-\realoptu}^2
        \le
        \frac{\tauTest_0}{2}\norm{u^0-\realoptu}^2
        \quad
        (N \ge 1).
    \end{equation*}
    If we take $\tauTest_i=\tau_i^{-1/2}$, this reads $\tauTest_{i+1} \defeq \tauTest_i+2\gamma\tauTest_i^{-1/2}$. Since $\tauTest_{N}$ is of the order $\Theta(N^2)$ \cite{chambolle2010first,tuomov-cpaccel}, we get the claimed $O(1/N^2)$ convergence from \eqref{eq:di-prox}.
    If, on the other hand, we keep $\tau_i \equiv \tau_0$ fixed, then clearly $\tauTest_N=(1+2\gamma\tau_0)^N\tauTest_0$. Since this is exponential when $\gamma>0$, we get linear convergence from \eqref{eq:di-prox}.
    Finally, if $\tau_i \upto \infty$, we see from \eqref{eq:di-prox} that $\tauTest_i/\tauTest_{i+1} \downto 0$. We now obtain superlinear convergence from \cref{cor:convergence-result-main-h-alt,prop:superlinear}.
\end{demonstration}

The next lemma starts our analysis of gradient descent and forward--backward splitting.
It relies on the three-point smoothness inequalities of \cref{sec:three-point}, which the reader may want to study at this point.
\cbend

\begin{lemma}
    \label{lemma:graddesc}
    Let $H=\subdiff G+\grad J$ for $G,J \in \convex(X)$ such that $\grad J$ is $L$-Lipschitz. For all $i \in \N$, take $\Precond_{i+1} \equiv I$ and $\AltBregFn_{i+1}(u) \defeq \tau_i(\grad J(\thisu) - \grad J(u))$ with $\Step_{i+1}=\tau_i I$ as well as $\Test_{i+1} \equiv \tauTest_i I$ for some $\tau_i,\tauTest_i>0$.\cbstart
    Then the fundamental condition \eqref{eq:convergence-fundamental-condition-iter-h-alt} holds if
    \begin{enumerate}[label=(\roman*)] 
        \item\label{item:graddesc-unaccel} $\phi_i=\phi_0$ is constant, $\tau_i L < 2$, and $\Penalty_{i+1}(\realoptu) \defeq -\tauTest_i(1-\tau_i L/2)\norm{u-\thisu}^2/2$.
        In this case the iteration outer semicontinuity condition \eqref{eq:precond-continuity} moreover holds provided $\inf_i\tau_i >0$.
    \end{enumerate} \cbend
    If $J$ is strongly convex with factor $\gamma>0$, alternatively:
    \begin{enumerate}[label=(\roman*),resume]
        \item\label{item:graddesc-accel} 
            $\tau_0 L^2 < \gamma$, 
            $\tauTest_{i+1} \defeq \tauTest_i(1+\tau_i(2\gamma-\tau_iL^2))$, $\tau_i \defeq \tauTest_i^{-1/2}$ or $\tau_i \defeq \tau_0$, 
            and $\Penalty_{i+1}(\realoptu)=0$.
    \end{enumerate}
\end{lemma}

\begin{proof}
    \cbstart
    We expand the fundamental condition \eqref{eq:convergence-fundamental-condition-iter-h-alt} as
    \begin{equation*}
        \frac{\tauTest_i}{2}\norm{u-\thisu}^2
        + \frac{\tauTest_i-\tauTest_{i+1}}{2}\norm{u-\realoptu}^2
        +\tauTest_i\tau_i\iprod{H(\nextu)-H(\realoptu)}{u-\realoptu}
        \ge 0.
    \end{equation*}
    By the monotonicity of $\subdiff G$, this holds if
    \begin{equation}
        \label{eq:convergence-fundamental-condition-iter-h-graddesc}
        \frac{\tauTest_i}{2}\norm{u-\thisu}^2
        + \frac{\tauTest_i-\tauTest_{i+1}}{2}\norm{u-\realoptu}^2
        +\tauTest_i\tau_i\iprod{\grad J(\thisu)-\grad J(\realoptu)}{u-\realoptu}
        \ge 0.
    \end{equation}
    \ref{item:graddesc-unaccel}
    The three-point inequality \eqref{eq:three-point-hypomonotonicity} in \cref{lemma:smoothness} states
    \[
        \iprod{\grad J(\thisu)-  \grad J(\realoptu)}{u-\realoptu}
        \ge
        -\frac{L}{4}\norm{u-\thisu}^2.
    \]
    This clearly reduces \eqref{eq:convergence-fundamental-condition-iter-h-graddesc} to
    \[
        \frac{\tauTest_i-L\tau_i/2}{2}\norm{u-\thisu}^2
        + \frac{\tauTest_i-\tauTest_{i+1}}{2}\norm{u-\realoptu}^2
        \ge \Delta_{i+1}(\realoptu),
    \]
    which holds under the conditions of \ref{item:graddesc-unaccel}.
    The satisfaction of \eqref{eq:precond-continuity} is immediate from the weak-to-strong outer semicontinuity of $\subdiff F$ (\cref{lemma:maxmono-outersemi}), the Lipschitz continuity of $\grad G$, and the bounds on $\tau_i$.

    \ref{item:graddesc-accel}
    The three-point smoothness inequality \eqref{eq:three-point-hypomonotonicity-sc} in \cref{lemma:sc-smoothness} gives
    \[
        \iprod{\grad J(\thisu)-  \grad J(\realoptu)}{u-\realoptu}
        \ge
        \frac{2\gamma-\tau_i L^2}{2}\norm{u-\realoptu}^2
        -\frac{1}{2\tau_i}\norm{u-\thisu}^2.
    \]\cbend
    Inserting this into \eqref{eq:convergence-fundamental-condition-iter-h-graddesc}, we see it to hold with $\Delta_{i+1}(\realoptu)=0$ if
    \begin{equation}
        \label{eq:graddesc-accel-cond}
        \tauTest_i+\tauTest_i\tau_i(2\gamma-\tau_iL^2) \ge \tauTest_{i+1}.
    \end{equation}
    Clearly our two alternative choices of $\{\tau_i\}_{i \in \N}$ are non-increasing.
    Therefore, \eqref{eq:graddesc-accel-cond} follows from the initialisation condition $\tau_0 L^2 < \gamma$ and the update rule $\tauTest_{i+1} \defeq \tauTest_i+\tauTest_i\tau_i(\gamma-\tau_iL^2)$ in \ref{item:graddesc-accel}.
\end{proof}

\cbstart
\begin{remark}
    It is also possible to exploit the strong convexity of $G$ instead of $J$ for acceleration.
\end{remark}
\cbend

\begin{example}[Gradient descent]
	\label{example:graddesc}
    \cbstart
    Let $H=\grad J$ for $J \in \convex(\Space)$ with $\grad J$ $L$-Lipschitz.
    \cbend
    Taking $\tau_i=\tau$ and $G=0$ constant in \cref{lemma:graddesc}, \eqref{eq:pp} reads
	\begin{equation*}
	    0 = \tau \grad J(\thisu) + \nextu - \thisu.
	\end{equation*}
	This is the gradient descent method.
    Direct application of \cref{lemma:graddesc}\ref{item:graddesc-unaccel} with $u=\nextu$ and $u^*=\realoptu$ together with \cref{cor:convergence-result-main-h-alt,prop:rateless} now verifies the well-known weak convergence of the method to a root $\realoptu$ of $H$ when $\tau L < 2$.

    Observe that $\BregFn_{i+1}=\grad Q_{i+1}$ for
    \[
    	Q_{i+1}(u) \defeq \frac{1}{2}\norm{u-\thisu}^2 + \tau\left[J(\thisu)+\iprod{\grad J(\thisu)}{u-\thisu}-J(u)\right].
    \]
    Each step of \eqref{eq:pp} therefore minimises the \emph{surrogate objective} \cite{daubechies2004surrogate}
    \begin{equation}
    	\label{eq:surrogate}
    	u \mapsto J(u) + \inv\tau Q_{i+1}(u).
    \end{equation}
 	The function $Q_{i+1}$ on one hand penalises long steps, and on the other hand 
    allows longer steps when the local linearisation error is large. 
    In this example, $Q_{i+1}$ is, in fact, a Bregman divergence.
\end{example}

\cbstart
Under strong convexity, we again get rates via \cref{lemma:graddesc}\ref{item:graddesc-accel}.
Minding our remarks before \cref{example:prox-accel}, we only state the case $\tau_i=\tau_0$.
Due to the upper bound $\tau_0 < \gamma/L^2$, we cannot get superlinear convergence as in \cref{example:prox-accel}.

\begin{example}[Acceleration and linear convergence of gradient descent]
    \label{example:graddesc-accel}
    Continuing from \cref{example:graddesc}, if $J$ is strongly convex with factor $\gamma>0$ and $\grad J$ is $L$-Lipschitz,
    and we keep $\tau_i=\tau_0 < \gamma/L^2$ fixed, we get linear convergence.
\end{example}
\cbend

Now comes the full power of \cref{lemma:graddesc}: we can easily bolt on a proximal step to gradient descent.

\begin{example}[Forward--backward splitting]
    \label{example:fb}
    Let $H=\subdiff G+\grad J$ for $G, J \in \convex(X)$ with $\grad J$ Lipschitz. Taking $\Precond_{i+1}$, $\Step_{i+1}$, and $\AltBregFn_{i+1}$ as in \cref{lemma:graddesc}, the preconditioned proximal point method \eqref{eq:pp} becomes
    \[
        0 \in \tau_i \subdiff G(\nextu) + \tau_i \grad J(\thisu) + \nextu-\thisu.
    \]
    This is the forward--backward splitting method
    \[
        \nextu \defeq \inv{(I+\tau_i \subdiff G)}(\thisu-\tau_i \grad J(\thisu)).
    \]
    By \cref{lemma:graddesc}, convergence and acceleration work exactly as for gradient descent in \cref{example:graddesc,example:graddesc-accel}.
\end{example}

\cbstart
We can also do fully non-smooth splitting methods by a lifting approach:
\cbend

\begin{example}[Douglas--Rachford splitting]
	Let $A, B: \Space \setto \Space$ be maximal monotone operators. Consider the problem of finding $\realoptu$ with $0 \in A(\realoptu)+B(\realoptu)$.
	For $\lambda>0$, let
	\begin{align}
        \notag
		H(u, v) & \defeq
		\begin{pmatrix}
			\lambda B(u) + u - v
			\\
			\lambda A(u) + v - u
		\end{pmatrix},
        &
		\Precond_{i+1} &\defeq \begin{pmatrix}0 & 0 \\ 0 & I \end{pmatrix},
        \quad\text{and}\\
        \label{eq:dg-tildeh}
        \tilde H_{i+1}(u, v)
        &
        \defeq
        \begin{pmatrix}
            \lambda B(\nextu) + \nextu - \thisv
            \\
            \lambda A(\nextu+\nextv-\thisv) + \thisv-\nextu
        \end{pmatrix}.
    \end{align}
	Then $0 \in A(\realoptu)+B(\realoptu)$ if and only if $0 \in H(\realoptu, \realoptv)$, where $\realoptv \in (\realoptu-\lambda A(\realoptu))\isect(\realoptu+\lambda B(\realoptu))$.
    The preconditioned proximal point method \eqref{eq:ppext} becomes the Douglas--Rachford splitting \cite{douglas1956numerical}
    \begin{subequations}%
    \label{eq:drs}
	\begin{align}
	    \nextu & \defeq \inv{(I+\lambda B)}(\thisv), \\
	    \nextv & \defeq \thisv + \inv{(I+\lambda A)}(2\nextu-\thisv) - \nextu.
	\end{align}%
    \end{subequations}%
    We work with \eqref{eq:ppext} since in \eqref{eq:pp}, $\AltBregFn_{i+1}$ would have to be set-valued.	
    If $A$ and $B$ are maximal monotone, the variables $\{\thisv\}_{i \in \N}$ converge weakly to $\realoptv$. 
\end{example}

\begin{demonstration}
    Write $\this{\bar u} \defeq (\thisu, \thisv)$ and $\realopt{{\bar u}} \defeq (\realoptu, \realoptv)$. Observe that
	\[
		\nextu-\nextv =: \nexxt{q} \in \lambda A(\nextu-\nextv-\thisv)
		\quad\text{and}\quad
		\realoptu-\realoptv =: \realopt{q} \in \lambda A(\realoptu).
	\]
	Using the monotonicity of $A$ and $B$, with $\Test_{i+1} \defeq I$, we have
	\begin{equation*}
        \begin{split}
		\iprod{\tilde H_{i+1} & (\nexxt{\bar u})}{\Test_{i+1}^*(\nexxt{\bar u}-\realopt{{\bar u}})}
		\subset
		\iprod{\tilde H_{i+1}(\nexxt{\bar u})-H(\realopt{{\bar u}})}{\Test_{i+1}^*(\nexxt{\bar u}-\realopt{{\bar u}})}
		\\
        &
			=
			\lambda\iprod{B(\nextu)-B(\realoptu)}{\nextu-\realoptu}
			+\lambda\iprod{\nexxt{q}-\realopt{q}}{\nextv-\realoptv}
		\\
        & \phantom{ == }
            +\iprod{\nextu-\thisv}{(\nextu-\nextv)-(\realoptu-\realoptv)}
		\\
        & = \lambda\iprod{B(\nextu)-B(\realoptu)}{\nextu-\realoptu} +
			\lambda\iprod{\nexxt{q}-\realopt{q}}{\nextu+\nextv-\thisv-\realoptv}
		\ge 0.
        \end{split}
	\end{equation*}
	Thus the fundamental condition \eqref{eq:convergence-fundamental-condition-iter-h} holds with $\Penalty_{i+1}(\realopt{{\bar u}}) \defeq -\frac{1}{2}\norm{\nexxt{\bar u}-\this{\bar u}}_{\Test_{i+1}\Precond_{i+1}}^2$.  Using \eqref{eq:dg-tildeh} and the weak-to-strong outer semicontinuity of $A$ and $B$ (see \cref{lemma:maxmono-outersemi}), we easily verify \eqref{eq:precond-continuity}.
    \cbstart%
    Since $\Test_{i+1}\Precond_{i+1} \equiv \Test_0\Precond_0$ is non-invertible, we also have to verify that $\inv{(\Test_0H+\Test_0\Precond_0)} \circ \Test_0\Precond_0$ is bounded on bounded sets. This is to say that \eqref{eq:drs} bounds $\nexxt{\bar u}=(\nextu, \thisv)$ in terms of $\thisv$. This is an easy consequence of the Lipschitz-continuity of the resolvent of maximal monotone operators \cite[Corollary 23.10]{bauschke2017convex}.
    \cbend%
    Weak convergence now follows from \cref{thm:convergence-result-main-h,prop:rateless}.
\end{demonstration}

\subsection{Examples of second-order methods}
\label{sec:examples-second}

\cbstart
We now look at how are techniques are applicable to Newton's method.
Through the three-point inequalities of \cref{lemma:c2-smoothness} for $C^2$ functions, the analysis turn out to be very close to that of gradient descent.
Our analysis is not as short as the conventional analysis of Newton's method, but has its advantages. Indeed, the convergence of proximal Newton's method will be an automatic corollary of our approach, exactly how the convergence of forward--backward splitting was a corollary of the convergence of gradient descent.
\cbend

\begin{example}[Newton's method]
    \label{example:newton}
    Suppose $H=\grad J$ for $J \in C^2(\Space)$. 
    Take
    \[
        \BregFn_{i+1}(u) \defeq \grad^2 J(\thisu)(u-\thisu) + \grad J(\thisu) - \grad J(u),
        \quad\text{and}\quad
        \Step_{i+1} \defeq I
    \]
    Then the preconditioned proximal point method \eqref{eq:pp} reads
    \begin{equation*}
        0 = \grad J(\thisu) + \grad^2 J(\thisu)(\nextu - \thisu).
    \end{equation*}
    This is Newton's method.
    \cbstart
    By \cref{lemma:newton} (below) and \cref{prop:rate}, we obtain local linear convergence if $\grad^2 J(\realoptu)>0$.
    By \cref{lemma:newton-superlinear} (below), this convergence is, further, superlinear (quadratic if $\grad^2 J$ is locally Lipschitz near $\realoptx$).
    \cbend

    Observe that now $\BregFn_{i+1}(u)$ is the gradient of
    \[
        Q_{i+1}(u) \defeq J(\thisu)+\iprod{\grad J(\thisu)}{u-\thisu}+\frac{1}{2}\norm{u-\thisu}^2_{\grad^2 J(\thisu)}-J(u).
    \]
    In the surrogate objective \eqref{eq:surrogate}, this allows longer steps when the second-order Taylor expansion under-approximates, and forces shorter steps when it over-approximates.
\end{example}

\cbstart
Again, we can easily bolt on a proximal step:
\cbend

\begin{example}[Proximal Newton's method]
    Let $H=\subdiff G+\grad J$ for $J \in C^2(X)$, and $G \in \convex(X)$. Taking $\Precond_{i+1}$, $\Step_{i+1}$, and $\AltBregFn_{i+1}$ as in \cref{example:newton}, the preconditioned proximal point method \eqref{eq:pp} becomes
    \[
        0 \in \subdiff G(\nextu) + \grad J(\thisu) + \grad^2 J(\thisu)(\nextu-\thisu).
    \]
    This is the proximal Newton's method \cite[see, e.g.,][]{lee2014proximal}
    \[
        \nextu \defeq \inv{(I+\inv{[\grad^2 J(\thisu)]}\subdiff G)}(\thisu-\inv{[\grad^2 J(\thisu)]}\grad J(\thisu)),
    \]
    where $\inv{(I+\inv A \subdiff G)}(v)$ solves
    $
        \min_u \frac{1}{2}\norm{u-v}_A^2 + G(u).
    $
    Convergence and acceleration work exactly as for Newton's method in \cref{example:newton}, based on the same lemmas that we state next.
\end{example}

\begin{lemma}
    \label{lemma:newton}
    Let $H=\subdiff G+\grad J$ for $G \in \convex(\Space)$ and $J \in C^2(\Space)$. Take
    \[
        \BregFn_{i+1}(u) \defeq \grad^2 J(\thisu)(u-\thisu) + \grad J(\thisu) - \grad J(u),
        \quad\text{and}\quad
        \Step_{i+1} \defeq I
    \]
    \cbstart
    For an initial iterate $u^0 \in \Space$, let $\{\nextu\}_{i \in \N}$ be defined through \eqref{eq:pp}.
    \cbend
    If $\grad^2 J(\realoptu)>0$, there exists $\epsilon>0$ such that if $\norm{u^0-\realoptu}_{\grad^2 J(\realoptu)} \le \epsilon$, then the fundamental condition \eqref{eq:convergence-fundamental-condition-iter-h-alt} holds with $\Penalty_{i+1}(\realoptu) = 0$  and $\Precond_{i+1}=\grad^2 J(\thisu)$ for all $i \in \N$. Moreover, we can take $\Test_{i+1}=\tauTest_i I$ such that \cbstart $\Test_{N}\Precond_{N} \ge \kappa^{N} \grad^2 J(\realoptu)$ for some $\kappa>1$.
    In particular, $\norm{\thisu - \realoptu}^2 \to 0$ at the linear rate $O(1/\kappa^N)$.
\end{lemma}

\begin{proof}
    We set $\Precond_{i+1} \defeq \grad^2 J(\thisu)$ and $\Test_{i+1} \defeq \tauTest_i I$ for some $\tauTest_i > 0$.
    Then $\grad^2 J(\realoptu)>0$ imply that $\Test_{i+1}\Precond_{i+1}=\tauTest_i \grad^2 J(\thisu)$ is positive and self-adjoint for $\thisu$ close to $\realoptu$.

    By assumption, for some $\epsilon>0$, we have
    \[
        u^0 \in \realopt\B(\epsilon) \defeq \{ u \in \Space \mid \norm{u-\realoptu}_{\grad^2 J(\realoptu)}\le\epsilon\}.
    \]
    For a fixed $i \in \N$, let us assume that $\thisu \in \realopt\B(\epsilon)$.
    Since $\subdiff F$ is monotone, similarly to the proof of \cref{lemma:graddesc}, the fundamental condition \eqref{eq:convergence-fundamental-condition-iter-h-alt} holds  
    if
    \begin{equation}
        \label{eq:convergence-fundamental-condition-iter-h-newton}
        \tauTest_i D_{i+1}
        \ge
        \frac{1}{2}\norm{\nextu-\realoptu}_{\tauTest_{i+1}\grad^2 J(\nextu)-\tauTest_i\grad^2 J(\thisu)}^2
        - \frac{1}{2}\norm{\nextu-\thisu}_{\tauTest_i\grad^2 J(\thisu)}^2
        -\Penalty_{i+1}(\realoptu),
    \end{equation}
    where we use \eqref{eq:three-point-hypomonotonicity-c2} in \cref{lemma:c2-smoothness} with $\tau=1+\delta_i$ to estimate
    \begin{equation*}
        D_{i+1}
        \defeq
        \iprod{\grad J(\thisu)-\grad J(\realoptu)}{\nextu-\realoptu}
        \ge \frac{(1-\delta_i)^2}{2}\norm{\nextu-\realoptu}^2_{\grad^2 J(\thisu)}-\frac{1}{2}\norm{\nextu-\thisu}^2_{\grad^2 J(\thisu)}
    \end{equation*}
    for
    \begin{equation}
        \label{eq:newton-delta}
        \delta_i \defeq \inf\left\{\delta' \ge 0 \,\middle|\, \begin{array}{r}
            (1-\delta')\grad^2 J(\thisu) \le \grad^2 J(\zeta) \le (1+\delta')\grad^2 J(\thisu)
            \\
            \text{ for all } \zeta \in \realopt\B(\norm{\thisu-\realoptu}_{\grad^2 J(\realoptu)})
            \end{array}
            \right\}.
    \end{equation}
    Consequently, \eqref{eq:convergence-fundamental-condition-iter-h-newton} holds 
    with $\Penalty_{i+1}(\realoptu)=0$
    if we take $\tauTest_{i+1}>0$ such that
    \begin{equation}
        \label{eq:newton:linear:test}
        \tauTest_i(1+(1-\delta_i)^2)\grad^2 J(\thisu)
        \ge
        \tauTest_{i+1}\grad^2 J(\nextu).
    \end{equation}
    This can always be satisfied for some $\tauTest_{i+1}>0$ for $\epsilon>0$ small enough because $\grad^2 J(\realoptu)>0$ then implies $\grad^2 J(\thisu)>0$.

    Now \cref{cor:convergence-result-main-h-alt} shows the quantitative $\Delta$-Féjer monotonicity \cref{eq:quantitative-fejer-pre}, which with \eqref{eq:newton:linear:test} implies
    \begin{equation}
        \label{eq:quantitative-fejer-newton}
        \norm{\nextu-\realoptu}^2_{[1+(1-\delta_i^2)]\grad^2 J(\thisu)}
        \le \norm{\thisu-\realoptu}^2_{\grad^2 J(\thisu)}.
    \end{equation}
    If $\delta_i \in (0, 1)$, this implies by \eqref{eq:newton-delta} that $\norm{\nextu-\realoptu}_{[1+(1-\delta_i^2)]/(1+\delta_i)\grad^2 J(\realoptu)}^2 \le \norm{\thisu-\realoptu}_{\grad^2 J(\realoptu)/(1-\delta_i)}^2$.
    Consequently, if $\delta_i \in (0, 1)$ is small enough, that is, if $\epsilon>0$ is small enough due to the continuity of $\grad^2 J$, we obtain $\norm{\nextu-\realoptu}_{\grad^2 J(\realoptu)} \le \norm{\thisu-\realoptu}_{\grad^2 J(\realoptu)}$ so that also $\nextu \in \realopt\B(\epsilon)$.
    In particular, our assumption $u^0 \in \realopt\B(\epsilon)$ guarantees $\{\thisu\}_{i \in \N} \subset \realopt\B(\epsilon)$.
    Consequently also $\delta_{i+1} \le \delta_i \le \delta_0$ for all $i \in \N$.
    We can now take $\zeta=\nextu$ in \eqref{eq:newton-delta}, so that \eqref{eq:newton:linear:test} gives
    \[
        \tauTest_i(1+(1-\delta_i)^2) \ge (1-\delta_i) \tauTest_{i+1}.
    \]
    Since $\kappa(\delta) \defeq (1+(1-\delta)^2)/(1-\delta)$ is increasing within $(0, 1)$, and $\kappa \defeq \kappa(0)=2$, we see that $\tauTest_{i+1} \ge \kappa \tauTest_i$.
    Taking $\tauTest_0 \defeq 1+\delta_0$ we now get $\Test_{N}\Precond_{N} \ge \kappa^N(1+\delta_0) \grad^2 J(u^N) \ge \kappa^N \grad^2 J(\realoptu)$.
    This implies the convergence rate claim.
\end{proof}

We can also show superlinear convergence, however, this is somewhat more elaborate as we need to make use of $\Delta_{i+1}(\realoptu)$.

\begin{lemma}
    \label{lemma:newton-superlinear}
    With everything as in \cref{lemma:newton}, the convergence rate claim can be improved to superlinear.
    If $\grad^2 J$ is locally Lipschitz near $\realoptu$, for example, if $J \in C^3(\Space)$, then this convergence is quadratic (superlinear convergence of order $q=2$).
\end{lemma}

\begin{proof}
    We continue with the initial setup of the proof of \cref{lemma:newton} until \eqref{eq:convergence-fundamental-condition-iter-h-newton}.
    Now, for $\delta_i$ given by \eqref{eq:newton-delta}, \eqref{eq:three-point-hypomonotonicity-c2x} in \cref{lemma:c2x-smoothness} gives
    \[
        D_{i+1}
        \ge
        \frac{1-\delta_i}{2}\norm{\nextu-\realoptu}^2_{\grad^2 J(\thisu)}
        +
        \frac{1-\delta_i}{2}\norm{\thisu-\realoptu}^2_{\grad^2 J(\thisu)}
        -
        \frac{1}{2}\norm{\nextu-\thisu}^2_{\grad^2 J(\thisu)}.
    \]
    With this, \eqref{eq:convergence-fundamental-condition-iter-h-newton}, hence the fundamental condition \eqref{eq:convergence-fundamental-condition-iter-h-alt}, holds if
    \begin{equation*}
        \Penalty_{i+1}(\realoptu)
        \ge
        \frac{1}{2}\norm{\nextu-\realoptu}_{\tauTest_{i+1}\grad^2 J(\nextu)-(2-\delta_i)\tauTest_i\grad^2 J(\thisu)}^2
        -\frac{1}{2}\norm{\thisu-\realoptu}^2_{\tauTest_i(1-\delta_i)\grad^2 J(\thisu)}.
    \end{equation*}
    This holds for
    \begin{equation}
        \label{eq:newton-superlinear-penalty}
        \Penalty_{i+1}(\realoptu) \defeq
        \frac{1}{2}\norm{\nextu-\realoptu}^2_{\tauTest_{i+1}(1-\delta_{i+1})\grad^2 J(\nextu)}
        -\frac{1}{2}\norm{\thisu-\realoptu}^2_{\tauTest_i(1-\delta_i)\grad^2 J(\thisu)} 
    \end{equation}
    provided
    \begin{equation}
        \label{eq:newton-superlinear-cond0}
        \tauTest_i(2-\delta_i)\grad^2 J(\thisu) \ge \tauTest_{i+1} \delta_{i+1}\grad^2 J(\nextu).
    \end{equation}
    This can always be satisfied for some $\tauTest_{i+1}>0$ if $\epsilon>0$ is small enough because then $\grad^2 J(\thisu)>0$ due to $\grad^2 J(\realoptu)>0$.

    By \cref{cor:convergence-result-main-h-alt} we now obtain the quantitative $\Delta$-Féjer monotonicity \eqref{eq:quantitative-fejer-pre}, which with \eqref{eq:newton-superlinear-penalty} gives
    \begin{equation}
        \label{eq:newton-superlinear-fejer-pre}
        \norm{\nextu-\realoptu}_{\tauTest_{i+1}\delta_{i+1}\grad^2J(\nextu)}^2
        \le
        \norm{\thisu-\realoptu}_{\tauTest_i\delta_i\grad^2J(\thisu)}^2.
    \end{equation}
    Due to \eqref{eq:newton-delta}, we have $(1-\delta_i)\grad^2 J(\realoptu) \le \grad^2 J(\thisu) \le (1+\delta_i)\grad^2 J(\realoptu)$.
    Hence, also using \eqref{eq:newton-superlinear-cond0}, \eqref{eq:newton-superlinear-fejer-pre} implies
    \begin{equation}
        \label{eq:newton-superlinear-fejer}
        \norm{\nextu-\realoptu}_{(2-\delta_i)(1-\delta_i)\grad^2 J(\realoptu)}^2
        \le
        \norm{\thisu-\realoptu}_{\delta_i(1+\delta_i)\grad^2J(\realoptu)}^2.
    \end{equation}
    If $\delta_i \in (0, 1/2]$, this and $\thisu \in \realopt\B(\epsilon)$ imply $\nextu \in \realopt\B(\epsilon)$, hence our assumption $u^0 \in \realopt\B(\epsilon)$ implies $\{\thisu\}_{i \in \N} \subset \realopt\B(\epsilon)$.
    Consequently also $\delta_{i+1} \le \delta_i \le \delta_0$ for all $i \in \N$,
    If now $\delta_0 < 1/2$, which is guaranteed by $\epsilon>0$ small enough and the continuity of $\grad^2 J$, then \eqref{eq:newton-superlinear-fejer} implies $\norm{\thisu-\realoptu}_{\grad^2 J(\realoptu)} \to 0$. Consequently $\delta_i \to 0$.

    Let $\tilde \delta_i \defeq \delta_i(1+\delta_i)/[(2-\delta_i)(1-\delta_i)]$.
    From \eqref{eq:newton-superlinear-fejer}, we get superlinear convergence if $\tilde\delta_i \to 0$, which follows from $\delta_i \to 0$.
    Superlinear convergence of order $q>1$ occurs if $\norm{\nextu-\realoptu}_{\grad^2 J(\realoptu)}/\norm{\thisu-\realoptu}_{\grad^2 J(\realoptu)}^q \to c$ for some $c \ge 0$.
    From \eqref{eq:newton-superlinear-fejer}, we see this to hold if $\tilde\delta_i/\norm{\thisu-\realoptu}^{2(q-1)} \to c \in \R$. If $\grad^2 J$ is Lipschitz near $\realoptu$, then $\delta_i \le C \norm{\thisu-\realoptu}$ for some constant $C>0$. Therefore we get superlinear convergence of order $q=2$.
\end{proof}

\cbend

\cbstart
\subsection{Convergence of function values}
\label{sec:value}

We now study how our framework can be used to derive the convergence, or ergodic convergence, of function values. We concentrate on algorithms that are variants of forward--backward splitting, including gradient descent and the proximal point method, although other algorithms can be handled similarly. We again use the three-point inequalities of \cref{sec:three-point}.

\begin{lemma}
    \label{lemma:graddesc-value}
    Let $H=\subdiff G+\grad J$ for $G,J \in \convex(X)$ with $\grad J$ $L$-Lipschitz. For all $i \in \N$, take $\Precond_{i+1} \equiv I$ and $\AltBregFn_{i+1}(u) \defeq \tau_i(\grad J(\thisu) - \grad J(u))$ with $\Step_{i+1}=\tau_i I$ as well as $\Test_{i+1} \equiv \tauTest_i I$ for some $\tau_i,\tauTest_i>0$.
    Then the fundamental condition \eqref{eq:convergence-fundamental-condition-iter-h} holds if 
    \begin{enumerate}[label=(\roman*)] 
        \item\label{item:graddesc-unaccel-value} $\phi_i \equiv \phi_0$ is constant, $\tau_i L < 1$, and \[
            \Penalty_{i+1}(\realoptu) \defeq -\tauTest_i\tau_i([G+J](\nextu)-[G+J](\realoptu)) - \tauTest_i(1-\tau_i L)\norm{u-\thisu}^2/2.
        \]
    \end{enumerate}
    If $J$ is strongly convex with factor $\gamma>0$, alternatively:
    \begin{enumerate}[label=(\roman*),resume]
        \item\label{item:graddesc-accel-value} 
            $\tau_0 L^2 < \gamma$, 
            $\tauTest_{i+1} \defeq \tauTest_i(1+\tau_i(\gamma-\tau_iL^2))$, $\tau_i \defeq \tauTest_i^{-1/2}$ or $\tau_i \defeq \tau_0$, 
            and
            \[
                \Penalty_{i+1}(\realoptu)=-\tauTest_i\tau_i([G+J](\nextu)-[G+J](\realoptu)).
            \]
    \end{enumerate}
\end{lemma}

\begin{proof}
    We fellow the proof of \cref{lemma:graddesc}, where we start by expanding \eqref{eq:convergence-fundamental-condition-iter-h} (instead of \eqref{eq:convergence-fundamental-condition-iter-h-alt}) as
    \begin{equation*}
        \frac{\tauTest_i}{2}\norm{u-\thisu}^2
        + \frac{\tauTest_i-\tauTest_{i+1}}{2}\norm{u-\realoptu}^2
        +\tauTest_i\tau_i\iprod{H(\nextu)}{u-\realoptu}
        \ge 0.
    \end{equation*}
    Note that we have not inserted $H(\realoptu) \ni 0$ here.
    Now, as the next step, we do not eliminate $G$ through monotonicity of $\subdiff G$, but use the definition of the convex subdifferential.
    Then we use the value three-point inequality \eqref{eq:three-point-smoothness} in place of the non-value inequality \eqref{eq:three-point-hypomonotonicity} and the value inequality \eqref{eq:three-point-smoothness-sc} in place of the non-value inequality \eqref{eq:three-point-hypomonotonicity-sc}. From here the claims follow as in the proof of \cref{lemma:graddesc}.
    Note the factor-of-two differences between these formulas, which are reflected in the step length rules: $\tau_i L < 1$ instead of $\tau_i L < 2$; $\tau_0 L^2 < \gamma$ instead of  $\tau_0 L^2 < 2\gamma$; and $\tauTest_{i+1} \defeq \tauTest_i(1+\tau_i(\gamma-\tau_iL^2))$ instead of $\tauTest_{i+1} \defeq \tauTest_i(1+\tau_i(2\gamma-\tau_iL^2))$.
\end{proof}

We now obtain the convergence to zero of a weighted function value difference over the history of iterates, and as a consequence, for an ergodic sequence formed from the iterates:

\begin{corollary}
    \label{cor:graddesc-ergodic}
    Suppose the conditions of \cref{lemma:graddesc-value} hold. Then
    \begin{equation}
        \label{eq:graddesc-di}
        \frac{\tauTest_N}{2}\norm{u^N-\realoptu}^2 + \sum_{i=0}^{N-1} \tauTest_i\tau_i([G+J](\nextu)-[G+J](\realoptu)) \le C_0 \defeq \frac{\tauTest_0}{2}\norm{u^0-\realoptu}^2.
    \end{equation}
    In consequence, if we define the ergodic sequence
    \[
        \tilde u_{N} \defeq \inv\zeta_{N}\sum_{i=0}^{N-1} \tauTest_i\tau_i \nextx,
        \quad\text{where}\quad
        \zeta_N \defeq \sum_{i=0}^{N-1} \tauTest_i\tau_i,
    \]
    then
    \begin{equation}
        \label{eq:graddesc-ergodic}
        [G+J](\tilde u_N) - [G+J](\realoptu) \le \frac{\tauTest_0}{2\zeta_N}\norm{u^0-\realoptu}^2.
    \end{equation}
    In particular, if \cref{lemma:graddesc-value}\ref{item:graddesc-unaccel-value} holds, then $[G+J](\tilde u_N) \to [G+J](\realoptu)$ at the rate $O(1/N)$. If, instead, \cref{lemma:graddesc-value}\ref{item:graddesc-accel-value} holds, then this convergence is linear.
\end{corollary}

\begin{proof}
    The basic inequality \eqref{eq:graddesc-di} is a consequence of the fundamental \cref{thm:convergence-result-main-h}. The ergodic estimate \eqref{eq:graddesc-ergodic} follows from there by Jensen's inequality. The first convergence rate estimate follows from \eqref{eq:graddesc-ergodic} are based on the fact that under \cref{lemma:graddesc-value}\ref{item:graddesc-unaccel-value} $\tauTest_i\tau_i=\tauTest_0\tau_0$ is a constant, so $\zeta_N=N\tauTest_0\tau_0$. Under \cref{lemma:graddesc-value}\ref{item:graddesc-unaccel-value} we recall from \cref{example:prox-accel} that the rule for $\tauTest_{i+1}$ shows that $\tauTest_{i+1}$ is grows exponentially with $\tau_i=\tau_0$ constant. Then also $\zeta_N$ is exponential, so we obtain linear rates.
\end{proof}

The following three examples follow from \cref{cor:graddesc-ergodic}.
For the proximal point method, additionally, since we can still let $\tau_i \upto \infty$ due to $L=0$, we can also get superlinear convergence. Also, in the case of the proximal point method, we use the strong convexity of $F$, which is for simplicity not considered in \eqref{lemma:graddesc-value}, but can easily be added.

\begin{example}[Proximal point method ergodic function value]
    \label{example:prox-gap}
    For the proximal point method of \cref{example:prox,example:prox-accel}, applied to $H=\subdiff G$ with $G \in \convex(\Space)$, we have $G(\tilde u_N) \to G(\realoptu)$ at the rate $O(1/N)$ when $\tau_i \equiv \tau_0$ and no strong convexity is present. If $G$ is strongly convex, and $\tau_i \equiv \tau_0$, the convergence is linear; if $\tau_i \upto \infty$, the convergence is superlinear.
\end{example}

\begin{example}[Gradient descent ergodic function value]
    \label{example:graddesc-gap}
    For the gradient descent method of \cref{example:graddesc,example:graddesc-accel}, applied to $J \in \convex(\Space)$ with $L$-Lipschitz gradient, if $\tau_i \equiv \tau_0$ with $\tau_0 L \le 1$, we have $J(\tilde u_N) \to J(\realoptu)$ at the rate the $O(1/N)$.
    If $J$ is strongly convex, $\tau_0 L^2 < \gamma$, and we update $\tau_{i+1} \defeq \tau_i/\sqrt{1+(2\gamma-\tau_iL^2)}$, then this convergence is $O(1/N^2)$.
\end{example}

\begin{example}[Forward--backward splitting ergodic function value]
    For the forward--backward splitting of \cref{example:fb}, $[G+J](\tilde u_N) \to [G+J](\realoptu)$ at exactly the same rates and conditions are for gradient descent in  \cref{example:graddesc-gap}.
\end{example}

For Newton's method, we can use similar arguments: we can replace \eqref{eq:three-point-hypomonotonicity-c2} by \eqref{eq:three-point-smoothness-c2} in \cref{lemma:newton}, and \eqref{eq:three-point-hypomonotonicity-c2x} by \eqref{eq:three-point-smoothness-c2x} in \cref{lemma:newton-superlinear}.
This can be done because the preceding non-value lemmas show that $\{\thisu\}_{i \in \N} \in \realopt\B(\epsilon)$.
In \cref{lemma:newton} the effect of the change is to replace $(1-\delta_i)^2$ by $\delta_i^2-3\delta_i$ everywhere, and in \cref{lemma:newton-superlinear}, to replace $2-\delta_i$ by $1-2\delta_i$. With these changes, the main arguments go through, although the exact value of $\kappa$ and the upper bounds for $\delta_i$ in the final paragraphs are changed.

\begin{example}[Newton's method function value]
    \label{example:newton-value}
    For Newton's method in \cref{example:newton}, we have $\tau_i = 1$ and $\tauTest_N \defeq \kappa^N\tauTest_0$ for some $\kappa>1$. We have $J(\tilde u^N) \to J(\realoptu)$ (super)linearly.
\end{example}

We can also obtain non-ergodic convergence for \term{monotone} methods. We demonstrate the idea only for the unaccelerated ($\tauTest_i\tau_i=\tauTest_0\tau_0$) proximal point method, but unaccelerated forward--backward splitting and gradient descent can be handled analogously.

\begin{example}[Proximal point method function value]
    \label{example:prox-gap-non-ergodic}
    For the proximal point method of \cref{example:prox,example:prox-accel}, applied to $H=\subdiff G$ with $G \in \convex(\Space)$, we have $G(u^N) \to G(\realoptu)$ at the rate $O(1/N)$ when $\tau_i \equiv \tau_0$ and no strong convexity is present. If $G$ is strongly convex, and $\tau_i \equiv \tau_0$, the convergence is linear; if $\tau_i \upto \infty$, the convergence is superlinear.
\end{example}

\begin{demonstration}
    From \eqref{eq:pp}, that is $0 \in \subdiff F(\nextu) + \tau_i(\nextu-\thisu)$, we have
    \begin{equation}
        \label{eq:prox-monotone}
        0 \le
        \inv\tau_i \norm{\nextx-\thisx}_X^2
        =
        \iprod{\subdiff G(\nextx)}{\thisx-\nextx}_X
        \le G(\thisx)-G(\nextx).
    \end{equation}
    That is, the proximal point method is monotone:
    Now we use \cref{cor:graddesc-ergodic}.
    Using \eqref{eq:prox-monotone} to unroll the function value sum in \eqref{eq:graddesc-di}  gives $\zeta_N [G(u^N)-G(\realoptu)] \le C_0$. The rates follow as in \cref{cor:graddesc-ergodic,example:prox-gap}.
\end{demonstration}

\cbend

\subsection{Connections to fixed point theorems}
\label{sec:browder}

\cbstart
We demonstrate connections of our approach to established fixed point theorems.
The following result in its modern form, stated for firmly non-expansive or more generally $\alpha$-averaged maps, can be first found in \cite{browder1967convergence}. Similar results for what are now known as Krasnoselski--Mann iterations, closely related to $\alpha$-averaged maps, were, however, stated earlier for more limited settings in \cite{mann1953mean,schaefer1957,petryshyn1966construction,krasnoselski1955two,opial1967weak}.
\cbend

\begin{example}[Browder's fixed point theorem]
    \label{example:browder}
    Let $T: \Space \to \Space$ be $\alpha$-averaged, that is $T=(1-\alpha)J+\alpha I$ for some non-expansive $J$ and $\alpha \in (0, 1)$.
    Suppose there exists a fixed point $\realoptu=T(\realoptu)$.
    Let $\nextu \defeq T(\thisu)$.
    Then $\thisu \weakto u^*$ for some fixed point $u^*$ of $T$.
\end{example}

\begin{demonstration}[Proof]
    Let us set $H(u) \defeq T(u)-u$, as well as $\Test_{i+1} \defeq \Step_{i+1} \defeq \Precond_{i+1} \defeq I$ and $\AltBregFn_{i+1}(u) \defeq T(\thisu)+\thisu-T(u)-u$. 
    We have
    \begin{equation}
        \label{eq:browder-tildeh}
        \tilde H_{i+1}(\nextu) \defeq \Step_{i+1} H(\nextu) + \AltBregFn_{i+1}(\nextu) = T(\thisu)+\thisu-2\nextu = \thisu-\nextu,
    \end{equation}
    where the last step follows by observing from the previous steps that \eqref{eq:pp} says $\nextu=T(\thisu)$.
    The expression \eqref{eq:browder-tildeh} easily gives the \cbstart iteration outer semicontinuity condition \eqref{eq:precond-continuity}, \cbend and reduces the fundamental condition \eqref{eq:convergence-fundamental-condition-iter-h} to
    \[
        \frac{1}{2}\norm{\nextu-\thisu}^2
        +
        \iprod{\thisu-\nextu}{\nextu-\realoptu}
        \ge
        -\Penalty_{i+1}(\realoptu).
    \]
    Using $\nextu=T(\thisu)$ and $\realoptu=T(\realoptu)$, and taking $\beta>0$, \eqref{eq:convergence-fundamental-condition-iter-h} therefore holds for
    \begin{equation}
        \label{eq:browder-penalty}
        \Penalty_{i+1}(\realoptu)
        =
        \frac{\alpha+2\beta-1}{2(1-\alpha)}\norm{\nextu-\thisu}^2
    \end{equation}
    provided
    \[
        0 \le D \defeq \frac{\beta}{1-\alpha}\norm{T(\thisu)-\thisu}^2
        +
        \iprod{
        \thisu-\realoptu
        -(T(\thisu)-T(\realoptu))
        }{T(\thisu)-T(\realoptu)}.
    \]
    Using the $\alpha$-averaged property and $\realoptu=J(\realoptu)$, we expand
    \begin{multline*}
        \frac{D}{1-\alpha}
        =
        \beta\norm{J(\thisu)-\thisu}^2
        +
        \iprod{\thisu-\realoptu-J(\thisu)+J(\realoptu)}{(1-\alpha)(J(\thisu)-J(\realoptu))+\alpha(\thisu-\realoptu)}
        \\
        =
        (\alpha+\beta)\norm{\thisu-\realoptu}^2
        +(\beta+\alpha-1)\norm{J(\thisu)-J(\realoptu)}^2
        -(2\alpha+2\beta-1)\iprod{J(\thisu)-J(\realoptu)}{\thisu-\realoptu}.
    \end{multline*}
    We take $\beta \defeq \max\{0, 1/2-\alpha\}$. Then $2\alpha+2\beta \ge 1$.
    Cauchy's inequality and non-expansivity of $J$ thus give
    \[
        \frac{D}{1-\alpha}
        \ge
        \frac{1}{2}\norm{\thisu-\realoptu}^2
        -\frac{1}{2}\norm{J(\thisu)-J(\realoptu)}^2
        \ge 0.
    \]
    This verifies \eqref{eq:convergence-fundamental-condition-iter-h}. From \eqref{eq:browder-penalty}, $\Penalty_{i+1}(\realoptu) \le -\frac{1}{2}\min\{1, \alpha/(1-\alpha)\}\norm{\nextu-\thisu}^2$.
    We now obtain the claimed convergence from \cref{cor:convergence-result-main-h-alt,prop:rateless}. 
\end{demonstration}

\section{Stochastic methods}
\label{sec:stochastic-examples}

\cbstart
We now exploit the fact that the step length $\Step_{i+1}$ can be a non-invertible operator. We do this in the context of stochastic block-coordinate methods. Towards this end we introduce the following probabilistic notations:
\cbend

\begin{definition}
    We write $x \in \Random(X)$ if $x$ is an $X$-valued random variable: $x: \Omega \to X$ for some (in the present work fixed) probability space $(\Omega, \SAlg)$, where $\SAlg$ is a $\sigma$-algebra on $\Omega$.
    We denote by $\E$ the expectation with respect to a probability measure $\P$ on $\Omega$.
    As is common, we abuse notation and write $x=x(\omega)$ for the unknown random realisation $\omega \in \Omega$.
    We also write $\E[\cdot|i]$ for the conditional expectation with respect to random variable realisations up to and including iteration $i$.
\end{definition}

We refer to \cite{shiryaev1996probability} for more details on measure-theoretic probability.

\cbstart
The following is an immediate corollary of \cref{thm:convergence-result-main-h}, obtained by taking the expectation of both \eqref{eq:convergence-fundamental-condition-iter-h} and \eqref{eq:convergence-result-main-h}. By only requiring these inequalities to hold in expectation may
may produce more lenient step length and other conditions. In the section, we demonstrate the flexibility of our techniques to stochastic methods with a few basic examples. We refer to the review article \cite{wright2015coordinate} for an  introduction and further references to stochastic coordinate descent, and to our companion paper \cite{tuomov-blockcp} for primal--dual methods based on the work here.

\begin{corollary}
    \label{cor:convergence-result-main-h-stoch}
    On a Hilbert space $\Space$ and a probability space $(\Omega, \SAlg)$, let $\tilde H_{i+1}: \Random(\Space \setto \Space)$, and $\Precond_{i+1}, \Test_{i+1} \in \Random(\linear(\Space; \Space))$ for $i \in \N$.
    Suppose \eqref{eq:ppext} is solvable for $\{\nextu\}_{i \in \N} \subset \Random(\Space)$.
    If for all $i \in \N$ and almost all random events $\omega \in \Omega$, $(\Test_{i+1}\Precond_{i+1})(\omega)$ is self-adjoint, and for some $\Penalty_{i+1} \in \Random(\R)$ and $\realoptu \in \Space$ the \term{expected fundamental condition}
    \begin{equation}
        \label{eq:convergence-fundamental-condition-iter-h-stoch}
        \tag{C$\E{\sim}$}
        \begin{split}
        \E[\iprod{\tilde H_{i+1}(\nextu)}{\nextu-\realoptu}_{\Test_{i+1}}]
        &
        \ge 
        \E\left[\frac{1}{2}\norm{\nextu-\realoptu}_{\Test_{i+2}\Precond_{i+2}-\Test_{i+1}\Precond_{i+1}}^2\right]
        \\ \MoveEqLeft[-1]
        - \E\left[\frac{1}{2}\norm{\nextu-\thisu}_{\Test_{i+1} \Precond_{i+1}}^2\right]
        - \E[\Penalty_{i+1}(\realoptu)],
        \end{split}
    \end{equation}
    holds, then so does the \term{expected descent inequality}
    \begin{equation}
        \label{eq:convergence-result-main-h-stoch}
        \tag{D$\E$}
        \E\left[\frac{1}{2}\norm{u^N-\realoptu}^2_{\Test_{N+1}\Precond_{N+1}}\right]
        \le
        \E\left[\frac{1}{2}\norm{u^0-\realoptu}^2_{\Test_{1}\Precond_{1}}\right]
        +
        \sum_{i=0}^{N-1} \E[\Penalty_{i+1}(\realoptu)]
        \quad
        (N \ge 1).
    \end{equation}
\end{corollary}

In block-coordinate descent methods, we write $u=\sum_{j=1}^m P_j u$ for some mutually orthogonal projections operators, and on each step of the method, only update some of the ``blocks'' $P_j u$. Functions with respect to which we take a proximal step, we assume \term{separable} with respect to these projections or subspaces: $G=\sum_{j=1}^m G_{j} \circ P_j$.
To perform forward steps, we introduce a blockwise version of standard smoothness conditions of convex functions. The idea is that the factor $L_{S(i)}$ for the subset of blocks $S(i)$ can be better than the global smoothness or Lipschitz factor $L$.

\cbend

\begin{definition}
    We write $(P_1, \ldots, P_m) \in \mathcal{P}(\Space)$ if $P_1,\ldots,P_m$ are projection operators in $\Space$ with $\sum_{j=1}^m P_j = I$, and $P_jP_i=0$ for $i \ne j$.
    For random $S(i) \subset \{1,\ldots,m\}$ and an iteration $i \in \N$, we then set
    \[
        P_{S(i)} \defeq \sum_{j \in S(i)} P_j,
        \quad\text{and}\quad
        \Pi_{S(i)} \defeq \sum_{j \in S(i)} \inv \pi_{j,i} P_j,
        \quad\text{where}\quad
        \pi_{j,i} \defeq \P[j \in S(i)] > 0.
    \]
    For smooth $J \in \convex(\Space)$, we let $L_{S(i)} > 0$ be the $\Pi_{S(i)}$-relative smoothness factor, satisfying
    \begin{equation}
        \label{eq:pi-smooth}
        J(u+\Pi_{S(i)}h) \le J(u) + \iprod{\grad J(u)}{h}_{\Pi_{S(i)}}+\frac{L_{S(i)}}{2}\norm{h}_{\Pi_{S(i)}}^2
        \quad (u, h \in \Space),
    \end{equation}
    and consequently (see \cref{lemma:subspace-smoothness})
    \begin{equation}
        \label{eq:pi-coco}
        \inv L_{S(i)} \norm{\grad J(u)-\grad J(v)}_{\Pi_{S(i)}}^2
        \le
        \iprod{\grad J(u)-\grad J(v)}{u-v},
        \quad (u, v \in \Space).
    \end{equation}
\end{definition}

\begin{example}[Stochastic block-coordinate descent]
    \label{example:stochgd}
    Let $H=\grad J$ for $J \in \convex(\Space)$ with Lipschitz gradient.
    Also let $(P_1,\ldots,P_m) \in \mathcal{P}(\Space)$.
    For each $i \in \N$, take random $S(i) \subset \{1,\ldots,n\}$, and set
    \begin{equation}
        \label{eq:stochgd-altbregfn}
        \Step_{i+1} \defeq \tau_i \Pi_{S(i)},
        \quad
        \Precond_{i+1} \defeq I,
        \quad\text{and}\quad
        \AltBregFn_{i+1}(u) \defeq
        \Step_{i+1}[\grad J(\thisu)-\grad J(u)].
    \end{equation}
    Then \eqref{eq:pp} says that we take a forward step on the random subspace $\mathop{\mathrm{range}}(\Pi_{S(i)})$:
    \begin{equation}
        \label{eq:stochastic-graddesc-update}
        \nextu=\thisu - \tau_i \Pi_{S(i)} \grad J(\thisu).
    \end{equation}
    If the step lengths are deterministic and satisfy $\epsilon \le \tau_i$ and $\tau_i L_{S(i)} \le \pi_{j,i}$ for all $j \in S(i)$ for some $\epsilon>0$, we have $\E[J(\tilde u_N)] \to J(\realoptu)$ at the rate $O(1/N)$ \cbstart for the ergodic sequence
    \[
        \tilde u_N \defeq \inv\zeta_N\sum_{i=0}^{N-1}\E[\tau_i\Pi_{S(i)}\nextu]
        \quad\text{where}\quad
        \zeta_N \defeq \sum_{i=0}^{N-1} \tau_i
        \quad (N \ge 1).
    \] \cbend
    Through the use of the ``local'' smoothness factors $L_{S(i)}$, the method may be able to take larger steps $\tau_i$ than those allowed by the global factor $L$ in \cref{example:graddesc}.
\end{example}

The smoothness of $G$ limits the usefulness of \cref{example:stochgd}. However, it forms the basis for popular stochastic forward--backward splitting methods, of which we now provide an example.

\begin{example}[Stochastic forward--backward splitting]
    \label{example:stochfb}
    Let $(P_1,\ldots,P_m) \in \mathcal{P}(\Space)$.
    Suppose $H=\subdiff G+\grad J$ for $J, G \in \convex(\Space)$, where $J$ has Lipschitz gradient, and $G$ is separable: $G=\sum_{j=1}^m G_{j} \circ P_j$.
    Take $\Precond_{i+1}$, $\Step_{i+1}$, and $\AltBregFn_{i+1}$ as in \cref{example:stochgd}.
    Then \eqref{eq:pp} describes the stochastic forward--backward splitting method
    \begin{equation*}
        \nextu \defeq \inv{(I+\tau_i \Pi_{S(i)}\subdiff G)}\bigl(\thisu - \tau_i \Pi_{S(i)} \grad J(\thisu)\bigr).
    \end{equation*}
    With $u_j \defeq P_j u$, this can be written
    \begin{equation*}
        \nextu_j \defeq
        \begin{cases}
            \inv{(I+\tau_i \inv\pi_{j,i} \subdiff G_{j})}\bigl(\thisu_j - \tau_i \inv\pi_{j,i} P_j \grad J(\thisu)\bigr), & j \in S(i), \\
            u_j, & j \not\in S(i).
        \end{cases}
    \end{equation*}
    The method has exactly the same convergence properties as the stochastic gradient descent of \cref{example:stochgd}.
\end{example}

\begin{remark}
    Following \cref{example:graddesc-accel}, if $G$ or $J$ is strongly convex, it is also possible to construct accelerated versions of both \cref{example:stochgd,example:stochfb}. Then we can obtain from \eqref{eq:convergence-result-main-h-stoch} convergence rates for $\E[\norm{\nextu-\realoptu}^2]$.
\end{remark}


\begin{demonstration}[Proof of convergence of stochastic gradient descent and forward--backward splitting]\cbstart
    We take as the testing operator $\Test_{i+1} \defeq I$. Then, since $\Test_{i+1}\Precond_{i+1} \equiv I$, \eqref{eq:convergence-fundamental-condition-iter-h-stoch} expands as
    \begin{equation}
        \label{eq:convergence-fundamental-condition-iter-h-stochgd}
        \E[\tau_i\iprod{\subdiff G(\nextu)+\grad J(\thisu)}{\nextu-\realoptu}_{\Pi_{S(i)}}]
        \ge 
        - \E\left[\frac{1}{2}\norm{\nextu-\thisu}^2\right]
        - \E[\Penalty_{i+1}(\realoptu)].
    \end{equation}
    From the decomposition $G=\sum_{j=1}^m G_{j} \circ P_j$ and the convexity of $G_{j}$, we observe that
    \[
        \begin{split}
        \tau_i\iprod{\subdiff G(\nextu)}{\nextu-\realoptu}_{\Pi_{S(i)}}
        &
        =
        \sum_{j=1}^m \tau_i\inv\pi_{j,i}\chi_{S(i)}(j)\iprod{\subdiff G_{j}(P_j \nextu)}{P_j(\nextu-\realoptu)}
        \\
        &
        \ge
        \sum_{j=1}^m \tau_i\inv\pi_{j,i}\chi_{S(i)}(j)(G_{j}(P_j \nextu)-G_{j}(P_j \realoptu)).
        \end{split}
    \]
    Since $\tau_i$ is deterministic and $\E[\inv\pi_{j,i}\chi_{S(i)}(j)P_j]=\E[\Pi_{S(i)}]=I$, such that $\sum_{i=0}^{N-1} \E[\tau_i \inv\pi_{j,i}\chi_{S(i)}(j)P_j] = \zeta_N$ for all $j=1,\ldots,m$, by Jensen's inequality, therefore,
    \begin{equation}
        \label{eq:g-ergodic-convexity-stochgd}
        \sum_{i=0}^{N-1} \E[\tau_i\iprod{\subdiff G(\nextu)}{\nextu-\realoptu}_{\Pi_{S(i)}}]
        \ge
        \zeta_N\left(G(\tilde u_N)-G(\realoptu)\right).
    \end{equation}
    If we show the ergodic three-point smoothness condition
    \begin{equation}
        \label{eq:j-smooth-ergodic-stochgd}
        J(\realoptu)-
        J(\tilde u_N)
        \ge
        \sum_{i=0}^{N-1} \E\bigl[\inv\zeta_N\tau_i\iprod{\grad  J(\thisu)}{\realoptu-\nextu}_{\Pi_{S(i)}}
        -\frac{L_{S(i)}\inv\zeta_N\tau_i}{2}\norm{\nextu-\thisu}_{\Pi_{S(i)}}^2\bigr],
    \end{equation}     
    then using our assumption $\tau_i L_{S(i)} \le \pi_{j,i}$ and \eqref{eq:g-ergodic-convexity-stochgd}, we verify \eqref{eq:convergence-fundamental-condition-iter-h-stochgd}, hence \eqref{eq:convergence-fundamental-condition-iter-h-stoch}, for some $\Delta_{i+1}(\realoptu)$ such that
    \[
        \sum_{i=0}^{N-1} \E[\Delta_{i+1}(\realoptu)] =
        -\zeta_N\left(G(\tilde u_N)-G(\realoptu)\right).
    \]
    Since $\zeta_N \ge \epsilon N$ by our assumption $\tau_i \ge \epsilon$, \cref{cor:convergence-result-main-h-stoch} now shows the $O(1/N)$ convergences of function values for the ergodic sequence $\{\tilde u_N\}_{N \ge 1}$.
    \cbend

    To prove \eqref{eq:j-smooth-ergodic-stochgd}, from \eqref{eq:pi-smooth} with $h \defeq \nextu-\thisu$ and $\nexxt{\bar u} \defeq (I-\Pi_{S(i)})\thisu+\Pi_{S(i)}\nextu$ we have
    \begin{equation}
        \label{eq:j-smooth-ergodic-deriv1-stochgd}
        J(\thisu)-J(\nexxt{\bar u}) \ge \iprod{\grad J(\thisu)}{\thisu-\nextu}_{\Pi_{S(i)}}
            -\frac{L_{S(i)}}{2}\norm{\nextu-\thisu}^2_{\Pi_{S(i)}}.
    \end{equation} 
    By convexity, we also have
    \begin{equation}
        \label{eq:j-smooth-ergodic-deriv2-stochgd}
        \begin{split}
        J(\realoptu)-J(\thisu) & \ge 
        \iprod{\grad J(\thisu)}{\realoptu-\thisu}
        =\iprod{\grad J(\thisu)}{\realoptu-\thisu}_{\E[\Pi_{S(i)}|i]}
        \\ &
        =\E[\iprod{\grad J(\thisu)}{\realoptu-\thisu}_{\Pi_{S(i)}}|i].
        \end{split}
    \end{equation}
    Summing \eqref{eq:j-smooth-ergodic-deriv1-stochgd} and \eqref{eq:j-smooth-ergodic-deriv2-stochgd}, multiplying by $\tilde\tau_i$, and taking the expectation, 
    \begin{equation}
        \label{eq:j-smooth-erogdic-deriv3-stochgd}
        J(\realoptu)-\E[\tau_i J(\nexxt{\bar u})]
        \ge
        \E\bigl[\tau_i\iprod{\grad  J(\thisu)}{\realoptu-\nextu}_{\Pi_{S(i)}}
        -\frac{L_{S(i)}\tau_i}{2}\norm{\nextu-\thisu}_{\Pi_{S(i)}}^2\bigr].
    \end{equation}
    Since $\sum_{i=0}^{N-1} \tau_i=\zeta_N$, Jensen's inequality shows
    \[
        \sum_{i=0}^{N-1}\E[\inv\zeta_N\tau_i J(\nexxt{\bar u})]
        \ge
        J\Biggl(\sum_{i=0}^{N-1}\E[\inv\zeta_N\tau_i\nexxt{\bar u}]\Biggr)
        \ge
        J\Biggl(\sum_{i=0}^{N-1} \E[\inv\zeta_N\tau_i\Pi_{S(i)}\nextu]\Biggr)
        =J(\tilde u_N).
    \]
    Therefore, summing \eqref{eq:j-smooth-erogdic-deriv3-stochgd} over $i=0,\ldots,N-1$ verifies \eqref{eq:j-smooth-ergodic-stochgd}.
\end{demonstration}


\begin{example}[Stochastic Newton's method]
    \label{example:newton-stoch}
    Suppose $(P_1,\ldots,P_m) \in \mathcal{P}(\Space)$ and $J \in C^2(\Space)$. 
    Take $H=\grad J$, $\Step_{i+1} \defeq P_{S(i)}$, and
    \begin{equation*}
        \BregFn_{i+1}(u)
            \defeq [\grad^2 J(\thisu)-(I-P_{S(i)})\grad^2 J(\thisu)P_{S(i)}](u-\thisu)
            + P_{S(i)}[\grad J(\thisu) - \grad J(u)].
    \end{equation*}
    Then \eqref{eq:pp} reads
    \begin{equation*}
        0 = P_{S(i)} \grad J(\thisu) + [\grad^2 J(\thisu)]_{S(i)}(\nextu - \thisu) + [\grad^2 J(\thisu)]_{S(i)^c}(\nextu-\thisu),
    \end{equation*}
    where we abbreviate $A_{S(i)} \defeq P_{S(i)}AP_{S(i)}$.
    We get
    \[
        \nextu = \thisu + \pinv{[\grad^2 J(u)]_{S(i)}}\grad J(\thisu),
    \]
    where we define $\pinv A_{S(i)}$ to satisfy $\pinv A_{S(i)}=P_{S(i)}\pinv A_{S(i)}P_{S(i)}$ and $A_{S(i)} \pinv A_{S(i)} = \pinv A_{S(i)} A_{S(i)} = P_{S(i)}$.
    This is a variant of stochastic Newton's method and ``sketching'' \cite{qu2015sdna,pilanci2016iterative}.
    Notice how $\pinv{[\grad^2 J(u)]_{S(i)}}$ can be significantly cheaper to compute than $\inv{[\grad^2 J(u)]}$.

    \cbstart
    Let
    \begin{equation}
        \label{eq:stochgd-delta-j}
        \delta_J \defeq \inf\left\{\delta \ge 0 \,\middle|\,
            \begin{array}{r}
                 (1-\delta)\grad^2 J(\eta) \le \grad^2 J(\zeta) \le (1+\delta)\grad^2 J(\eta) \\
                 \text{ for all } \eta, \zeta \in \Space
            \end{array}
            \right\},
    \end{equation}
    as well as
    \begin{equation}
        \label{eq:stochgd-bar-p}
        \bar p \defeq \sup\left\{
            \bar p \in (0, 1]
            \,\middle|\,
            \begin{array}{r}            
                \E[(I-P_{S(i)})\grad^2 J(\zeta)(I-P_{S(i)})|i] \le (1-\bar p)\grad^2 J(\zeta)
                \\
                \text{ for all } \zeta \in \Space
                \text{ and iterations } i \in \N
            \end{array}
            \right\}.
    \end{equation}   
    If $0 \le \delta_J < \frac{3 - \sqrt{9-8\bar p}}{4}$, then $\E[\norm{u^N-\realoptu}^2] \to 0$ at a linear rate. 
    \cbend
\end{example}

\begin{remark}
    If $J(u)=\iprod{u}{Au-c}$ for some self-adjoint positive definite $A \in \linear(\Space; \Space)$ and $c \in \Space$, then $\delta_J=0$, so the upper bound on $\delta_J$ is satisfied for any $\bar p \in (0, 1]$.
    If $\E[P_{S(i)}|i] \equiv p I$ for some $p>1/2$, then $\bar p>0$ due to
    \[
        \E[(I-P_{S(i)})\grad^2 J(\zeta)(I-P_{S(i)})|i]
        =
        (1-2p)\grad^2 J(\zeta)+\E[P_{S(i)}\grad^2 J(\zeta)P_{S(i)}|i] \le
        2(1-p)\grad ^2 J(\zeta).
    \]
\end{remark}

An advantage of our techniques is the immediate convergence of:

\begin{example}[Stochastic proximal Newton's method]
    Let $(P_1,\ldots,P_m) \in \mathcal{P}(\Space)$.
    Let $H=\subdiff G+\grad J$ for $G \in \convex (\Space)$ and $J \in C^2(X)$ with $G=\sum_{j=1}^m G_{j} \circ P_j$.
    Take $\Precond_{i+1}$, $\Step_{i+1}$, and $\AltBregFn_{i+1}$ as in \cref{example:newton-stoch}.
    Then we obtain the algorithm
    \begin{equation*}
        \nextu \defeq \inv{(I+\pinv{[\grad^2 J(u)]_{S(i)}}P \subdiff G)}\bigl(\thisu - \pinv{[\grad^2 J(u)]_{S(i)}} \grad J(\thisu)\bigr).
    \end{equation*}
    We have $\E[\norm{u^N-\realoptu}^2] \to 0$ at a linear rate under the same conditions as in \cref{example:newton-stoch}.
\end{example}

\begin{demonstration}[Proof of convergence of stochastic Newton's and proximal Newton's methods]
    \cbstart
    We set as the preconditioner $\Precond_{i+1} \defeq \grad^2 J(\thisu)$ and as the test $\Test_i \defeq \tauTest_i I$ for some $\tauTest_i > 0$.
    Clearly we have the following simpler non-value version of the value estimate \eqref{eq:g-ergodic-convexity-stochgd}:
    \begin{multline}
        \label{eq:stoch-newton-g0-monotone}
        \iprod{\subdiff G(\nextu)-\subdiff G(\realoptu)}{\nextu-\realoptu}_{\Test_{i+1}\Step_{i+1}}
        =
        \tauTest_i \iprod{\subdiff G(\nextu)-\subdiff G(\realoptu)}{\nextu-\realoptu}_{P_{S(i)}}
        \\
        =
        \sum_{j=1}^m \tauTest_i \chi_{S(i)}(j) \iprod{\subdiff G_{j}(P_j\nextu)-\subdiff G_{j}(P_j\realoptu)}{P_j(\nextu-\realoptu)}
        \ge
        0.
     \end{multline}
    Therefore, since $0 \in \subdiff G(\realoptu)+\grad J(\realoptu)$, the expected fundamental condition \cref{eq:convergence-fundamental-condition-iter-h-stoch} becomes
    \begin{equation}
        \label{eq:convergence-fundamental-condition-iter-h-newton-expected}
        \E[\tauTest_i D_{i+1}+\Penalty_{i+1}(\realoptu)]
        \ge  
        \E\left[
        \frac{1}{2}\norm{\nextu-\realoptu}_{\tauTest_{i+1}\grad^2 J(\nextu)-\tauTest_i\grad^2 J(\thisu)}^2
        - \frac{1}{2}\norm{\nextu-\thisu}_{\tauTest_i\grad^2 J(\thisu)}^2\right].
    \end{equation}
    for
    \begin{equation*}
        D_{i+1}
        \defeq
        \iprod{\grad J(\thisu)-\grad J(\realoptu)}{\nextu-\realoptu}_{P_{S(i)}}
        -\iprod{(I-P_{S(i)})\grad^2 J(\thisu)P_{S(i)}(\nextu-\thisu)}{\nextu-\realoptu}.
    \end{equation*}
    Adapting the argumentation of \cref{lemma:c2-smoothness,lemma:c2x-smoothness} to the present projected setting, by the mean value theorem, for some $\zeta$ between $\thisu$ and $\realoptu$, and using the definition of $\delta_J$ in \eqref{eq:stochgd-delta-j} and the three-point identity \eqref{eq:standard-identity}, we rearrange
    \begin{equation*}
        \begin{split}
        D_{i+1} &
        =
        \iprod{\grad^2 J(\thisu)(\thisu-\realoptu)}{\nextu-\realoptu}_{P_{S(i)}}
        +
        \iprod{[\grad^2 J(\zeta)-\grad^2 J(\thisu)](\thisu-\realoptu)}{\nextu-\realoptu}_{P_{S(i)}}
        \\
        \MoveEqLeft[-1]
        -\iprod{(I-P_{S(i)})\grad^2 J(\thisu)P_{S(i)}(\nextu-\thisu)}{\nextu-\realoptu}
        \\
        &
        =
        \iprod{\grad^2 J(\thisu)(\thisu-\realoptu)}{\nextu-\realoptu}
        +
        \iprod{[\grad^2 J(\zeta)-\grad^2 J(\thisu)](\thisu-\realoptu)}{\nextu-\realoptu}_{P_{S(i)}}
        \\
        \MoveEqLeft[-1]
        -\iprod{(I-P_{S(i)})\grad^2 J(\thisu)P_{S(i)}(\nextu-\realoptu)}{\nextu-\realoptu}
        \\
        &
        =
        \frac{1}{2}\norm{\nextu-\realoptu}_{\grad^2 J(\thisu)}^2
        +\frac{1}{2}\norm{\thisu-\realoptu}_{\grad^2 J(\thisu)}^2
        -\frac{1}{2}\norm{\nextu-\thisu}_{\grad^2 J(\thisu)}^2
        \\
        \MoveEqLeft[-1]
        +
        \iprod{[\grad^2 J(\zeta)-\grad^2 J(\thisu)](\thisu-\realoptu)}{\nextu-\realoptu}_{P_{S(i)}}
        \\
        \MoveEqLeft[-1]
        -\iprod{(I-P_{S(i)})\grad^2 J(\thisu)P_{S(i)}(\nextu-\realoptu)}{\thisu-\realoptu}.
        \end{split}
    \end{equation*}
    By the definition of $\bar p$ in \eqref{eq:stochgd-bar-p} and by Cauchy's inequality, for any $\alpha>0$, we obtain the expected three-point inequality
    \begin{equation*}
        \begin{split}
        \E[D_{i+1}]
        &
        \ge
        \E\Bigl[
        \frac{1-\delta_J-\inv\alpha}{2}\norm{\nextu-\realoptu}_{\grad^2 J(\thisu)}^2
        +\frac{1-\delta_J-\alpha(1-\bar p)}{2}\norm{\thisu-\realoptu}_{\grad^2 J(\thisu)}^2
        \\
        \MoveEqLeft[-2]
        -\frac{1}{2}\norm{\nextu-\thisu}_{\grad^2 J(\thisu)}^2
        \Bigr].
        \end{split}
    \end{equation*}

    We take $\alpha=(1-\delta_J)/(1-\bar p)$.
    Then \eqref{eq:convergence-fundamental-condition-iter-h-newton-expected} holds when
    \[
        \E[\Penalty_{i+1}(\realoptu)]
        \ge        
        \E\Bigl[\frac{1}{2}\norm{\nextu-\realoptu}_{\tauTest_{i+1}\grad^2 J(\nextu)-\tauTest_i(2-\delta_J-\inv\alpha)\grad^2 J(\thisu)}^2
        \Bigr].
    \]
    This is the case for some $\Delta_{i+1}(\realoptu) \in \Random(\R)$ with $\E[\Delta_{i+1}(\realoptu)]=0$ provided $2 > \delta_J + \inv\alpha$ and $\tauTest_{i+1}>0$ is small enough that $\tauTest_{i+1}\grad^2 J(\nextu) \le \tauTest_i(2-\delta_J-\inv\alpha)\grad^2 J(\thisu)$. Due to \eqref{eq:stochgd-delta-j}, we can take $\tauTest_{i+1} \ge \tauTest_i \kappa$ for
    \[
        \kappa \defeq \frac{2-\delta_J-\inv\alpha}{1+\delta_J}=
        \frac{2-\delta_J-\frac{1-\bar p}{1-\delta_J}}{1+\delta_J}
        =\frac{1+\bar p-3\delta_J+\delta_J^2}{1-\delta_J^2}.
    \]
    In particular, we obtain exponential growth of $\{\tauTest_i\}_{k \in \N}$ provided $\kappa>1$, which holds when $-3\delta_J + 2\delta_J^2 + \bar p > 0$, which is the case under our assumption $0 \le \delta_J < \frac{3 - \sqrt{9-8\bar p}}{4}$.
    Consequently, we can take $\tauTest_i \defeq \kappa^i/(1-\delta_J)$ for $\kappa >1$.
    By \cref{cor:convergence-result-main-h-stoch} we have
    \begin{equation*}
        \E\left[\frac{1}{2}\norm{u^N-\realoptu}^2_{\Test_{N+1}\Precond_{N+1}}\right]
        \le
        \E\left[
        \frac{1}{2}\norm{u^0-\realoptu}^2_{\Test_{1}\Precond_{1}}
        \right]
        \quad
        (N \ge 1).
    \end{equation*}
    Since $\Test_{N+1}\Precond_{N+1} =\tauTest_N\grad^2 J(\thisu) \ge \kappa^N \grad^2 J(\realoptu)$, we obtain the claimed linear expected convergence of iterates.
    %
\end{demonstration}

\cbend

\begin{remark}[Variance estimates]
    From an estimate of the type $\E[\norm{u^N-\realoptu}^2] \le C_N$, as above, Jensen's inequality gives $\norm{\E[u^N]-\realoptu}^2 \le C_N$. From this, with the application of the triangle and Cauchy's inequalities, it is easy to derive the variance estimate $\E[\norm{\E[u^N]-u^N}^2] \le 4C_N$.
\end{remark}

\section{Saddle point problems}
\label{sec:saddle}

\cbstart
We now momentarily forget the stochastic setting and ergodic estimates to which we will return in \cref{sec:gap}, and introduce our overall approach to primal--dual methods for saddle-point problems. \cbend
With $K \in \linear(X; Y)$; $G,J \in \convex(X)$; and $F^* \in \convex(Y)$ on Hilbert spaces $X$ and $Y$, we now wish to solve the following version of \eqref{eq:saddle}.
The first-order necessary optimality conditions read
\begin{equation*}
    -K^* \realopty \in \subdiff [G+J](\realoptx),
    \quad\text{and}\quad
    K \realoptx \in \subdiff F^*(\realopty).
\end{equation*}
Setting $\Space \defeq X \times Y$ and introducing the variable splitting notation $u=(x, y)$, $\realoptu=(\realoptx,\realopty)$, etc., this can succinctly be written as $0 \in H(\realoptu)$ in terms of the operator 
\begin{equation}
    \label{eq:h}
    H(u) \defeq
        \begin{pmatrix}
            \subdiff [G+J](x) + K^* y \\
            \subdiff F^*(y) -K x
        \end{pmatrix}.
\end{equation}
In this section, concentrating on this specific $H$, we specialise the theory of \cref{sec:convergence-basic} to saddle point problems.
Throughout, for some primal and dual step length and testing operators $\Tau_i,\TauTest_i \in \linear(X; X)$, and $\Sigma_{i+1}, \SigmaTest_{i+1} \in \linear(Y; Y)$, we take
\begin{equation}
    \label{eq:test}
    \Step_{i+1} \defeq
    \begin{pmatrix}
        \Tau_i & 0 \\
        0 & \Sigma_{i+1}
    \end{pmatrix},
    \quad\text{and}\quad
    \Test_{i+1} \defeq
        \begin{pmatrix}
            \TauTest_i & 0 \\
            0 & \SigmaTest_{i+1}
        \end{pmatrix}.
\end{equation}

To work with arbitrary step length operators, which will be necessary for stochastic algorithms in \cref{sec:stochastic-examples}, as well as the partially accelerated algorithms of \cite{tuomov-cpaccel}, we will need abstract forms of partial strong monotonicity of $G$ and $F^*$. As a first step, we take subspaces of operators
\[
    \TauMonotoneSpace \subset \linear(X; X),
    \quad\text{and}\quad
    \SigmaMonotoneSpace \subset \linear(Y; Y).
\]
We suppose that $\subdiff G$ is \emph{partially (strongly) $\TauMonotoneSpace$-monotone}, which we take to mean
\begin{equation}
    \label{eq:g-strong-monotone}
    \tag{G-PM}
    \iprod{\subdiff G(x') - \subdiff G(x)}{x'-x}_{\TauTilde} \ge 
        \norm{x'-x}_{\TauTilde\Gamma}^2,
    \quad (x, x' \in X;\, \TauTilde \in \TauMonotoneSpace)
\end{equation}
for some linear operator $0 \le \Gamma \in \linear(X; X)$. The operator $\TauTilde \in \TauMonotoneSpace$ acts as a testing operator. \cbstart Observe that we have already proven this in  \eqref{eq:stoch-newton-g0-monotone} for the setting of the stochastic Newton's method. \cbend
Similarly, we assume that $\subdiff F^*$ is \emph{$\SigmaMonotoneSpace$-monotone} in the sense 
\begin{equation}
    \label{eq:f-monotone}
    \tag{F$^*$-PM}
    \iprod{\subdiff F^*(y') - \subdiff F^*(y) }{y'-y}_\SigmaTilde \ge
        0 
    \quad (y, y' \in Y;\, \SigmaTilde \in \SigmaMonotoneSpace).
\end{equation}
Regarding $J$, we assume that $\grad J$ exists and is \emph{partially $\TauMonotoneSpace$-co-coercive} in the sense that for some $L \ge 0$ holds
\begin{equation}
    \label{eq:j-coco}
    \tag{J-PC}
    \iprod{\grad J(x') - \grad J(x)}{x'-x}_{\TauTilde} \ge 
        \inv L \norm{\grad J(x')-\grad J(x)}_{\TauTilde}^2,
    \quad (x, x' \in X;\, \TauTilde \in \TauMonotoneSpace).
\end{equation}
(We allow $L=0$ for the case $J=0$.)

We also introduce
\begin{equation}
    \label{eq:gammalift-q}
    \GammaLift{i+1}(\Gamma) \defeq \begin{pmatrix}
        2\Tau_i\Gamma  & 2\Tau_i K^* \\
        -2\Sigma_{i+1}K & 0 
    \end{pmatrix},
    \quad\text{and}\quad
    Q_{i+1}(L) \defeq \begin{pmatrix} L\Tau_i & 0 \\ 0 & 0 \end{pmatrix},
\end{equation}
which are operator measures of strong monotonicity and smoothness of $H$.
Finally, we introduce the forward--step preconditioner with respect to $J$, familiar from \cref{example:graddesc} as
\begin{equation}
    \label{eq:v-j}
    \BregFn_{i+1}^J(u) \defeq \begin{pmatrix}
        \Tau_i(\grad J(\thisx)-\grad J(x)) \\
        0
        \end{pmatrix}.
\end{equation}

\begin{example}[Block-separable structure, monotonicity]
	\label{example:separable}
	Let $P_1,\ldots,P_m$ be projection operators in $X$ with $\sum_{j=1}^m P_j = I$ and $P_jP_i=0$ if $i \ne j$.
    Suppose $G_1,\ldots,G_m \in \convex(X)$ are (strongly) convex with factors $\gamma_1,\ldots,\gamma_m \ge 0$.
    Then the partial strong monotonicity \eqref{eq:g-strong-monotone} holds with $\Gamma = \sum_{j=1}^m \gamma_j P_j$ for
	\begin{align}
	    \label{eq:g-tau-separable}
	    G(x) &= \sum_{j=1}^m G_j(P_j x),
	    \quad\text{and}\quad
	    \mathcal{T} = \biggl\{
	    	\Tau \defeq \sum_{j \in S} t_j P_j
	    	\biggm|
	    	t_j > 0, \, S \subset \{1,\ldots,m\}
	    \biggr\}.%
	\end{align}
\end{example}

\subsection{Estimates}

Using the (strong) $\TauMonotoneSpace$-monotonicity of $\subdiff G$, and the $\TauMonotoneSpace$-co-coercivity of $\grad J$, the next lemma simplifies \cref{cor:convergence-result-main-h-alt} for $H$ given by \eqref{eq:h}.
We introduce $\tilde\Gamma=\Gamma$ to facilitate later gap estimates that will require the conditions in the lemma to hold for $\tilde\Gamma=\Gamma/2$ instead of $\tilde\Gamma=\Gamma$.

\begin{theorem}
    \label{thm:convergence-result-saddle}
    Let $H$ have the structure \eqref{eq:h} and assume $\realoptu \in \inv H(0)$.
    Suppose $G$ satisfies the partial strong monotonicity \eqref{eq:g-strong-monotone} for some $0 \le \Gamma \in \linear(X; X)$,  $F^*$ similarly satisfies \eqref{eq:f-monotone}, and $J$ satisfies the partial co-coercivity \eqref{eq:j-coco} for some $L \ge 0$.
	For each $i \in \N$, let $\Tau_{i}, \TauTest_i \in \L(X;X)$ and $\Sigma_{i+1}, \SigmaTest_{i+1} \in \L(Y; Y)$ be such that $\TauTest_i\Tau_i \in \TauMonotoneSpace$ and $\SigmaTest_{i+1}\Sigma_{i+1} \in \SigmaMonotoneSpace$.
    Define $\Test_{i+1}$ and $\Step_{i+1}$ through \eqref{eq:test}.
    Also take $\AltBregFn_{i+1}: X \times Y \to X \times Y$, and $\Precond_{i+1} \in \linear(X \times Y; X \times Y)$.
    Suppose \eqref{eq:pp} is solvable for $\{\nextu\}_{i \in \N} \subset X \times Y$.
    Then the fundamental conditions \eqref{eq:convergence-fundamental-condition-iter-h-alt}, \eqref{eq:convergence-fundamental-condition-iter-h} and the descent inequality \eqref{eq:convergence-result-main-h} hold if \cbstart for all $i \in \N$, the operator \cbend $\Test_{i+1}\Precond_{i+1}$ is self-adjoint and for $\tilde\Gamma=\Gamma$ and $L_i \equiv L/2$ we have the
    \term{fundamental inequality for saddle-point problems}
    \begin{multline}
        \label{eq:convergence-fundamental-condition-iter}
        \tag{CI-$\Gamma$}
        \frac{1}{2}\norm{\nextu-\thisu}_{\Test_{i+1}(\Precond_{i+1}-Q_{i+1}(L_i))}^2
        + \frac{1}{2}\norm{\nextu-\realoptu}_{\Test_{i+1}(\GammaLift{i+1}(\tilde\Gamma)+\Precond_{i+1})-\Test_{i+2}\Precond_{i+2}}^2
        \\
        + \iprod{\AltBregFn_{i+1}(\nextu)-\BregFn_{i+1}^J(\nextu)}{\nextu - \realoptu}_{\Test_{i+1}}
        \ge
        - \Penalty_{i+1}(\realoptu).
    \end{multline}
\end{theorem}

We have introduced $\tilde\Gamma$ and $L_i$ for later gap estimates, where the specific choices of these will differ by a factor of two, similarly to the differences in the step length bounds for the function value estimates of \cref{sec:value} compared to the non-value estimates of \cref{sec:examples-basic}.

\begin{proof}
    Note that $\Test_{i+1}\Precond_{i+1}$ being self-adjoint implies that so is $\TauTest_i\Tau_i$.
    Using \eqref{eq:j-coco}, similarly to \cref{lemma:smoothness} we derive
    \begin{equation*}
        \iprod{\grad J(\thisx)-  \grad J(\realoptx)}{\nextx-\realoptx}_{\TauTest_i\Tau_i}
        \ge
        -\frac{L}{4}\norm{\nextx-\thisx}^2_{\TauTest_i\Tau_i}.
    \end{equation*}
    Using \eqref{eq:v-j}, therefore
    \begin{equation*}
    	\iprod{\BregFn_{i+1}^J(\nextu)}{\nextu - \realoptu}_{\Test_{i+1}}
    	\ge
    	-\frac{L}{4}\norm{\nextx-\thisx}^2_{\TauTest_i\Tau_i}
    	-\iprod{\grad J(\nextx)-  \grad J(\realoptx)}{\nextx-\realoptx}_{\TauTest_i\Tau_i}.
    \end{equation*}
	With this, \eqref{eq:g-strong-monotone}, and \eqref{eq:f-monotone}, we observe \eqref{eq:convergence-fundamental-condition-iter} to imply
    \begin{multline}
        \label{eq:convergence-fundamental-condition-iter0}
        \frac{1}{2}\norm{\nextu-\thisu}^2_{\Test_{i+1}\Precond_{i+1}}
        + \frac{1}{2}\norm{\nextu-\realoptu}_{\Test_{i+1}(\GammaLift{i+1}(0)+\Precond_{i+1})-\Test_{i+2}\Precond_{i+2}}^2
        \\
        +\iprod{\subdiff [G+J](\nextx)-\subdiff [G+J](\realoptx)}{\nextx - \realoptx}_{\TauTest_{i}\Tau_{i}}
        +
        \iprod{\subdiff F^*(\nexty)-\subdiff F^*(\realopty)}{\nexty - \realopty}_{\SigmaTest_{i+1}\Sigma_{i+1}}
        \\
        + \iprod{\AltBregFn_{i+1}(\nextu)}{\nextu-\realoptu}_{\Test_{i+1}}
        \ge
        - \Penalty_{i+1}(\realoptu).
    \end{multline}
    Here pay attention to the fact that \eqref{eq:convergence-fundamental-condition-iter0} employs $\GammaLift{i+1}(0)$ while \eqref{eq:convergence-fundamental-condition-iter} employs $\GammaLift{i+1}(\tilde\Gamma)$.
    If we show that \eqref{eq:convergence-fundamental-condition-iter-h-alt} follows from \eqref{eq:convergence-fundamental-condition-iter0}, then the descent inequality \eqref{eq:convergence-result-main-h} follows from \cref{cor:convergence-result-main-h-alt}.
    Indeed, using the expansion
    \begin{equation*}
        \Test_{i+1} \Step_{i+1}
        =
        \begin{pmatrix}
            \TauTest_i \Tau_i & 0 
            \\
            0 & \SigmaTest_{i+1}\Sigma_{i+1}
        \end{pmatrix},
    \end{equation*}
    we expand for any $\optu=(\optx, \opty)$ that
    \begin{equation*}
        \begin{split}
            \iprod{\Test_{i+1} & \Step_{i+1} (H(\nextu) - H(\optu))}{\nextu - \optu} 
            \\
            &
            =\iprod{\subdiff G(\nextx)-\subdiff G(\optx)}{\nextx - \optx}_{\TauTest_i\Tau_i}
             + \iprod{\subdiff F^*(\nexty)-\subdiff F^*(\opty)}{\nexty - \opty}_{\SigmaTest_{i+1}\Sigma_{i+1}}
            \\ & \phantom{ = }
            + \iprod{\TauTest_i\Tau_iK^*(\nexty-\opty)}{\nextx-\optx}
            - \iprod{\SigmaTest_{i+1}\Sigma_{i+1}K(\nextx-\optx)}{\nexty-\opty}.
        \end{split}
    \end{equation*}
    With the help of $\GammaLift{i+1}(0)$ we then obtain
    \begin{multline}%
        \notag%
        \iprod{H(\nextu) - H(\optu)}{\nextu - \optu}_{\Test_{i+1} \Step_{i+1}}
        \ge
        \frac{1}{2}\norm{\nextu-\optu}_{\Test_{i+1}\GammaLift{i+1}(0)}
        \\
        + \iprod{\subdiff G(\nextx)-\subdiff G(\optx)}{\nextx - \optx}_{\TauTest_i\Tau_i}
        + \iprod{\subdiff F^*(\nexty)-\subdiff F^*(\opty)}{\nexty - \opty}_{\SigmaTest_{i+1}\Sigma_{i+1}}.
    \end{multline}%
    Inserting this into 
    \eqref{eq:convergence-fundamental-condition-iter0}, we obtain the fundamental inequality \eqref{eq:convergence-fundamental-condition-iter-h-alt}. It implies \eqref{eq:convergence-fundamental-condition-iter-h} via \cref{cor:convergence-result-main-h-alt}.
    Finally, \cref{thm:convergence-result-main-h} gives \eqref{eq:convergence-result-main-h}.
\end{proof}

\subsection{Examples of primal--dual methods}
\label{sec:pd-example}

\cbstart
We now look at several known methods for the saddle point problem \eqref{eq:saddle}.
The fundamental idea in all of them is to design $\Precond_{i+1}$ such that the primal variable $\nexty$ and the dual variable $\nexty$ can be updated independently unlike in the standard proximal point method with $\Precond_{i+1}=I$.
To help verifying the condition \cref{thm:convergence-result-saddle} for these methods, we reformulate the result for scalar step length and testing parameters: we will only use the full power of the operator setup in our companion paper \cite{tuomov-blockcp}.

If for each $i \in \N$, we pick $\tau_i,\tauTest_i,\sigma_{i+1},\sigmaTest_{i+1} > 0$ and $\gamma \ge 0$, and define $\Tau_i=\tau_i I$, $\TauTest_i=\tauTest_i I$, $\Sigma_{i+1}=\sigma_{i+1} I, \SigmaTest_{i+1}=\sigmaTest_{i+1} I$, and $\Gamma \defeq \gamma I$, then \eqref{eq:test}, \eqref{eq:gammalift-q}, and \eqref{eq:v-j-scalar} reduce to
\begin{subequations}
\label{eq:pd-setup-scalar}
\begin{align}
    \Step_{i+1} & \defeq
    \begin{pmatrix}
        \tau_i I & 0 \\
        0 & \sigma_{i+1} I
    \end{pmatrix},
    &
    \Test_{i+1} & \defeq
        \begin{pmatrix}
            \tauTest_i I & 0 \\
            0 & \SigmaTest_{i+1} I
        \end{pmatrix}.
    \\
    \GammaLift{i+1}(\tilde\Gamma) & \defeq \begin{pmatrix}
        2\tau_i\tilde\gamma  & 2\tau_i K^* \\
        -2\sigma_{i+1}K & 0 
    \end{pmatrix},
    &
    Q_{i+1}(L) &\defeq \begin{pmatrix} L\tau_i I & 0 \\ 0 & 0 \end{pmatrix},
    \quad\text{and}
    \\
    \label{eq:v-j-scalar}
    \BregFn_{i+1}^J(u) & \defeq \begin{pmatrix}
        \tau_i(\grad J(\thisx)-\grad J(x)) \\
        0
        \end{pmatrix}.
\end{align}
\end{subequations}
Then we have the following corollary of \cref{thm:convergence-result-saddle}.

\begin{corollary}
    \label{cor:convergence-result-saddle-nogap}
    Let $H$ have the structure \eqref{eq:h} and assume $\realoptu \in \inv H(0)$.
    Assume that $G$ is ($\gamma$-strongly) convex and $\grad J$ is $L$-Lipschitz for some $\gamma \ge 0$ and $L>0$.
    For each $i \in \N$, assume the structure \eqref{eq:pd-setup-scalar} for $\tau_i,\tauTest_i,\sigma_{i+1},\sigmaTest_{i+1} > 0$.
    Also take $\AltBregFn_{i+1} \in X \times Y \to X \times Y$ and $\Precond_{i+1} \in \linear(X \times Y; X \times Y)$.
    Suppose \eqref{eq:pp} is solvable for $\{\nextu\}_{i \in \N} \subset X \times Y$.
    Suppose for all $i \in \N$ that $\Test_{i+1}\Precond_{i+1}$ is self-adjoint, and that the fundamental condition for saddle-point problems \eqref{eq:convergence-fundamental-condition-iter} holds for $\tilde\Gamma=\gamma I$ and $L_i \equiv L/2$.
    Then the fundamental conditions \eqref{eq:convergence-fundamental-condition-iter-h-alt}, \eqref{eq:convergence-fundamental-condition-iter-h} and the descent inequality \eqref{eq:convergence-result-main-h} hold.
\end{corollary}

\begin{proof}
    Clearly $\TauTest_i\Tau_i \in \TauMonotoneSpace \defeq [0, \infty)I$ and $\SigmaTest_{i+1}\Sigma_{i+1} \in \SigmaMonotoneSpace \defeq [0, \infty)I$.
    Moreover, $F^*$ satisfies the partial monotonicity condition \eqref{eq:f-monotone} and $G$ satisfies the partial partial monotonicity condition \eqref{eq:g-strong-monotone} with $\Gamma=\gamma I$ by the corresponding (strong) monotonicity of the subdifferentials.
    The rest follows from \cref{thm:convergence-result-saddle}.
\end{proof}

\cbend

\begin{example}[The primal--dual method of Chambolle and Pock \cite{chambolle2010first}]
	\label{example:cpock}
	\begin{subequations}%
	\label{eq:cp}%
	With $J=0$, this method consists of iterating the system
	\begin{align}
	    \nextx & \defeq (I+\tau_i \subdiff G)^{-1}(\thisx - \tau_i K^* y^{i}),\\
	    \label{eq:cp-overnextx}
	    \overnextx & \defeq \omega_i (\nextx-\thisx)+\nextx, \\
	    \nexty & \defeq (I+\sigma_{i+1} \subdiff F^*)^{-1}(\thisy + \sigma_{i+1} K \overnextx).
	\end{align} 
	\end{subequations}
	In the basic version of the algorithm, $\omega_i=1$, $\tau_i \equiv \tau_0 > 0$, and $\sigma_i \equiv \sigma_0 > 0$, assuming the step length parameters to satisfy
    \begin{equation}
        \label{eq:cpock-init-cond}
        \tau_0 \sigma_0 \norm{K}^2 < 1.
    \end{equation}
    If $K$ is compact, the iterates convergence weakly, and the method has $O(1/N)$ rate for the ergodic duality gap, to which we will return in \cref{sec:gap}. If $G$ is strongly convex with factor $\gamma>0$, we may accelerate
	\begin{equation}
	    \label{eq:cpaccel}
	    \omega_i \defeq 1/\sqrt{1+2\gamma\tau_i},
	    \quad
	    \tau_{i+1} \defeq \tau_i\omega_i,
	    \quad\text{and}\quad
	    \sigma_{i+1} \defeq \sigma_i/\omega_i.
	\end{equation}
	This yields $O(1/N^2)$ convergence of $\norm{x^N-\realoptx}^2$ to zero.
\end{example}

\begin{demonstration}[Proof of convergence of iterates]
    We formulate the method in our proximal point framework with $J=0$ and $G=G$ following \cite{tuomov-cpaccel,he2012convergence} by taking as the preconditioner
	\[
		\Precond_{i+1}=\begin{pmatrix} I & -\tau_i K^* \\ -\sigma_i K & I\end{pmatrix}
		\quad\text{and}\quad
		\AltBregFn_{i+1} = 0.
	\]
    For the rest of the operators, we use the setup of \eqref{eq:pd-setup-scalar}.
	Taking $\Penalty_{i+1}(\realoptu) \defeq -\frac{1}{2}\norm{\nextu-\thisu}_{\Test_{i+1}\Precond_{i+1}}^2$, we now reduce \eqref{eq:convergence-fundamental-condition-iter} to
	\begin{equation}
        \label{eq:cpock-d}
	    \frac{1}{2}\norm{\nextu-\realoptu}_{D_{i+2}}^2
		\ge
		0
	\quad\text{for}\quad
		D_{i+2} \defeq \Test_{i+1}(\GammaLift{i+1}(\gamma I)+\Precond_{i+1})-\Test_{i+2}\Precond_{i+2}.
	\end{equation}

	We may expand
    \begin{subequations}%
    \label{eq:cpock-zimi-d-expansion}%
	\begin{align}%
        \label{eq:cpock-zimi}
		\Test_{i+1}\Precond_{i+1} & =
		\begin{pmatrix} 
			\tauTest_i I & -\tauTest_i\tau_i K^* \\
			-\sigmaTest_{i+1}\sigma_{i} K & \sigmaTest_{i+1} I
		\end{pmatrix},
		\quad\text{and}\\
        \label{eq:cpock-d-expansion}
		D_{i+2} & =
		\begin{pmatrix} 
			(\tauTest_i(1+2\gamma\tau_i) - \tauTest_{i+1}) I & (\tauTest_i\tau_i + \tauTest_{i+1}\tau_{i+1})K^* \\
			(\sigmaTest_{i+2}\sigma_{i+1}-2\sigmaTest_{i+1}\sigma_{i+1}-\sigmaTest_{i+1}\sigma_i) K & (\sigmaTest_{i+1} -\sigmaTest_{i+2})I
		\end{pmatrix}.
	\end{align}%
    \end{subequations}%
	We have $\norm{\freevar}_{D_{i+2}}=0$ (but not $D_{i+2}=0$, as the former depends on the off-diagonals cancelling out), and $\Test_{i+1}\Precond_{i+1}$ is self-adjoint, if for some constant $\sigmaTest$ we take
	\begin{equation}
		\label{eq:cpock-par}
		\tauTest_{i+1} \defeq \tauTest_i(1+2\gamma\tau_i),
        \quad
        \tau_i \defeq \tauTest_i^{-1/2},
		\quad
		\sigma_{i} \defeq \tauTest_i\tau_i/\sigmaTest,
		\quad
		\text{and}
		\quad
		\sigmaTest_{i+1} \defeq \sigmaTest.
	\end{equation}
    This gives the acceleration scheme \eqref{eq:cpaccel}. 
	Moreover, for any $\delta \in (0, 1)$ holds
	\begin{equation}
		\label{eq:cpock-zimi-estimate}
		\Test_{i+1}\Precond_{i+1} \ge
		\begin{pmatrix}
			\delta\tauTest_i I & 0 \\
			0 & \sigmaTest I- \inv{(1-\delta)}\tauTest_i\tau_i^2 KK^*
		\end{pmatrix}.
	\end{equation}
	Thus $\Test_{i+1}\Precond_{i+1} \ge 0$ if $\sigmaTest \ge \inv{(1-\delta)}\tauTest_i\tau_i^2\norm{K}^2$. By \eqref{eq:cpock-par}, $\sigma_i\tau_i=1/\sigmaTest$. Since this fixes the ratio of $\sigma_i$ to $\tau_i$, we need to take $\sigmaTest \defeq 1/(\sigma_0\tau_0)$ as well as $\delta \defeq 1-\sigma_0\tau_0\norm{K}^2$. Through the positivity of $\delta$, we recover the initialisation condition \eqref{eq:cpock-init-cond}.

    \cbstart
    Recall that subdifferentials are weak-to-strong outer-semicontinuous.
    By the continuity of $K$, we thus deduce the strong-to-strong outer semicontinuity of $H$. To verify \eqref{eq:precond-continuity}, we use the assumed compactness of $K$, which implies for a further unrelabelled subsequence of $\{u^{i_k}\}_{k \in \N}$ that $w^{i_k} \in H(u^{i_k})$ satisfy  $0=\lim_{k \to \infty} w^{i_k} \in H(\optu)$.
	\Cref{cor:convergence-result-saddle-nogap,prop:rateless} now shows weak convergence of the iterates without a rate.
    \cbend

    If $G$ is strongly convex with factor $\gamma \ge 0$, the results in \cite{chambolle2010first,tuomov-cpaccel} show that $\tau_N$ is of the order $O(1/N)$, and consequently $\tauTest_N$ is of the order $\Theta(N^2)$.
    By \cref{prop:rate}, $\norm{x^N-\realoptx}^2$ converges to zero at the rate $O(1/N^2)$.
\end{demonstration}

\cbstart
\begin{remark}[Brezis--Crandall--Pazy property]
    It is possible to show that $H$ satisfies the Brezis--Crandall--Pazy property \cite{bresiz1970perturbations} without a compactness assumption on $K$.
    With a corresponding improvement to \cref{prop:rate}, the assumption could be dropped.
\end{remark}

\begin{remark}[Linear convergence]
    If $F^*$ is strongly convex with factor $\rho > 0$, the last equation of \eqref{eq:cpock-par} gets similar form as the first, $\sigmaTest_{i+1} \defeq \sigmaTest_i(1+2\rho\sigma_i)$. From here, if both $G$ and $F^*$ are strongly convex, it is possible to show linear convergence.
\end{remark}

We can also add an additional forward step to the method.
With that the method resembles the method of V\~u--Condat \cite{condat2013primaldual,vu2013splitting}, which also incorporates an additional outer over-relaxation step on the whole algorithm.
\cbend

\begin{example}[Chambolle--Pock with a forward step]
	\label{example:cpock-forward}
	Suppose $G$ is (strongly) convex with factor $\gamma \ge 0$, and $\grad J$ Lipschitz with factor $L$.
	In \cite{chambolle2014ergodic}, the Chambolle--Pock method was extended to take forward steps with respect to $J$.
    With everything else as in \cref{example:cpock}, take $\AltBregFn_{i+1}(u) \defeq (\tau_i(\grad J(\thisx)-\grad J(x)), 0)$.    
 	Then the preconditioned proximal point method \eqref{eq:pp} can be rearranged as
    \begin{align}
        \label{eq:np-stage2-x}
        \nextx & \defeq (I+\tau_i \subdiff G)^{-1}(\thisx - \tau_i \grad J(\thisx) - \tau_i K^* y^{i}),\\
        \label{eq:np-stage2-overnextx}
        \overnextx & \defeq \omega_i (\nextx-\thisx)+\nextx, \\
        \label{eq:np-stage2-y}
        \nexty & \defeq (I+\sigma_{i+1} \subdiff F^*)^{-1}(\thisy + \sigma_{i+1} K \overnextx).
    \end{align}
    The method inherits the convergences properties of \cref{example:cpock} if we use the step length update rules \eqref{eq:cpaccel}, and initialise $\tau_0,\sigma_0>0$ subject to \eqref{eq:cpock-init-cond}, and
    \begin{equation}
        \label{eq:cpock-init-j}
        0 < \theta \defeq 1-L\tau_0/(1-\tau_0\sigma_0\norm{K}^2).
    \end{equation}
\end{example}

\begin{demonstration}    
    With $D_{i+2}$ as in \eqref{eq:cpock-d}, the fundamental condition for saddle-point problems \eqref{eq:convergence-fundamental-condition-iter} becomes
	\begin{equation}
		\label{eq:cpock-j-est2}		
		\frac{1}{2}\norm{\nextu-\thisu}_{\Test_{i+1}\Precond_{i+1}}^2
        -\frac{\tau_i\tauTest_i L}{4}\norm{\nextx-\thisx}^2
	    + \frac{1}{2}\norm{\nextu-\realoptu}_{D_{i+2}}^2
		\ge
		- \Penalty_{i+1}(\realoptu).
	\end{equation}
    The rules \eqref{eq:cpock-par} force $\norm{\freevar}_{D_{i+2}}=0$. 
	We take $\Penalty_{i+1}(\realoptu)=-\frac{\theta}{2}\norm{\nextu-\thisu}_{\Test_{i+1}\Precond_{i+1}}^2$ for some $\theta > 0$, and deduce using Cauchy's inequality that \eqref{eq:cpock-j-est2} holds if
	\[
		(1-\theta)\Test_{i+1}\Precond_{i+1} \ge \tau_i\tauTest_i L \begin{pmatrix} I & 0 \\ 0 & 0 \end{pmatrix}.
	\]
    Recalling \eqref{eq:cpock-zimi-estimate}, this is true if $(1-\theta)\delta\tauTest_i \ge \tau_i\tauTest_i L$ and $\sigmaTest \ge \inv{(1-\delta)}\tauTest_i\tau_i^2\norm{K}^2$.
    Further recalling \eqref{eq:cpock-par}, and observing that $\{\tau_i\}$ is non-increasing, we only have to satisfy $(1-\theta)(1-\tau_0\sigma_0\norm{K}^2) \ge L \tau_0$. Otherwise put, we obtain \eqref{eq:cpock-init-j}.
\end{demonstration}

\cbstart
Finally, we have the following Generalised Iterative Soft Thresholding (GIST) method from \cite{loris2011generalization}.
\cbend

\begin{example}[GIST]
    \label{example:gist}
	Suppose $G=0$, $J(x)=\frac{1}{2}\norm{f-Ax}^2$, $\norm{A} < \sqrt{2}$, and $\norm{K} \le 1$. Take
	\[
		\AltBregFn_{i+1}(u) \defeq \begin{pmatrix} \grad J(\thisx)-\grad J(x) \\ 0
		\end{pmatrix},
		\quad\text{and}\quad
		\Precond_{i+1} \defeq \begin{pmatrix} I & 0 \\ 0 & I-KK^*\end{pmatrix}.
	\]
	With $\Tau_{i} \defeq I$ and $\Sigma_{i+1} \defeq I$, we obtain the method
	\begin{align*}
		\nexty & \defeq \inv{(I+\subdiff F^*)}((I-KK^*)\thisy+K(\thisx-\grad G(\thisx))), \\
		\nextx & \defeq \thisx - \grad G(\thisx) - K^*\nexty.
	\end{align*}
    If $K$ is compact, the iterates $\{\thisx\}_{i \in \N}$ converge weakly to $\realoptx$.
\end{example}

\begin{demonstration}
    Observe that the partial co-coercivity \eqref{eq:j-coco} holds with $L=\norm{A}^2$. 
    Clearly $\Test_{i+1}\Precond_{i+1}$ is positive semi-definite self-adjoint.
    If we take $\TauTest_i=I$ and $\SigmaTest_{i+1}=I$, then
    \[
        D_{i+2} \defeq \Test_{i+1}(\GammaLift{i+1}(0)+\Precond_{i+1})-\Test_{i+2}\Precond_{i+2}
        =
        \begin{pmatrix}
            0 & 2 K^* \\
            -2K & 0
        \end{pmatrix}.
    \]
    Thus $\frac{1}{2}\norm{u}_{D_{i+2}}^2=0$.
    Eliminating $\subdiff F^*$ by monotonicity, the fundamental condition for saddle-point problems \eqref{eq:convergence-fundamental-condition-iter} thus holds if
    \begin{equation*}
        \frac{1}{2}\norm{\nextu-\thisu}_{\Test_{i+1}\Precond_{i+1}}^2
        -\frac{L}{4}\norm{\nextx-\realoptx}^2
        \ge
        - \Penalty_{i+1}(\realoptu).
    \end{equation*}
    Expanding $\Test_{i+1}\Precond_{i+1}$, we see this to hold when $\norm{K} < 1$ and $L<2$, which are exactly our assumptions.
    Using \cref{cor:convergence-result-saddle-nogap,prop:rate}, and reasoning as in \cref{example:cpock} to verify the outer-semicontinuity properties of $H$, we obtain weak convergence.
\end{demonstration}

\section{An ergodic duality gap}
\label{sec:gap}

We now study the extension of the testing approach of \cref{sec:convergence-basic} to produce the convergence of an ergodic duality gap.
Throughout this section, we are in the saddle point setup of \cref{sec:saddle}. In particular, the operator $H$ is as in \eqref{eq:h}, and the step length and testing operators $\Step_{i+1}$ and $\Test_{i+1}$ as in \eqref{eq:test}.

\subsection{Preliminary gap estimates}
\label{sec:gap-preliminary}

Our first lemma demonstrates how to obtain a ``preliminary'' gap $\gap'_{i+1}(u)$ from $H$. 
If the step lengths and tests are scalar, $\Tau_i=\tau_i I$, and $\TauTest_i=\tauTest_i I$, etc., and satisfy $\tau_i\tauTest_i=\sigma_i\sigmaTest_{i+1}$, it is easy to bound this preliminary gap from below by $\tau_i\tauTest_i$ times the \cbstart ``relaxed'' duality gap
\begin{equation}
    \label{eq:gap}
    \gap(x, y) \defeq
    \bigl([G+J](x) + \iprod{\realopty}{K x}  - F(\realopty)\bigr)
    -\bigl([G+J](\realoptx) + \iprod{y}{K \realoptx} -  F^*(y)\bigr).
\end{equation}\cbend
To do the same for more general step length operators, we will in \cref{sec:gap-conversion} introduce abstract notions of convexity that incorporate ergodicity and stochasticity.

\cbstart
Observe that the ``relaxed'' gap \eqref{eq:gap} satisfies
\[
    0 \le \gap(x, y) \le [G+J](x)+F(Kx) + [G+J]^*(-K\realopty) + F^*(\realopty),
\]
where the right-hand side is the conventional duality gap guaranteed to be non-zero for a non-solution $x$.
\cbend

\begin{lemma}
    \label{lemma:h-gap}
    \cbstart For a fixed $i \in \N$, \cbend
    suppose $\TauTest_i\Tau_i$ and $\SigmaTest_{i+1}\Sigma_{i+1}$ are self-adjoint.
    Then for $H$ as in \eqref{eq:h}, we have
    \begin{equation}
        \label{eq:h-gap}
        \iprod{H(\nextu)}{\nextu - \realoptu}_{\Test_{i+1} \Step_{i+1}}
        =
        \gap'_{i+1}(\nextu)
        +\frac{1}{2}\norm{\nextu-\realoptu}_{\Test_{i+1}\GammaLift{i+1}(0)},
    \end{equation}
    where the ``preliminary gap''
    \begin{equation}
        \label{eq:preliminary-gap}
        \begin{split}
        \gap'_{i+1}(u)
        &
        \defeq
        \iprod{\subdiff [G+J](x)}{x - \realoptx}_{\TauTest_i\Tau_i}
        + \iprod{\subdiff F^*(y)}{y - \realopty}_{\SigmaTest_{i+1}\Sigma_{i+1}}
        \\
        &                 
        \phantom{ \defeq }
        - \iprod{\realopty}{(K\Tau_i^*\TauTest_i^* - \SigmaTest_{i+1}\Sigma_{i+1}K)\realoptx}
        - \iprod{y}{\SigmaTest_{i+1}\Sigma_{i+1}K \realoptx}
        + \iprod{\realopty}{K \Tau_i^*\TauTest_i^*x}.
        \end{split}
    \end{equation}
\end{lemma}
\begin{proof}
    Similarly to the proof of \cref{thm:convergence-result-saddle}, we have
    \begin{equation*}
        \begin{split}
            \iprod{H(\nextu)}{\nextu - \realoptu}_{\Test_{i+1} \Step_{i+1}}
            &
            = \iprod{\subdiff[G+J](\nextx)}{\nextx - \realoptx}_{\TauTest_i\Tau_i}
            + \iprod{\TauTest_i\Tau_iK^*\nexty}{\nextx-\realoptx}
            \\ & \phantom{ = }
            + \iprod{\subdiff F^*(\nexty)}{\nexty - \realopty}_{\SigmaTest_{i+1}\Sigma_{i+1}}
            - \iprod{\SigmaTest_{i+1}\Sigma_{i+1}K \nextx}{\nexty-\realopty}.
        \end{split}
    \end{equation*}
    A little bit of reorganisation gives \eqref{eq:h-gap}. Indeed
    \begin{equation*}
        \begin{split}
            \iprod{H(\nextu)}{\nextu - \realoptu}_{\Test_{i+1} \Step_{i+1}}
            &
            =
            \iprod{\subdiff[G+J](\nextx)}{\nextx - \realoptx}_{\TauTest_i\Tau_i}
            + \iprod{\subdiff F^*(\nexty)}{\nexty - \realopty}_{\SigmaTest_{i+1}\Sigma_{i+1}}
            \\ & \phantom{ = }            
            + \iprod{\nexty-\realopty}{(K\Tau_i^*\TauTest_i^* - \SigmaTest_{i+1}\Sigma_{i+1}K)(\nextx-\realoptx)}
            \\ & \phantom{ = }
            - \iprod{\realopty}{(K\Tau_i^*\TauTest_i^* - \SigmaTest_{i+1}\Sigma_{i+1}K)\realoptx}      
            \\ & \phantom{ = }
            - \iprod{\nexty}{\SigmaTest_{i+1}\Sigma_{i+1}K \realoptx}
            + \iprod{\realopty}{K \Tau_i^*\TauTest_i^* \nextx}
            \\
            &
            =
            \gap'_{i+1}(\nextu)+\frac{1}{2}\norm{\nextu-\realoptu}_{\Test_{i+1}\GammaLift{i+1}(0)}.
            \qedhere
        \end{split}
    \end{equation*}
\end{proof}

The next lemma extends \cref{thm:convergence-result-saddle} to estimate the preliminary gap.

\begin{lemma}
    \label{lemma:convergence-result-gap}
    Let $H$ have the structure \eqref{eq:h} and assume $\realoptu \in \inv H(0)$.
    For each $i \in \N$, let $\Tau_i, \TauTest_i \in \L(X; X)$ and $\Sigma_{i+1}, \SigmaTest_{i+1} \in \L(Y; Y)$, as well as $\AltBregFn_{i+1} \in X \times Y \to X \times Y$ and $\Precond_{i+1} \in \linear(X \times Y; X \times Y)$.
    Define $\Test_{i+1}$ and $\Step_{i+1}$ through \eqref{eq:test}.
    Suppose \eqref{eq:pp} is solvable for $\{\nextu\}_{i \in \N} \subset X \times Y$.
    If \cbstart for all $i \in \N$, \cbend $\Test_{i+1}\Precond_{i+1}$ is self-adjoint, and
    \begin{multline}
        \label{eq:convergence-fundamental-condition-iter-gap}
        \frac{1}{2}\norm{\nextu-\thisu}_{\Test_{i+1}\Precond_{i+1}}^2
        + \frac{1}{2}\norm{\nextu-\realoptu}_{\Test_{i+1}(\GammaLift{i+1}(0)+\Precond_{i+1})-\Test_{i+2}\Precond_{i+2}}^2
        + \iprod{\AltBregFn_{i+1}(\nextu)}{\nextu - \realoptu}_{\Test_{i+1}}
        \\
        \ge
        - \tilde\Penalty_{i+1}(\realoptu),
    \end{multline}
    then
    \begin{equation}
        \label{eq:convergence-result-gap}
        \frac{1}{2}\norm{u^N-\realoptu}^2_{\Test_{N+1}\Precond_{N+1}}
        + \sum_{i=0}^{N-1} \gap'_{i+1}(\nextu)
        \le
        \frac{1}{2}\norm{u^0-\realoptu}^2_{\Test_{1}\Precond_{1}}
        +
        \sum_{i=0}^{N-1} \tilde\Penalty_{i+1}(\realoptu)
        \quad
        (N \ge 1).
    \end{equation}
\end{lemma}

\begin{proof}
    Inserting \eqref{eq:h-gap} from \cref{lemma:h-gap} into \eqref{eq:convergence-fundamental-condition-iter-gap} shows that
    \begin{multline*}
        \frac{1}{2}\norm{\nextu-\thisu}_{\Test_{i+1}\Precond_{i+1}}^2
        + \frac{1}{2}\norm{\nextu-\realoptu}_{\Test_{i+1}(\Precond_{i+1}+\GammaLift{i+1}(0))-\Test_{i+2}\Precond_{i+2}}^2
        \\
        + \iprod{\Step_{i+1} H(\nextu)+\AltBregFn_{i+1}(\nextu)}{\nextu - \realoptu}_{\Test_{i+1}}
        \ge
        \gap'_{i+1}(\nextu) - \tilde\Penalty_{i+1}(\realoptu).
    \end{multline*}
    Hence the fundamental condition \eqref{eq:convergence-fundamental-condition-iter-h} holds for $\Penalty_{i+1}(\realoptu) \defeq \tilde\Penalty_{i+1}(\realoptu) - \gap'_{i+1}(\nextu)$.
    Now we use \cref{thm:convergence-result-main-h}.
\end{proof}

\subsection{General conversion formulas of preliminary gaps to ergodic gaps}
\label{sec:gap-conversion1}

The ``preliminary gaps'' are not as such very useful. To go further, the abstract partial monotonicity assumptions \eqref{eq:g-strong-monotone} and \eqref{eq:f-monotone} are not enough, and we need analogous convexity formulations. 
We formulate these conditions directly in the stochastic setting (recall \cref{sec:stochastic-examples}).

For the moment, we assume for all $N \ge 1$ that whenever $\TauTilde_{i}\, (\defeq\TauTest_i\Tau_i) \in \Random(\TauMonotoneSpace)$ and $\nextx \in \Random(X)$ for each $i=0,\ldots,N-1$ with $\sum_{i=0}^{N-1} \E[\TauTilde_i]=I$, then for some $\nexxt{\delta_G} \in \Random(\R)$ holds
\begin{equation}
    \label{eq:g-convex-abstract}
    [G+J](\realoptx)-
    [G+J]\Biggl(\sum_{i=0}^{N-1} \E[\TauTilde_{i}^*\nextx]\Biggr)
    \ge
    \sum_{i=0}^{N-1} \E\bigl[\iprod{\subdiff  [G+J](\nextx)}{\realoptx-\nextx}_{\TauTilde_{i}}
     + \nexxt{\delta_{G+J,N}} \bigr].
\end{equation}
Analogously, we assume for $\SigmaTilde_{i+1}\, (\defeq\SigmaTest_{i+1}\Sigma_{i+1}) \in \Random(\SigmaMonotoneSpace)$ and $\nexty \in \Random(Y)$ for each $i=0,\ldots,N-1$ with $\sum_{i=0}^{N-1} \E[\SigmaTilde_{i+1}]=I$ that for some $\nexxt{\delta_{F^*}} \in \Random(\R)$ holds
\begin{equation}
    \label{eq:f-convex-abstract}
    F^*(\realopty)-F^*\Biggl(\sum_{i=0}^{N-1} \E[\SigmaTilde_{i+1}^*\nexty]\Biggr)
    \ge
    \sum_{i=0}^{N-1} \E\bigl[\iprod{\subdiff F^*(\nexty)}{\realopty-\nexty}_{\SigmaTilde_{i+1}}
     + \nexxt{\delta_{F^*,N}} \bigr].
\end{equation}
These conditions can of course always be satisfied for some $\nexxt{\delta_G}$ and $\nexxt{\delta_{F^*}}$. After a few general lemmas, we will replace these placeholder values by more meaningful ones. 

To state those lemmas, we also assume \cbstart for some scalars $\bar\eta_i \in \R$, ($i \in \N$), \cbend either of the \term{primal--dual coupling conditions}
\begin{align}
    \label{eq:cond-eta}
    \tag{C$\mathcal{G}$}
    \E[\TauTest_i\Tau_i] & = \bar\eta_i I, 
    \quad\text{and}&
    \E[\SigmaTest_{i+1}\Sigma_{i+1}] & = \bar\eta_i I,
    \quad(i \ge 1),
    \\
\intertext{or}
    \label{eq:cond-etatwo}
    \tag{C$\mathcal{G}_*$}
    \E[\TauTest_i\Tau_i] & = \bar\eta_i I,
    \quad\text{and} &
    \E[\SigmaTest_{i}\Sigma_{i}] & = \bar\eta_i I,
    \quad(i \ge 1),
\end{align}
As will see in \cref{example:cpock-gap}, \eqref{eq:cond-etatwo} is satisfied by the accelerated Chambolle--Pock method of \cref{example:cpock}.
In our companion paper \cite{tuomov-blockcp}, we will however see that \eqref{eq:cond-eta} is required to develop doubly-stochastic methods.

\begin{lemma}
    \label{lemma:gap-transform}
    Assume \eqref{eq:g-convex-abstract}, \eqref{eq:f-convex-abstract}, and the first primal--dual coupling condition \eqref{eq:cond-eta}.
    Given iterates $\{(\thisx,\thisy)\}_{i=1}^\infty \subset X \times Y$, for all $N \ge 1$ set
    \begin{equation*}
        \zeta_N \defeq \sum_{i=0}^{N-1} \bar\eta_i,
    \end{equation*}
    and define the ergodic sequences
    \begin{equation}
        \label{eq:tildexnyn}
        \tilde x_{N} \defeq \inv\zeta_{N}\E\Biggl[\sum_{i=0}^{N-1} \Tau_i^*\TauTest_i^* \nextx\Biggr],    
        \quad\text{and}\quad
        \tilde y_{N} \defeq \inv\zeta_{N}\E\Biggl[\sum_{i=0}^{N-1} \Sigma_{i+1}^*\SigmaTest_{i+1}^* \nexty\Biggr].
    \end{equation}
    Then
    \[
        \sum_{i=0}^{N-1} \E[\gap'_{i+1}(\nextx,\nexty) + \zeta_N\nexxt{\delta_{G+J,N}} + \zeta_N\nexxt{\delta_{F^*,N}}] 
        \ge \zeta_N \gap(\tilde x_N, \tilde y_N)
        \cbstart \quad (N \ge 1). \cbend
    \]
\end{lemma}

\begin{proof}
    Let $N$ be fixed.
    With $\TauTilde_i \defeq \inv\zeta_N \TauTest_i\Tau_i$ over $i=0,\ldots,N-1$, \eqref{eq:g-convex-abstract} implies
    \begin{equation}
        \label{eq:g-convex-abstract-ergvar}
        \zeta_N\bigl(
            [G+J](\realoptx)-[G+J](\tilde x_{N})
        \bigr)
        \ge
        \sum_{i=0}^{N-1} \E\bigl[\iprod{\subdiff  [G+J](\nextx)}{\realoptx-\nextx}_{\TauTest_i\Tau_i}
        \bigr] + \zeta_N\nexxt{\delta_{G+J,N}}.
    \end{equation}
    Likewise, with $\SigmaTilde_{i+1} \defeq \inv\zeta_N \SigmaTest_{i+1}\Sigma_{i+1}$, \eqref{eq:f-convex-abstract} shows that
    \begin{equation}
        \label{eq:f-convex-abstract-ergvar}
        \zeta_N\bigl(
            F^*(\realopty)-F^*(\tilde y_{N})
        \bigr)
        \ge
        \sum_{i=0}^{N-1} \E\bigl[\iprod{\subdiff F^*(\nexty)}{\realopty-\nexty}_{\SigmaTest_{i+1}\Sigma_{i+1}}
        \bigr] + \zeta_N\nexxt{\delta_{F^*,N}}.
    \end{equation}
    From the definition of the preliminary gap in \eqref{eq:preliminary-gap}, applying \eqref{eq:cond-eta}, we obtain
    \begin{equation*}        
        \begin{split}
        \sum_{i=0}^{N-1} \E[\gap'_{i+1}(\nextu)]
        &
        =
        \sum_{i=0}^{N-1} \E[\iprod{\subdiff [G+J](\nextx)}{\nextx - \realoptx}_{\TauTest_i\Tau_i}
        + \iprod{\subdiff F^*(\nexty)}{\nexty - \realopty}_{\SigmaTest_{i+1}\Sigma_{i+1}}]
        \\
        &
        \phantom{ \defeq }
        - \sum_{i=0}^{N-1} \E[\iprod{\nexty}{\SigmaTest_{i+1}\Sigma_{i+1}K \realoptx}
        + \iprod{\realopty}{K \Tau_i^*\TauTest_i^*\nextx}].
        \end{split}
    \end{equation*}
    Recalling the definition of the gap $\gap$ in \eqref{eq:gap}, and using the estimates \eqref{eq:g-convex-abstract-ergvar}, \eqref{eq:f-convex-abstract-ergvar}, as well as the definition  \eqref{eq:tildexnyn} of the ergodic sequences, we obtain the claim.
\end{proof}

\begin{lemma}
    \label{lemma:gap-transformtwo}
    \cbstart
    Suppose $G$ and $F^*$ satisfy with $\Gamma=0$ the corresponding partial monotonicities \eqref{eq:g-strong-monotone} and \eqref{eq:f-monotone}.
    \cbend
    Also assume \eqref{eq:g-convex-abstract}, \eqref{eq:f-convex-abstract}, and the second primal--dual coupling condition \eqref{eq:cond-etatwo}.
    Given $\{(\thisx,\thisy)\}_{i=1}^\infty \subset X \times Y$, for all $N \ge 1$ set
    \begin{equation*}
        \zeta_{*,N} \defeq \sum_{i=1}^{N-1} \bar\eta_i,
    \end{equation*}
    and define the ergodic sequences
    \begin{equation*}
        \tilde x_{*,N} \defeq \inv\zeta_{*,N}\E\Biggl[\sum_{i=1}^{N-1} \Tau_i^*\TauTest_i^* \nextx\Biggr],
        \quad\text{and}\quad
        \tilde y_{*,N} \defeq \inv\zeta_{*,N}\E\Biggl[\sum_{i=1}^{N-1} \Sigma_i^*\SigmaTest_i^* \thisy\Biggr].
    \end{equation*}
    Then
    \[
        \sum_{i=0}^{N-1} \E[\gap'_{i+1}(\nextx,\nexty) + \zeta_{*,N}\nexxt{\delta_{G+J,N}} + \zeta_{*,N}\nexxt{\delta_{F^*,N}}] 
        \ge \zeta_{*,N} \gap(\tilde x_{*,N}, \tilde y_{*,N})
        \cbstart \quad (N \ge 1). \cbend
    \]  
\end{lemma}

\begin{proof}
    Shifting indices of $y^i$ by one compared to $\gap'_{i+1}$, we define
    \begin{equation*}
        \begin{split}
        \gap'_{*,i+1}
        & \defeq
        \iprod{\subdiff [G+J](\nextx)}{\nextx - \realoptx}_{\TauTest_i\Tau_i}
        + \iprod{\subdiff F^*(\thisy)}{\Sigma_{i}^*\SigmaTest_{i}^*(\thisy - \realopty)}
        \\ & \phantom{ \defeq }
        - \iprod{\realopty}{(K\Tau_i^*\TauTest_i^* - \SigmaTest_{i}\Sigma_{i}K)\realoptx} 
        - \iprod{\thisy}{\SigmaTest_{i}\Sigma_{i} K \realoptx}
        + \iprod{\realopty}{K \Tau_i^*\TauTest_i^*\nextx}.
        \end{split}
    \end{equation*}
    Reorganising terms, therefore
    \[
        \begin{split}
        \sum_{i=0}^{N-1} \gap'_{i+1}(\nextx,\nexty)
        &
        =
        \iprod{\subdiff [G+J](x^{1})-K^*\realopty}{x^{1} - \realoptx}_{\TauTest_{0}\Tau_{0}}
        \\ & \phantom{ = }
        +\iprod{\subdiff F^*(y^{N})+K\realoptx}{y^N - \realopty}_{\SigmaTest_{N}\Sigma_{N}}
        + \sum_{i=1}^{N-1} \gap'_{*, i+1}(\nextx,\nexty).
        \end{split}
    \]
    By virtue of $0 \in H(\realoptu)$, we have $K^*\realopty \in \subdiff G(\realoptx)$, and $-K\realoptx \in \subdiff F^*(\realopty)$. Estimating with \eqref{eq:g-strong-monotone} and \eqref{eq:f-monotone}, and afterwards taking the expectation, we therefore obtain
    \[
    	\sum_{i=0}^{N-1} \E[\gap'_{i+1}(\nextx,\nexty)]
    	\ge 
    	\sum_{i=1}^{N-1} \E[\gap'_{*,i+1}(\nextx,\nexty)].
    \]
    From here we may proceed analogously to the proof of \cref{lemma:gap-transform}.
\end{proof}

\subsection{Final gap estimates}
\label{sec:gap-conversion}

As now convert the abstract ergodic conditions \eqref{eq:g-convex-abstract} and \eqref{eq:f-convex-abstract} into ergodic strong convexity and smoothness conditions that can be derived from the corresponding standard properties in block-separable cases.

\cbstart Recall the spaces of operator $\TauMonotoneSpace$ and $\SigmaMonotoneSpace$ from \cref{sec:saddle}. \cbend
We assume for all $N \ge 1$ that whenever $\TauTilde_{i}\, (\defeq\TauTest_i\Tau_i) \in \Random(\TauMonotoneSpace)$ and $\nextx \in \Random(X)$ for each $i=0,\ldots,N-1$ with $\sum_{i=0}^{N-1} \E[\TauTilde_i]=I$, then for some $0 \le \Gamma \in \L(X; X)$  we have the \cbstart \term{ergodic strong convexity}\cbend
\begin{equation}
    \label{eq:g-strong-convex}
    \tag{G-EC}
    G(\realoptx)-
    G\Biggl(\sum_{i=0}^{N-1} \E[\TauTilde_{i}^*\nextx]\Biggr)
    \ge
    \sum_{i=0}^{N-1} \E\bigl[\iprod{\subdiff  G(\nextx)}{\realoptx-\nextx}_{\TauTilde_{i}}
    +
     \frac{1}{2}\norm{\realoptx-\nextx}_{\TauTilde_i\Gamma}^2\bigr].
\end{equation}
Analogously, we assume for $\SigmaTilde_{i+1}\, (\defeq\SigmaTest_{i+1}\Sigma_{i+1}) \in \Random(\SigmaMonotoneSpace)$ and $\nexty \in \Random(Y)$ for each $i=0,\ldots,N-1$ with $\sum_{i=0}^{N-1} \E[\SigmaTilde_{i+1}]=I$ the \cbstart\term{ergodic convexity}\cbend
\begin{equation}
    \label{eq:f-convex}
    \tag{F$^*$-EC}
    F^*(\realopty)-F^*\Biggl(\sum_{i=0}^{N-1} \E[\SigmaTilde_{i+1}^*\nexty]\Biggr)
    \ge
    \sum_{i=0}^{N-1} \E\bigl[\iprod{\subdiff F^*(\nexty)}{\realopty-\nexty}_{\SigmaTilde_{i+1}}
    \bigr].
\end{equation}
Finally, we assume $J$ is differentiable and satisfies for some parameters $L_i \ge 0$ the \cbstart\term{3-point ergodic smoothness}\cbend condition
\begin{equation}
    \label{eq:j-smooth-ergodic}
    \tag{J-ES}
    J(\realoptx)-
    J\Biggl(\sum_{i=0}^{N-1} \E[\TauTilde_{i}^*\nextx]\Biggr)
    \ge
    \sum_{i=0}^{N-1} \E\bigl[\iprod{\grad  J(\thisx)}{\realoptx-\nextx}_{\TauTilde_{i}}
    -\frac{L_i}{2}\norm{\nextx-\thisx}_{\TauTilde_i}^2\bigr].
\end{equation}
The shifting refers to uses of $\thisx$, where a typical definition of smoothness would use $\realoptx$.

\begin{example}[Block-separable structure, ergodic convexity]
    \label{example:ergodic-convexity-block-separable}
	Let $G$ and $\mathcal{T}$ have the separable structure of \cref{example:separable}. We claim that the ergodic strong convexity \eqref{eq:g-strong-convex} holds.
	Indeed, let us introduce $\TauTilde_i \defeq \sum_{j=1}^m \tilde\tau_{j,i} P_j \ge 0$, satisfying $\sum_{i=0}^{N-1} \E[\tilde\tau_{j,i}]=1$ for each $j=1,\ldots,m$. Splitting \eqref{eq:g-strong-convex} into separate inequalities over all $j=1,\ldots,m$, and using the strong convexity of $G_j$, we see \eqref{eq:g-strong-convex} to be true with $\Gamma = \sum_{j=1}^m \gamma_j P_j$ if for all $j=1,\ldots,m$ holds
	\begin{equation}
	    \label{eq:g-strong-convex-j-convexity-used}
        G_j(P_j\realoptx)
        -
	    G_j\Biggl(\sum_{i=0}^{N-1} \E[\tilde\tau_{j,i} P_j\nextx]\Biggr)
	    \ge
	    \sum_{i=0}^{N-1} \E\left[
	        \tilde\tau_i\left(G_j(P_j \realoptx)-G_j(P_j \nextx)\right)
	    \right].
	\end{equation}
	The right hand side can also be written as $\int_{\Omega^N} G_j(P_j \realoptx)-G_j(P_j x^i(\omega)) \d\mu^N(i,\omega)$ for the measure $\mu^N \defeq \tilde \tau_j \sum_{i=0}^{N-1} \delta_i \times \P$ on the domain $\Omega^N \defeq \{0,\ldots,N-1\} \times \Omega$. Using our assumption $\sum_{i=0}^{N-1} \E[\tilde\tau_{j,i}]=1$, we deduce $\mu^N(\Omega^N)=1$. 
	An application of Jensen's inequality now shows \eqref{eq:g-strong-convex-j-convexity-used}. Therefore \eqref{eq:g-strong-convex} is satisfied for $G=G$.
\end{example}

\begin{example}[Ergodic smoothness for smooth $J$]
    \label{example:ergodic-smoothness}
    \cbstart%
    If $J \in C(x)$ has $L$-Lipschitz gradient, then \cref{lemma:smoothness} shows the three-point inequality
    \begin{equation*}
        J(\realoptx)-J(\nextx) \ge \iprod{\grad J(\thisx)}{\realoptx-\nextx} -  \frac{L}{2}\norm{\nextx-\thisx}^2.
    \end{equation*}
    If $\tilde\Tau_i=\tilde\tau_i I$ for scalar $\tilde\tau_i I$, then proceeding as in \eqref{eq:g-strong-convex-j-convexity-used} in \cref{example:ergodic-convexity-block-separable}, we deduce the 3-point ergodic smoothness \eqref{eq:j-smooth-ergodic} with $L_i=L$. Similarly, we can treat the block-separable case  $J=\sum_{i=0}^m J_j(P_j x)$ when each $J_j$ individually has Lipschitz gradient.
    \cbend%
\end{example}

The next theorem is our main result for saddle point problems.
\cbstart
To clarify the statement of the theorem, which depends on various different combinations of several conditions in the definition of $\tilde g_N$, we recall here the rough meaning of each:
\begin{description}[labelwidth=6em,labelindent=\parindent,itemsep=1pt]
    \item[\eqrefpage{eq:convergence-fundamental-condition-iter}] Fundamental condition \eqref{eq:convergence-fundamental-condition-iter-h} for saddle point problems.
    \item[\eqrefpage{eq:g-strong-monotone}] Partial (testing and step length operator relative) strong monotonicity of $G$.
    \item[\eqrefpage{eq:f-monotone}] Partial monotonicity of $F^*$.
    \item[\eqrefpage{eq:j-coco}] Partial co-coercivity of $J$.
    \item[\eqrefpage{eq:g-strong-convex}] Partial strong ergodic convexity of $G$.
    \item[\eqrefpage{eq:f-convex}] Partial ergodic convexity of $F^*$.
    \item[\eqrefpage{eq:j-smooth-ergodic}] Partial 3-point ergodic smoothness of $J$.
    \item[\eqrefpage{eq:cond-eta}] First alternative primal--dual coupling condition
    \item[\eqrefpage{eq:cond-etatwo}] Second alternative primal--dual coupling condition
\end{description}
\cbend

\begin{theorem}
    \label{thm:convergence-result-stochastic-saddle}
    Let $H$ have the structure \eqref{eq:h} and assume $\realoptu \in \inv H(0)$.
    For each $i \in \N$, let $\Tau_i, \TauTest_i \in \Random(\L(X; X))$ and $\Sigma_{i+1}, \SigmaTest_{i+1} \in \Random(\L(Y; Y))$ be such that $\TauTest_i\Tau_i \in \Random(\TauMonotoneSpace)$ and $\SigmaTest_{i+1}\Sigma_{i+1} \in \Random(\SigmaMonotoneSpace)$.
    Define $\Test_{i+1}$ and $\Step_{i+1}$ through \eqref{eq:test}.
    Also take $\AltBregFn_{i+1} \in \Random(X \times Y \to X \times Y)$ and $\Precond_{i+1} \in \Random(\linear(X \times Y; X \times Y))$.
    Suppose \eqref{eq:pp} is solvable for $\{\nextu\}_{i \in \N} \subset X \times Y$.
    Assuming one of the following cases to hold with $0 \le \Gamma \in \linear(X; X)$ and $L_i \ge 0$, let
    \[
        \tilde g_N \defeq
        \begin{cases}
             0, & \tilde\Gamma=\Gamma, \ \text{\eqref{eq:g-strong-monotone}, \eqref{eq:f-monotone} and \eqref{eq:j-coco} hold},
             \\
             \zeta_N \gap(\tilde x_N, \tilde y_N), & \tilde\Gamma=\Gamma/2; \text{\eqref{eq:g-strong-convex}, \eqref{eq:f-convex}, \eqref{eq:j-smooth-ergodic}, and \eqref{eq:cond-eta} hold},
             \\
             \zeta_{*,N} \gap(\tilde x_{*,N}, \tilde y_{*,N}), & \tilde\Gamma=\Gamma/2; \text{\eqref{eq:g-strong-monotone} for $\Gamma=0$, \eqref{eq:f-monotone}}, \\ &\text{\quad\eqref{eq:g-strong-convex}, \eqref{eq:f-convex}, \eqref{eq:j-smooth-ergodic}, and \eqref{eq:cond-eta} hold}.
        \end{cases}
    \]
    If \cbstart for all $i \in \N$, \cbend $\Test_{i+1}\Precond_{i+1}$ is self-adjoint and \eqref{eq:convergence-fundamental-condition-iter} holds for $\tilde\Gamma$ given above, then so does the following \term{ergodic gap descent inequality}:
    \begin{equation*}
        \label{eq:convergence-result-stochastic-saddle}
        \tag{DI-$\gap$}
        \E\Bigl[\frac{1}{2}\norm{u^N-\realoptu}^2_{\Test_{N+1}\Precond_{N+1}}\Bigr]
        + \tilde g_N
        \le
        \E\left[\frac{1}{2}\norm{u^0-\realoptu}_{\Test_1 \Precond_1}^2\right] + \sum_{i=0}^{N-1} \E[\Penalty_{i+1}(\realoptu)]
        \cbstart \quad (N \ge 1). \cbend
    \end{equation*}
\end{theorem}

\begin{proof}
	The case $\tilde g_N=0$ is simply the result of taking the expectation in the claim of \cref{thm:convergence-result-saddle}; \cbstart
    compare how \cref{cor:convergence-result-main-h-stoch} follows form \cref{thm:convergence-result-main-h}. \cbend
    %
    Regarding the remaining two cases, clearly \eqref{eq:convergence-fundamental-condition-iter} implies \eqref{eq:convergence-fundamental-condition-iter-gap} for
    \[
        \begin{split}
        \tilde\Penalty_{i+1}(\realoptu)
        &
        \defeq
        \Penalty_{i+1}(\realoptu)
        - \frac{1}{2}\norm{\realoptx-\nextx}_{\TauTest_i\Tau_i\tilde\Gamma}+\frac{L_i}{2}\norm{\nextx-\thisx}_{\TauTest_i\Tau_i}^2
        \\ & \phantom{ \defeq }
        +\iprod{\grad J(\thisx)-\grad J(\nextx)}{\realoptx-\nextx}_{\TauTest_i\Tau_i}.
        \end{split}
    \]
    Thus \cref{lemma:convergence-result-gap} shows the descent estimate \eqref{eq:convergence-result-gap}.

    The ergodic strong convexity \eqref{eq:g-strong-convex} and \eqref{eq:j-smooth-ergodic} imply \eqref{eq:g-convex-abstract} for
    \[
        \nexxt{\delta_{G+J,N}} \defeq \frac{1}{2}\norm{\realoptx-\nextx}_{\TauTilde_i\Gamma}-\frac{L_i}{2}\norm{\nextx-\thisx}_{\TauTilde_i}^2+\iprod{\grad J(\thisx)-\grad J(\nextx)}{\realoptx-\nextx}_{\TauTilde_i},
    \]
    where $\TauTilde_i \in \Random(\TauMonotoneSpace)$.
    Likewise the ergodic convexity \eqref{eq:f-convex} implies \eqref{eq:f-convex-abstract} for $\nexxt{\delta_{F^*,N}} \defeq 0$.
    When the first primal--dual coupling condition \eqref{eq:cond-eta} holds, we take above $\TauTilde_i=\inv\zeta_N \TauTest_i\Tau_i$, which we have assumed to belong to $\Random(\TauMonotoneSpace)$. \cbstart
    If the alternative second primal--dual coupling condition \eqref{eq:cond-etatwo} holds, we take $\TauTilde_i=\inv\zeta_{*,N} \TauTest_i\Tau_i$. \cbend
    Therefore, \eqref{eq:convergence-result-gap} can be rewritten
    \begin{equation}
        \label{eq:convergence-result-stochastic-saddle-pre}
        \frac{1}{2}\norm{u^N-\realoptu}^2_{\Test_{N+1}\Precond_{N+1}}
        + g_N'
        \le
        \frac{1}{2}\norm{u^0-\realoptu}^2_{\Test_{1}\Precond_{1}}
        +
        \sum_{i=0}^{N-1} \Penalty_{i+1}(\realoptu)
    \end{equation}    
    for
    \[
        g_N' \defeq \sum_{i=0}^{N-1} \bigl[
            \gap'_{i+1}(\nextx,\nexty)
            +\zeta_{N}\nexxt{\delta_{G+J,N}}+\zeta_{N}\nexxt{\delta_{F^*,N}}
            \bigr].
    \]
    Now we just take the expectation in \eqref{eq:convergence-result-stochastic-saddle-pre}, and apply \cref{lemma:gap-transform} \cbstart or  \cref{lemma:gap-transformtwo}. \cbend
\end{proof}


\subsection{Primal--dual examples revisited}

We now study gap estimates for several of the examples from \cref{sec:saddle}.
\cbstart
We start by verifying partial monotonicity and ergodic convexity and smoothness conditions for in the case of simple deterministic scalar step length and testing operators: the block-separable and stochastic case we leave to the companion paper \cite{tuomov-blockcp}.
\cbend


\cbstart

Similarly to \cref{cor:convergence-result-saddle-nogap} of \cref{thm:convergence-result-saddle}, we now have the following non-stochastic scalar corollary of \cref{thm:convergence-result-stochastic-saddle}.
From the corollary, if $\Penalty_{i+1} \le 0$, we clearly get the convergence of $\gap(\tilde x_{*,N}, \tilde y_{*,N})$ or $\gap(\tilde x_N, \tilde y_N)$ to zero at the respective rate $O(1/\zeta_{*,N})$ or $(1/\zeta_{N})$.

\begin{corollary}
    \label{cor:convergence-result-saddle}
    Let $H$ have the structure \eqref{eq:h} and assume $\realoptu \in \inv H(0)$.
    Assume that $G$ is ($\gamma$-strongly) convex and $\grad J$ is $L$-Lipschitz for some $\gamma \ge 0$ and $L>0$.
    For each $i \in \N$, assume the structure \eqref{eq:pd-setup-scalar} for $\tau_i,\tauTest_i,\sigma_{i+1},\sigmaTest_{i+1} > 0$.
    Also take $\AltBregFn_{i+1} \in X \times Y \to X \times Y$ and $\Precond_{i+1} \in \linear(X \times Y; X \times Y)$.
    Suppose \eqref{eq:pp} is solvable for $\{\nextu\}_{i \in \N} \subset X \times Y$.
    Suppose for all $i \in \N$ that $\tauTest_i\tau_i=\sigmaTest_i\sigma_i$, that $\Test_{i+1}\Precond_{i+1}$ is self-adjoint, and that the fundamental condition for saddle-point problems \eqref{eq:convergence-fundamental-condition-iter} holds for $\tilde\Gamma=(\gamma/2)I$ and $L_i \equiv L$. Then
    \begin{equation*}
        \frac{1}{2}\norm{u^N-\realoptu}^2_{\Test_{N+1}\Precond_{N+1}}
        + \zeta_{*,N}\gap(\tilde x_{*,N}, \tilde y_{*,N})
        \le
        \frac{1}{2}\norm{u^0-\realoptu}_{\Test_1 \Precond_1}^2 + \sum_{i=0}^{N-1} \Penalty_{i+1}(\realoptu)
        \quad (N \ge 1).
    \end{equation*}
    If, instead, $\tauTest_i\tau_i=\sigmaTest_{i+1}\sigma_{i+1}$, then the gap expression is replaced by $\zeta_{N}\gap(\tilde x_{N}, \tilde y_{N})$.
\end{corollary}

\begin{proof}
    As in the proof of \cref{cor:convergence-result-saddle-nogap}, clearly $\TauTest_i\Tau_i \in \TauMonotoneSpace \defeq [0, \infty)I$ and $\SigmaTest_{i+1}\Sigma_{i+1} \in \SigmaMonotoneSpace \defeq [0, \infty)I$, so that the partial monotonicities \eqref{eq:f-monotone} and \eqref{eq:g-strong-monotone} (with $\Gamma=0$) hold by the monotonicity of the subdifferentials of $G$ and $F^*$. Similarly, the ergodic (strong) convexity \eqref{eq:g-strong-convex} of $G$ with $\Gamma=\gamma I$ and \eqref{eq:f-convex} of $F^*$ hold by a Jensen argument similar to \cref{example:ergodic-convexity-block-separable}. Likewise, the ergodic smoothness \eqref{eq:j-smooth-ergodic} holds by the three-point inequality \cref{eq:three-point-smoothness} and a Jensen argument similar to \cref{example:ergodic-smoothness}. Note that with everything deterministic, the expectations disappear.

    With this, the result follows immediately from \cref{thm:convergence-result-stochastic-saddle} for the second and third cases of $\tilde g_N$.
    The primal--dual coupling conditions \eqref{eq:cond-etatwo} and \eqref{eq:cond-eta} reduce to our respective conditions $\tauTest_i\tau_i=\sigmaTest_i\sigma_i$ and $\tauTest_i\tau_i=\sigmaTest_{i+1}\sigma_{i+1}$,
\end{proof}

In \cref{example:cpock,example:gist}, we proved \eqref{eq:convergence-fundamental-condition-iter} for the Chambolle--Pock method and the GIST with $\tilde\Gamma=\gamma I$ and $L_i \equiv L/2$. Now we have to do the same but with the factor-of-two different $\tilde\Gamma=(\gamma/2)I$ and $L_i \equiv L$.
The different $\tilde\Gamma$ will merely change the acceleration factor of the method. The larger $L_i$, on the other hand, will change the step length bound \eqref{eq:cpock-init-j} of the forward-step Chambolle--Pock, \cref{example:cpock-forward}, to
\begin{equation}
    \label{eq:cpock-init-j-gap}
    0 < \theta \defeq 1-2L\tau_0/(1-\tau_0\sigma_0\norm{K}^2),
\end{equation}
and the the bound $\norm{A} \le \sqrt{2}$ of the GIST of \cref{example:gist} to $\norm{A} \le 1$.

\begin{example}[Gap for Chambolle--Pock with a forward step]
	\label{example:cpock-gap}
    In the demonstration of \cref{example:cpock,example:cpock-forward}, we have seen the Chambolle--Pock method to satisfy $\tauTest_i\tau_i=\sigmaTest_i\sigma_i$ and the self-adjointness of $\Test_{i+1}\Precond_{i+1}$.
    As discussed above, \eqref{eq:convergence-fundamental-condition-iter} holds with $\Delta_{i+1} \le 0$ subject to the conditions $\tilde\gamma \in [0,\gamma/2]$ and \eqref{eq:cpock-init-j-gap}.
    We now have $\zeta_{*,N}=\sum_{i=1}^{N-1} \tauTest_i^{1/2}$.
    In the unaccelerated case ($\gamma=0$), we get $\zeta_{*,N} = N \tauTest_0^{1/2}$. Therefore, we get from \cref{cor:convergence-result-saddle} the $O(1/N)$ convergence of $\gap(\tilde x_{*,N}, \tilde y_{*,N})$ to zero.
    In the accelerated case ($\gamma>0$), $\tauTest_i$ is of the order $\Theta(i^2)$. Therefore also $\zeta_{*,N}$ is of the order $\Theta(N^2)$, so we get $O(1/N^2)$ convergence of $\gap(\tilde x_{*,N}, \tilde y_{*,N})$ to zero.
\end{example}

\begin{example}[Gap for GIST]
    In \cref{example:gist} we have seen the GIST to satisfy $\tau_i=\tauTest_i=\sigma_{i+1}=\sigmaTest_{i+1}=1$, the self-adjointness of $\Test_{i+1}\Precond_{i+1}$. Moreover, as discussed above, \eqref{eq:convergence-fundamental-condition-iter} with $\Delta_{i+1} \le 0$ if $\norm{A} \le 1$. It therefore has $\zeta_N = N-1$ and $\zeta_{*,N} = N$. 
    Consequently, \cref{cor:convergence-result-saddle} yields the $O(1/N)$ convergence of both $\gap(\tilde x_{*,N}, \tilde y_{*,N})$ and $\gap(\tilde x_N, \tilde y_N)$ to zero.
\end{example}

\cbend

\section*{Conclusion}

We have unified common convergence proofs of optimisation methods, employing the ideas of non-linear preconditioning and testing of the classical proximal point method. We have demonstrated that popular classical and modern algorithms can be presented in this framework, and their convergence, including convergence rates, proved with little effort. The theory was, however, not developed with existing algorithms in mind. It was developed to allow the development of new spatially adapted block-proximal methods in \cite{tuomov-blockcp}.
We will demonstrate there and in other works to follow, the full power of the theory.
For one, we did not yet fully exploit the fact that $\Step_{i+1}$ and $\Test_{i+1}$ are operators, to construct step-wise step lengths and acceleration.

\appendix

\section{Outer semicontinuity of maximal monotone operators}
\label{sec:maxmono-outersemi}

We could not find the following result explicitly stated in the literature, although it is hidden in, e.g., the proof of \cite[Theorem 1]{rockafellar1976monotone}.

\begin{lemma}
    \label{lemma:maxmono-outersemi}
    Let $H: \Space \setto \Space$ be maximal monotone on a Hilbert space $\Space$.
    Then $H$ is is weak-to-strong outer semicontinuous: for any sequence $\{\thisu\}_{i \in \N}$, and any $\this{z} \in H(\thisu)$ such that $\thisu \weakto u$ weakly, and $\this{z} \to z$ strongly, we have $z \in H(u)$.
\end{lemma}

\begin{proof}
    By monotonicity, for any $u' \in \Space$ and $z' \in \Space$ holds $D_i \defeq \iprod{u'-\thisu}{z'-\thisz} \ge 0$. Since a weakly convergent sequence is bounded, we have $D_i \ge \iprod{u'-\thisu}{z'-z}-C\norm{z-\thisz}$ for some $C>0$ independent of $i$.
    Taking the limit, we therefore have $\iprod{u'-u}{z'-z} \ge 0$. If we had $z \not \in H(u)$, this would contradict that $H$ is maximal, i.e., its graph not contained in the graph of any monotone operator.
\end{proof}

\cbstart

\section{Three-point inequalities}
\label{sec:three-point}

The following three-point formulas are central to handling forward steps with respect to smooth functions.

\begin{lemma}
    \label{lemma:smoothness}
    If $J \in \convex(X)$ has $L$-Lipschitz gradient. Then
    \begin{equation}
        \label{eq:three-point-hypomonotonicity}
        \iprod{\grad J(z)-\grad J(\realoptx)}{x-\realoptx} \ge
            -\frac{L}{4}\norm{x-z}^2
        \quad (\realoptx, z, x \in X),
    \end{equation}
    as well as
    \begin{equation}
        \label{eq:three-point-smoothness}
        \iprod{\grad J(z)}{x-\realoptx}
        \ge
        J(x)-J(\realoptx) -  \frac{L}{2}\norm{x-z}^2
        \quad (\realoptx, z, x \in X).
    \end{equation}
\end{lemma}

\begin{proof}
    Regarding the ``three-point hypomonotonicity'' \eqref{eq:three-point-hypomonotonicity}, the $L$-Lipschitz gradient implies co-coercivity (see \cite{bauschke2017convex} or \cref{sec:subspace-smoothness})
    \[
        \iprod{\grad J(z)-\grad J(\realoptx)}{z-\realoptx}
        \ge \inv L \norm{\grad J(z)-\grad J(\realoptx)}^2.
    \]
    Thus using Cauchy's inequality
    \[
        \begin{split}
        \iprod{\grad J(z)-\grad J(\realoptx)}{x-\realoptx}
        &
        =\iprod{\grad J(z)-\grad J(\realoptx)}{z-\realoptx}
          +\iprod{\grad J(z)-\grad J(\realoptx)}{x-z}
        \\
        &
        \ge -\frac{L}{4}\norm{x-z}^2.
        \end{split}
    \]

    To prove \eqref{eq:three-point-smoothness}, the Lipschitz gradient implies the smoothness or ``descent inequality'' (again, \cite{bauschke2017convex} or \cref{sec:subspace-smoothness})
    \begin{equation}
        \label{eq:smoothness}
        J(z)-J(x) \ge \iprod{\grad J(z)}{z-x} - \frac{L}{2}\norm{x-z}^2.
    \end{equation}
    By convexity $J(\realoptx)-J(z) \ge \iprod{\grad J(z)}{\realoptx-z}$.
    Summed, we obtain \eqref{eq:three-point-smoothness}.    
\end{proof}

\begin{lemma}
    \label{lemma:sc-smoothness}
    If $J \in \convex(X)$ has $L$-Lipschitz gradient and is $\gamma$-strongly convex. Then for any $\tau>0$ holds
    \begin{equation}
        \label{eq:three-point-hypomonotonicity-sc}
        \iprod{\grad J(z)-\grad J(\realoptx)}{x-\realoptx} \ge
            \frac{2\gamma-\tau L^2}{2}\norm{x-\realoptx}^2
            -\frac{1}{2\tau}\norm{x-z}^2
        \quad (\realoptx, z, x \in X),
    \end{equation}
    as well as
    \begin{equation}
        \label{eq:three-point-smoothness-sc}
        \iprod{\grad J(z)}{x-\realoptx}
        \ge
        J(x)-J(\realoptx) + \frac{\gamma-\tau L^2}{2}\norm{x-\realoptx}^2
            -\frac{1}{2\tau}\norm{x-z}^2
        \quad (\realoptx, z, x \in X).
    \end{equation}
\end{lemma}

\begin{proof}
    To prove \eqref{eq:three-point-smoothness-sc}, using strong convexity,the Lipschitz gradient, and Cauchy's inequality, we have
    \[
        \begin{split}
        \iprod{\grad J(z)}{x-\realoptx}
        &
        =\iprod{\grad J(x)}{x-\realoptx}
          +\iprod{\grad J(z)-\grad J(x)}{x-\realoptx}
        \\
        & \ge J(x)-J(\realoptx) + \frac{\gamma}{2}\norm{x-\realoptx}^2 
        -\frac{1}{2\tau}\norm{x-z}^2 - \frac{\tau L^2}{2}\norm{x-\realoptx}^2.
        \end{split}
    \]

    Regarding \eqref{eq:three-point-hypomonotonicity-sc}, using the $\gamma$-strong monotonicity of $\grad J$, we estimate completely analogously
    \[
        \begin{split}
        \iprod{\grad J(z)-\grad J(\realoptx)}{x-\realoptx}
        &
        =\iprod{\grad J(x)-\grad J(\realoptx)}{x-\realoptx}
          +\iprod{\grad J(z)-\grad J(x)}{x-\realoptx}
        \\
        & \ge \gamma\norm{x-\realoptx}^2 
        -\frac{1}{2\tau}\norm{x-z}^2 - \frac{\tau L^2}{2}\norm{x-\realoptx}^2.
        \qedhere
        \end{split}
    \]
\end{proof}

Since smooth functions with a positive Hessian are locally convex, the above lemmas readily extend to this case, locally. In fact, we have following more precise result:

\begin{lemma}
    \label{lemma:c2-smoothness}
    Suppose $J \in C^2(X)$ with $\grad^2 J(\realoptx) > 0$ at given $\realoptx \in X$. Then for any $\tau \in (0, 2]$ and all $z, x, \eta \in X$, we have
    \begin{equation}
        \label{eq:three-point-hypomonotonicity-c2}
        \iprod{\grad J(z)-\grad J(\realoptx)}{x-\realoptx}
        \ge \frac{(1-\delta_{z,\eta})(2-\tau)}{2}\norm{x-\realoptx}^2_{\grad^2 J(\eta)}
        -\frac{1+\delta_{z,\eta}}{2\tau} \norm{x-z}^2_{\grad^2 J(\eta)}
    \end{equation}
    with
    \begin{equation}
        \label{eq:delta-eta}
        \delta_{z,\eta} \defeq \inf\left\{\delta \ge 0 \,\middle|\,
            \begin{array}{r}
                 (1-\delta)\grad^2 J(\eta) \le \grad^2 J(\zeta) \le (1+\delta)\grad^2 J(\eta) \\
                 \text{ for all } \zeta \in \closure \B(\norm{z-\realoptx}, \realoptx)
            \end{array}
            \right\}.
    \end{equation}
    If $x \in \closure \B(\norm{z-\realoptx}, \realoptx)$, then also
    \begin{equation}
        \label{eq:three-point-smoothness-c2}
        \iprod{\grad J(z)}{x-\realoptx}
        \ge
        J(x)-J(\realoptx) + \frac{(1-\delta_{z,\eta})(1-\tau)-2\delta_{z,\eta}}{2}\norm{x-\realoptx}_{\grad^2 J(\eta)}^2
            -\frac{1+\delta_{z,\eta}}{2\tau}\norm{x-z}_{\grad^2 J(\eta)}^2.
    \end{equation}

\end{lemma}

\begin{proof}
    By Taylor expansion, for some $\zeta$ between $z$ and $\realoptx$, and any $\tau>0$, we have
    \begin{equation}
        \label{eq:three-point-hypomonotonicity-c2-proof}
        \begin{split}
        \iprod{\grad J(z)-\grad J(\realoptx)}{x-\realoptx}
        &
        =\iprod{\grad^2 J(\zeta)(z-\realoptx)}{x-\realoptx}
        \\
        &
        =\norm{x-\realoptx}^2_{\grad^2 J(\zeta)}
        +\iprod{\grad^2 J(\zeta)(z-x)}{x-\realoptx}
        \\
        &
        \ge \frac{2-\tau}{2}\norm{x-\realoptx}^2_{\grad^2 J(\zeta)}
        -\frac{1}{2\tau} \norm{x-z}^2_{\grad^2 J(\zeta)}.
        \end{split}
    \end{equation}
    Since $\zeta \in \closure \B(\norm{z-\realoptx}, \realoptx)$, by the definition of $\delta_{z,\eta}$, we obtain \eqref{eq:three-point-hypomonotonicity-c2}.

    Similarly, by Taylor expansion, for some $\zeta_0$ between $x$ and $\realoptx$, we have
    \begin{equation}
        \label{eq:three-point-smoothness-c2-proof}
        \iprod{\grad J(z)}{x-\realoptx} - J(x) + J(\realoptx)
        =
        \iprod{\grad J(z)-\grad J(\realoptx)}{x-\realoptx}
        -\frac{1}{2}\iprod{\grad^2 J(\zeta_0)(x-\realoptx)}{x-\realoptx}
    \end{equation}
    Using \eqref{eq:three-point-hypomonotonicity-c2-proof} we obtain
    \[
        \begin{split}
        \iprod{\grad J(z)}{x-\realoptx} - J(x) + J(\realoptx)
        &
        \ge \frac{1}{2}\norm{x-\realoptx}^2_{(2-\tau)\grad^2 J(\zeta) - \grad^2 J(\zeta_0)}
        -\frac{1}{2\tau} \norm{x-z}^2_{\grad^2 J(\zeta)}.
        \end{split}
    \]
    Using the assumption $x \in \closure \B(\norm{z-\realoptx}, \realoptx)$, we have $\zeta_0 \in \closure \B(\norm{z-\realoptx}, \realoptx)$. Hence we obtain \eqref{eq:three-point-smoothness-c2} by the definition of $\delta_{z,\eta}$ and $(1-\delta_{z,\eta})(2-\tau)-(1+\delta_{z,\eta})=(1-\delta_{z,\eta})(1-\tau)-2\delta_{z,\eta}$.
\end{proof}

We can also derive the following alternate result:

\begin{lemma}
    \label{lemma:c2x-smoothness}
    Suppose $J \in C^2(X)$ with $\grad^2 J(\realoptx) > 0$ at given $\realoptx \in X$. Then for all $z, x, \eta \in X$ we have
    \begin{equation}
        \label{eq:three-point-hypomonotonicity-c2x}
        \iprod{\grad J(z)-\grad J(\realoptx)}{x-\realoptx}
        \ge
        \frac{1-\delta_{z,\eta}}{2}\norm{x-\realoptx}^2_{\grad^2 J(\eta)}
        +
        \frac{1-\delta_{z,\eta}}{2}\norm{z-\realoptx}^2_{\grad^2 J(\eta)}
        -
        \frac{1}{2}\norm{x-z}^2_{\grad^2 J(\eta)}
    \end{equation}
    for $\delta_{z,\eta}$ given by \eqref{eq:delta-eta}.
    If $x \in \closure \B(\norm{z-\realoptx}, \realoptx)$, then also
    \begin{equation}
        \label{eq:three-point-smoothness-c2x}
        \begin{split}
        \iprod{\grad J(z)}{x-\realoptx}
        &
        \ge
        -\delta_{z,\eta}\norm{x-\realoptx}^2_{\grad^2 J(\eta)}
        +
        \frac{1-\delta_{z,\eta}}{2}\norm{z-\realoptx}^2_{\grad^2 J(\eta)}
        -
        \frac{1}{2}\norm{x-z}^2_{\grad^2 J(\eta)}
        \\ \MoveEqLeft[-1]
        + J(x)-J(\realoptx).
        \end{split}
    \end{equation}
\end{lemma}

\begin{proof}
    By Taylor expansion, for some $\zeta$ between $z$ and $\realoptx$, we have
    \begin{equation}
        \label{eq:three-point-hypomonotonicity-pre}
        \begin{split}
        \iprod{\grad J(z)-\grad J(\realoptx)}{x-\realoptx}
        &
        =\iprod{\grad^2 J(\zeta)(z-\realoptx)}{x-\realoptx}
        \\
        &
        =\iprod{\grad^2 J(\eta)(z-\realoptx)}{x-\realoptx}
        \\ \MoveEqLeft[-1]
        +\iprod{[\grad^2 J(\zeta)-\grad^2 J(\eta)](z-\realoptx)}{x-\realoptx}
        \\
        &
        \ge
        \iprod{\grad^2 J(\eta)(z-\realoptx)}{x-\realoptx}
        \\ \MoveEqLeft[-1]
        -
        \frac{\delta_{z,\eta}}{2}\norm{x-\realoptx}_{\grad^2 J(\eta)}
        -
        \frac{\delta_{z,\eta}}{2}\norm{z-\realoptx}_{\grad^2 J(\eta)}.
        \end{split}
    \end{equation}
    In the last step we have used Cauchy's inequality, and the definition of $\delta_{z,\eta}$ following $\zeta \in \closure \B(\norm{z-\realoptx}, \realoptx)$.
    The standard three-point or Pythagoras' identity states
    \[
        \iprod{\grad^2 J(\eta)(z-\realoptx)}{x-\realoptx}
        =
        \frac{1}{2}\norm{z-\realoptx}^2_{\grad^2 J(\eta)}
        +
        \frac{1}{2}\norm{x-\realoptx}^2_{\grad^2 J(\eta)}
        -
        \frac{1}{2}\norm{x-z}^2_{\grad^2 J(\eta)}.
    \]
    Applying this in \eqref{eq:three-point-hypomonotonicity-pre}, we obtain \eqref{eq:three-point-hypomonotonicity-c2x}.

    To prove \eqref{eq:three-point-smoothness-c2x}, we use \eqref{eq:three-point-smoothness-c2-proof}, the definition of $\delta_{z,\eta}$, and \eqref{eq:three-point-hypomonotonicity-c2x}.
    %
\end{proof}
\cbend

\section{Projected gradients and smoothness}
\label{sec:subspace-smoothness}

The next lemma generalises well-known properties \cite[see, e.g.,][]{bauschke2017convex} of smooth convex functions to projected gradients, when we take $P$ as projection operator.
With $P$ a random projection, taking the expectation in \eqref{eq:subspace-smoothness}, we in particular obtain a connection to the Expected Separable Over-approximation property in the stochastic coordinate descent literature \cite{richtarik2012parallel}.

\begin{lemma}
    \label{lemma:subspace-smoothness}
    Let $J \in \convex(X)$, and $P \in \linear(X; X)$ be self-adjoint and positive semi-definite on a Hilbert space $X$. Suppose $P$ has a pseudo-inverse $\pinv P$ satisfying $ P \pinv P P = P$.
    Consider the properties:
    \begin{enumerate}[label=(\roman*)]
        \item\label{item:subspace-lipschitz}
            $P$-relative Lipschitz continuity of $\grad J$ with factor $L$:
            \begin{equation}
                \label{eq:subspace-lipschitz}
                \norm{\grad J(x)-\grad J(y)}_P \le L \norm{x-y}_{\pinv P}
                \quad (x, y \in X).
            \end{equation}
        \item\label{item:subspace-nx}
            The $P$-relative property
            \begin{equation}
                \label{eq:subspace-nx}
                \iprod{\grad J(x+Ph) - \grad J(x)}{Ph} \le L\norm{h}_P^2
                \quad (x, h \in X).
            \end{equation}
        \item\label{item:subspace-smoothness}
            $P$-relative smoothness of $J$ with factor $L$:
            \begin{equation}
                \label{eq:subspace-smoothness}
                J(x+Ph) \le J(x) + \iprod{\grad J(x)}{Ph}+\frac{L}{2}\norm{h}_P^2
                \quad (x, h \in X).
            \end{equation}
        \item\label{item:subspace-grad-smoothness}
            The $P$-relative property
            \begin{equation}
                \label{eq:subspace-grad-smoothness}
                J(y) \le J(x) + \iprod{\grad J(y)}{y-x}-\frac{1}{2L}\norm{\grad J(x)-\grad J(y)}_P^2
                \quad (x, h \in X).
            \end{equation} 
        \item\label{item:subspace-coco}
            $P$-relative co-coercivity of $\grad J$ with factor $\inv L$:
            \begin{equation}
                \label{eq:subspace-coco}
                \inv L \norm{\grad J(x)-\grad J(y)}_P^2
                \le
                \iprod{\grad J(x)-\grad J(y)}{x-y}
                \quad (x, y \in X).
            \end{equation}           
    \end{enumerate}   
    We have \ref{item:subspace-lipschitz} $\implies$ \ref{item:subspace-nx} $\iff$ \ref{item:subspace-smoothness} $\implies $ \ref{item:subspace-grad-smoothness} $\implies$ \ref{item:subspace-coco}.
    If $P$ is invertible, all are equivalent.
\end{lemma}

\begin{proof}
    \ref{item:subspace-lipschitz} $\implies$ \ref{item:subspace-nx}:
    Take $y=x+Ph$ and multiply \eqref{eq:subspace-lipschitz} by $\norm{h}_P$.
    Then use Cauchy--Schwarz.

    \ref{item:subspace-nx} $\implies$ \ref{item:subspace-smoothness}:
    Using the mean value theorem and \eqref{eq:subspace-nx}, we compute \eqref{eq:subspace-smoothness}:
    \begin{equation*}
        \begin{split}
        J(x+Ph) & - J(x) - \iprod{\grad J(x)}{Ph} 
        =\int_0^1 \iprod{\grad J(x+tPh)}{Ph} \d t - \iprod{\grad J(x)}{Ph} 
        \\
        &
        =\int_0^1 \iprod{\grad J(x+tPh)-\grad J(x)}{Ph} \d t
        \le \int_0^1 t \d t \cdot L\norm{h}_P^2
        = \frac{L}{2} \norm{h}_P^2.
        \end{split}
    \end{equation*}

    \ref{item:subspace-smoothness} $\implies$ \ref{item:subspace-nx}:
    Add together \eqref{eq:subspace-smoothness} for $x=x'$ and $x=x'+Ph$.

    \ref{item:subspace-smoothness} $\implies$ \ref{item:subspace-grad-smoothness}:
    Adding $-\iprod{\grad J(y)}{x+Ph}$ on both sides of \eqref{eq:subspace-smoothness}, we get
    \[
        J(x+Ph) - \iprod{\grad J(y)}{x+Ph} \le J(x) - \iprod{\grad J(y)}{x} + \iprod{\grad J(x)-\grad J(y)}{Ph}+\frac{L}{2}\norm{h}_P^2.
    \]
    The left hand side is minimised with respect to $x$ by taking $x=y-Ph$.
    Taking on the right-hand side $h=\inv L(\grad J(y)-\grad J(x))$ therefore gives
    \eqref{eq:subspace-grad-smoothness}.

    \ref{item:subspace-grad-smoothness} $\implies$ \ref{item:subspace-coco}:
    Summing the estimate \eqref{eq:subspace-grad-smoothness} with the same estimate with $x$ and $y$ exchanged, we obtain \eqref{eq:subspace-coco}.

    \ref{item:subspace-coco} $\implies$ \ref{item:subspace-lipschitz} when $P$ is invertible: Cauchy--Schwarz.
\end{proof}


\bibliographystyle{plain}
\input{proxtest_.bbl}


\end{document}

%% file: texbase/texbase.tex
\usepackage{bbm}
\usepackage{amssymb}
\usepackage{mdframed}



\mdfdefinestyle{examplebg}{
    backgroundcolor=black!7,%
    hidealllines=true,%
    innertopmargin=-.3em,%
    innerbottommargin=.7em,%
    innerleftmargin=.7em,%
    innerrightmargin=.7em,%
}

\surroundwithmdframed[style=examplebg]{example}
\surroundwithmdframed[style=examplebg]{algorithm}

\newcounter{Hequation}

\makeatletter
\g@addto@macro\equation{\stepcounter{Hequation}}
\makeatother


%% file: texbase/symbols-common.tex
\newcommand{\term}{\emph}

\newcommand{\field}[1]{\mathbb{#1}}
\newcommand{\N}{\mathbb{N}}

\newcommand{\R}{\field{R}}
\newcommand{\extR}{\overline \R}
\newcommand{\B}{B}

\newcommand{\norm}[1]{\|#1\|}

\newcommand{\inv}[1]{#1^{-1}}
\newcommand{\grad}{\nabla}
\newcommand{\freevar}{\,\boldsymbol\cdot\,}

\newcommand{\Union}\bigcup
\newcommand{\Isect}\bigcap
\newcommand{\union}\cup
\newcommand{\isect}\cap
\newcommand{\bigunion}\bigcup
\newcommand{\bigisect}\bigcap

\newcommand{\defeq}{:=}

\newcommand{\downto}{\searrow}
\newcommand{\upto}{\nearrow}
\newcommand{\subdiff}{\partial}

\newcommand{\pinv}[1]{#1^\dag}

\DeclareMathOperator{\closure}{cl}

\DeclareMathOperator{\Dom}{dom}

\DeclareMathOperator{\conv}{conv}

\makeatletter
\def \uminus@sym{\setbox0=\hbox{$\cup$}\rlap{\hbox 
        to\wd0{\hss\raise0.5ex\hbox{$\scriptscriptstyle{-}$}\hss}}\box0}
    \def \uminus    {\mathrel{\uminus@sym}}
\makeatother

\newcommand{\iprod}[2]{\langle #1,#2\rangle}

\newcommand{\weakto}{\mathrel{\rightharpoonup}}
\makeatletter
\def \weaktostar@sym{\setbox0=\hbox{$\rightharpoonup$}\rlap{\hbox 
        to\wd0{\hss\raise1ex\hbox{$\scriptscriptstyle{*\,}$}\hss}}\box0}
    \def \weaktostar    {\mathrel{\weaktostar@sym}}
\makeatother

%% file: texbase/symbols-ten.tex
\renewcommand{\tilde}{\widetilde}

\renewcommand{\L}{\mathcal{L}}

\renewcommand{\d}{\,d} 

%% file: texbase/symbols-ip.tex
\newcommand{\E}{\mathcal{E}}

%% file: cpaccel-common.tex
\usepackage{mathtools}

\floatstyle{ruled}
\newfloat{Algorithm}{tbp!}{loa}
\crefname{Algorithm}{Algorithm}{Algorithms}

\newcommand{\setto}{\rightrightarrows}

\def\extR{\overline \R}
\def\linear{\mathcal{L}}

\newcommand{\linearLArrow}[1][]{\linear_{\triangleleft\ifx&#1&\else,\,#1\fi}}
\newcommand{\linearLArrowSpecial}[1][]{\linear^{\star}_{\triangleleft\ifx&#1&\else,\,#1\fi}}

\def\opt#1{\widetilde #1}
\def\realopt#1{\widehat #1}
\def\this#1{#1^i}
\def\nexxt#1{#1^{i+1}}
\def\overnext#1{\bar #1^{i+1}}

\def\optu{{\opt{u}}}
\def\optx{{\opt{x}}}

\def\opty{{\opt{y}}}

\def\realoptu{{\realopt{u}}}
\def\realoptv{{\realopt{v}}}
\def\realoptx{{\realopt{x}}}

\def\realopty{{\realopt{y}}}

\def\nextu{\nexxt{u}}
\def\nextx{\nexxt{x}}

\def\nexty{\nexxt{y}}
\def\nextv{\nexxt{v}}

\def\thisu{\this{u}}
\def\thisx{\this{x}}
\def\thisz{\this{z}}
\def\thisy{\this{y}}
\def\thisv{\this{v}}

\def\overnextx{\overnext{x}}

\def\E{\mathbb{E}}
\def\P{\mathbb{P}}

\def\Tau{T}
\def\TauTest{\Phi}
\def\tauTest{\phi}
\def\SigmaTest{\Psi}
\def\sigmaTest{\psi}

\def\TauTilde{{\tilde \Tau}}
\def\SigmaTilde{{\tilde \Sigma}}

\def\TauMonotoneSpace{\mathcal{T}}
\def\SigmaMonotoneSpace{\mathcal{S}}

\def\gap{\mathcal{G}}

\def\BregFn{V}

\def\AltBregFn{V'}

\def\GammaLift#1{\Xi_{#1}}
\def\SAlg{\mathcal{O}}

\def\Space{U}

\newcommand{\Test}{Z}
\newcommand{\Precond}{M}
\newcommand{\Step}{W}
\newcommand{\Random}{\mathcal{R}}

\newcommand{\convex}{\mathrm{cpl}}

\DeclareFontFamily{U}{mathx}{\hyphenchar\font45}
\DeclareFontShape{U}{mathx}{m}{n}{<-> mathx10}{}
\DeclareSymbolFont{mathx}{U}{mathx}{m}{n}
\DeclareMathAccent{\widebar}{0}{mathx}{"73}

\newcommand{\Penalty}{\Delta}

%% file: proxtest_.bbl
 \providecommand{\eprint}[1]{\href{http://arxiv.org/abs/#1}{arXiv:#1}}
  \providecommand{\eprint}[1]{\href{http://arxiv.org/abs/#1}{arXiv:#1}}
  \providecommand{\noopsort}[1]{}